\documentclass[10pt,oneside,reqno]{amsart}
\usepackage{hyperref}
\usepackage{amsmath}
\usepackage{amssymb}
\usepackage{amsxtra}
\usepackage{amsthm}

\allowdisplaybreaks[4]
\usepackage{stmaryrd}
\usepackage{bm}
\usepackage{bbm}
\usepackage{upgreek}
\usepackage{mathrsfs}
\usepackage{thmtools}
\usepackage{mathtools}
\usepackage{geometry}
\usepackage{enumitem}
\usepackage{mathpazo}
\usepackage{domitian}
\usepackage[T1]{fontenc}

\AtBeginDocument{
 \DeclareSymbolFont{AMSb}{U}{msb}{m}{n}
 \DeclareSymbolFontAlphabet{\mathbb}{AMSb}
 }
 
\let\SSec\S

\newcommand\lopen{\mathopen{]}}
\newcommand\ropen{\mathclose{[}}

\newcommand\rmd{\mathrm{d}}
\newcommand\rme{\mathrm{e}}
\newcommand\rmi{\mathrm{i}}

\newcommand\cK{\mathcal{K}}

\newcommand\cO{\mathcal{O}}

\newcommand\cS{\mathcal{S}}

\newcommand\cX{\mathcal{X}}

\renewcommand\AA{\mathbb{A}}

\newcommand\CC{\mathbb{C}}

\newcommand\FF{\mathbb{F}}

\newcommand\QQ{\mathbb{Q}}
\newcommand\RR{\mathbb{R}}

\newcommand\ZZ{\mathbb{Z}}

\newcommand\mf[1]{\mathfrak{#1}}
\newcommand\ms[1]{\mathscr{#1}}

\newcommand\A{\mathsf{A}}
\newcommand\B{\mathsf{B}}

\newcommand\G{\mathsf{G}}
\renewcommand\H{\mathsf{H}}

\newcommand\M{\mathsf{M}}
\newcommand\N{\mathsf{N}}

\renewcommand\S{\mathsf{S}}
\newcommand\T{\mathsf{T}}

\newcommand\Z{\mathsf{Z}}

\newcommand\SL{\mathsf{SL}}
\newcommand\GL{\mathsf{GL}}

\newcommand\triv{\mathbbm{1}}
\newcommand\dpii{2\uppi\rmi}
\newcommand\bs{\backslash}
\newcommand\diff{\partial}

\renewcommand\Re{\mathop{\mathrm{Re}}}

\DeclareMathOperator\Tr{Tr}
\DeclareMathOperator\norm{N}
\DeclareMathOperator\sgn{sgn}

\DeclareMathOperator\res{res}
\DeclareMathOperator\fp{f.p.}

\DeclareMathOperator\ord{ord}

\newcommand\fin{\mathrm{fin}}

\newcommand\el{\mathrm{ell}}

\newcommand\cusp{\mathrm{cusp}}

\newcommand\Kl{\mathrm{Kl}}
\newcommand\Sym{\mathrm{Sym}}
\DeclareMathOperator\vol{vol}
\DeclareMathOperator\orb{orb}
\renewcommand\leq{\leqslant}
\renewcommand\geq{\geqslant}
\let\oldsslash\sslash
\renewcommand\sslash{{\oldsslash}}

\newcommand\legendresymbol[2]{\genfrac{(}{)}{}{}{#1}{#2}}

\makeatletter
\def\subsection{\@startsection{subsection}{2}%
  \z@{3pt\@plus0pt}{-.5em}%
  {\normalfont\bfseries}}
\makeatother

\makeatletter
\def\@seccntformat#1{
  \protect\textup{\protect\@secnumfont
    \ifnum\pdfstrcmp{subsection}{#1}=0 \bfseries\fi
    \csname the#1\endcsname
    \protect\@secnumpunct
  }
}  
\makeatother

\newtheoremstyle{THEOREM}
{2.5pt}
{2pt}
{\itshape}
{}
{\bfseries}
{.}
{.5em}
{\thmname{#1}\thmnumber{ #2}\thmnote{ (#3)}}

\newtheoremstyle{DEFINITION}
{2.5pt}
{2pt}
{}
{}
{\bfseries}
{.}
{.5em}
{\thmname{#1}\thmnumber{ #2}\thmnote{ (#3)}}

\newtheoremstyle{EXERCISE}
{2pt}
{2pt}
{}
{}
{\scshape}
{.}
{.5em}
{\thmname{#1}\thmnumber{ #2}\thmnote{ (#3)}}

\theoremstyle{THEOREM}
\newtheorem{theorem}{Theorem}[section]
\newtheorem{lemma}[theorem]{Lemma}
\newtheorem{proposition}[theorem]{Proposition}
\newtheorem{corollary}[theorem]{Corollary}

\theoremstyle{DEFINITION}
\newtheorem{definition}[theorem]{Definition}
\newtheorem{assumption}[theorem]{Assumption}

\theoremstyle{EXERCISE}
\newtheorem{remark}[theorem]{Remark}

\numberwithin{equation}{section}

\makeatletter
\renewenvironment{proof}[1][\proofname]{\par
  \vspace{-6pt}
  \pushQED{\qed}
  \normalfont \topsep6\p@\@plus6\p@\relax
  \trivlist
  \item[\hskip\labelsep\rmfamily\bfseries
    #1\@addpunct{:}]\ignorespaces
}{
  \popQED\endtrivlist\@endpefalse
  \vspace{-6pt}
}
\makeatother

\makeatletter
\newenvironment{insertproof}[1][\proofname]{
  \par
  \vspace{-6pt}
  \pushQED{\qed}
  \normalfont \topsep6\p@\@plus6\p@\relax
  \trivlist
  \item[\hskip\labelsep\rmfamily\bfseries
    #1\@addpunct{:}]\ignorespaces
}{
  \popQED\endtrivlist\@endpefalse
  \vspace{-6pt}
}
\makeatother

\begin{document}
\setlength \lineskip{3pt}
\setlength \lineskiplimit{3pt}
\setlength \parskip{1pt}
\setlength \partopsep{0pt}

\title{Beyond endoscopy for the symmetric square representation: The simple trace formula case}
\author{Yuhao Cheng}
\address{Qiuzhen College, Tsinghua University, 100084, Beijing, China}
\email{chengyuhaomath@gmail.com}
\dedicatory{Ya tutarsa——Nasreddin Hoca}
\date{28 July, 2026}
\begin{abstract}
At the beginning of this century, Langlands introduced a strategy known as \emph{Beyond Endoscopy} to attack the principle of functoriality. Altu\u{g} studied $\GL_2$ over $\QQ$ in the unramified setting for the standard representation. We consider the case with ramification at $S=\{\infty,q_1,\dots,q_r\}$ with $2\in S$ and derive an asymptotic formula for the symmetric square representation adding some additional conditions on the test function so that the trace formula is simple. The limit is nonzero in general and we may detect the dihedral forms by using such limit form of the trace formula which is similar to Venkatesh's thesis. The proof involves a second Poisson summation, computation of the transformed Kloosterman sum and the corresponding series, and giving an asymptotic formula for the main term by residue analysis and using technical analysis to deal with the error term.
\end{abstract}

\maketitle
 
\tableofcontents

\section{Introduction}
\subsection{Background}
Let $\G$ be a connected reductive algebraic group over the field $\QQ$ of rational numbers and let $\AA$ denote the ad\`ele ring of $\QQ$. Let $\prescript{L}{}\G$ be the $L$-group of $\G$. Let $\pi$ be an automorphic representation of $\G$ over $\QQ$ and $\rho$ be a finite dimensional complex representation of $\prescript{L}{}\G$. For a finite set $S$ of places of $\QQ$ containing the archimedean place 
$\infty$, the partial $L$-function $L^{S}(s,\pi,\rho)$ can be defined (see \cite{getz2024} for example) and expressed as a Dirichlet series
\[
L^{S}(s,\pi,\rho)=\sum_{\substack{n=1\\ \gcd(n,S)=1}}^{+\infty}\frac{a_{\pi,\rho}(n)}{n^s}
\]
for $\Re s$ sufficiently large.

In general, Langlands' \emph{Beyond Endoscopy} proposal \cite{langlands2004} aims to provide a formula for
\begin{equation}\label{eq:beyondendoscopy}
\sum_{\pi}m_{\pi}\ord_{s=1}L^{S}(s,\pi,\rho)\prod_{v\in S}\Tr(\pi_v(f_v))
\end{equation}
for cuspidal representations $\pi$ and suitable functions $f_v$ on $\G(\QQ_v)$ for each $v\in S$, where $m_\pi$ denotes the multiplicity of $\pi$ in $L^2_{\mathrm{cusp}}(\G(\QQ)\bs \G(\AA)^1)$. Using this, we can analyze $\ord_{s=1}L^S(s,\pi,\rho)$ and determine whether $\pi$ arises from an automorphic representation of a smaller group $\H$.

By applying the trace formula, we are able to evaluate
\[
\sum_{\pi}m_\pi a_{\pi,\rho}(p)\prod_{v\in S}\Tr(\pi_v(f_v))=I_{\mathrm{cusp}}(f^{p,\rho}),
\]
where $f^{p,\rho}=\bigotimes_{v\in S}f_v\otimes\bigotimes_{v\notin S}'f_v^{p,\rho}$, with $f_v^{p,\rho}$ spherical for $v\notin S$, and $\Tr(\pi_p(f_p^{p,\rho}))=a_{\pi,\rho}(p)$ and $\Tr(\pi_\ell(f_\ell^{p,\rho}))=1$ if $\ell\notin S\cup\{p\}$.

If $L^S(s,\pi,\rho)$ admits meromorphic continuation on $\Re s>1-\delta$ with neither zeros nor poles on the vertical line $s=1+\rmi t$ except at $s=1$, by Ikehara theorem, we expect the following asymptotic formula:
\[
\lim_{X\to +\infty}\frac{1}{X}\sum_{\substack{p<X\\ p\notin S}}a_{\pi,\rho}(p)\log p=\ord_{s=1}L^S(s,\pi,\rho),
\]
where $p$ runs over all primes less than $X$.
For example, for $L^S(s,\pi,\rho)=\zeta(s)$, the above formula reduces to
\[
\lim_{X\to +\infty}\frac{1}{X}\vartheta(X)=\lim_{X\to +\infty}\frac{1}{X}\sum_{p<X}\log p =1,
\]
which is equivalent to the prime number theorem.

If such asymptotic formula holds, \eqref{eq:beyondendoscopy} can be rewritten as
\[
\lim_{X\to +\infty}\frac{1}{X}\sum_{\pi}m_\pi\prod_{v\in S}\Tr(\pi_v(f_v))\sum_{\substack{p<X\\ p \notin S}}a_{\pi,\rho}(p)\log p .
\] 

For certain functions $f^{p,\rho}$ and certain representations $\pi$ depending on $f^{p,\rho}$, we have
\[
\frac{1}{X}\sum_{\pi}m_\pi\prod_{v\in S}\Tr(\pi_v(f_v))\sum_{\substack{p<X\\ p \notin S}}a_{\pi,\rho}(p)\log p 
=\frac{1}{X}\sum_{\substack{p<X\\ p \notin S}}\log pI_{\mathrm{cusp}}(f^{p,\rho}).
\] 
Therefore, we expect that 
\[
\lim_{X\to +\infty}\frac{1}{X}\sum_{\substack{p<X\\ p \notin S}}\log pI_{\mathrm{cusp}}(f^{p,\rho})=\sum_{\pi}m_\pi \prod_{v\in S}\Tr(\pi_v(f_v))\ord_{s=1}L^S(s,\pi,\rho),
\]
where $\pi$ runs over certain representations of $\G(\AA)$ depending on $f^{p,\rho}$. 

We can consider a more general setting, which was proposed by Sarnak \cite{sarnak2001}. For $\gcd(n,S)=1$, we have
\[
\sum_{\pi}m_\pi a_{\pi,\rho}(n)\prod_{v\in S}\Tr(\pi_v(f_v))=I_{\mathrm{cusp}}(f^{n,\rho}),
\]
where $f^{n,\rho}=\bigotimes_{v\in S}f_v\otimes\bigotimes_{v\notin S}'f_v^{n,\rho}$, with $f_v^{n,\rho}$ spherical for $v\notin S$, and $\Tr(\pi_p(f_p^{n,\rho}))=a_{\pi,\rho}(p^{v_p(n)})$ for all $p\notin S$. By Ikehara theorem, we expect to have
\[
\lim_{X\to +\infty}\frac{1}{X}\sum_{\substack{n<X\\ \gcd(n,S)=1}}a_{\pi,\rho}(n)=\res_{s=1}L^S(s,\pi,\rho),
\]
and we expect that
\[
\lim_{X\to +\infty}\frac{1}{X}\sum_{\substack{n<X\\ \gcd(n,S)=1}}I_{\mathrm{cusp}}(f^{n,\rho})=\sum_{\pi}m_\pi \prod_{v\in S}\Tr(\pi_v(f_v))\res_{s=1}L^S(s,\pi,\rho),
\]
where $\pi$ runs over certain representations of $\G(\AA)$ depending on $f^{n,\rho}$. 

We temporarily assume that $\G=\GL_2$, $\rho=\mathrm{Std}$ is the standard representation.
For any prime number $p$ and $r\in \ZZ_{\geq 0}$, we define
\[
\cX_p^{m}=\{X\in \M_2(\ZZ_p)\,|\, \mathopen{|}\det X\mathclose{|}_p = p^{-m}\}.
\]
For example, if $r=0$, $\cX_p^{r}$ is just $\cK_p=\GL_2(\ZZ_p)$. By Hecke operator theory, we can choose $f^{n}$ to be $\bigotimes_{v\in \mf{S}}f^{n}_v$, where $\mf{S}$ denotes the set of places of $\QQ$, and
\begin{enumerate}[itemsep=0pt,parsep=0pt,topsep=2pt,leftmargin=0pt,labelsep=3pt,itemindent=9pt,label=\textbullet]
  \item If $v=p$, $f^{n}_p=p^{-n_p/2}\triv_{\cX_p^{n_p}}$, where $n_p=v_p(n)$.
  \item If $v=\infty$,  $f^{n}_\infty=f_\infty\in C^\infty(Z_+\bs \G(\RR))$ such that the orbital integrals are compactly supported modulo $Z_+$, and other than this condition they are arbitrary.
\end{enumerate}
In this case, $\pi$ runs over all unramified cuspidal representations. Venkatesh \cite{venkatesh2002,venkatesh2004} established an asymptotic formula for the residue case for $d\leq 2$, using the Petersson-Kuznetsov trace formula.
As a new approach, Altu\u{g} \cite{altug2015,altug2017,altug2020} used the Arthur-Selberg trace formula to prove a variant asymptotic formula for the case $d=1$. More precisely, Altu\u{g} proved that
\[
\sum_{n<X}\Tr(T_m(n))\ll_{m,\varepsilon} X^{31/32+\varepsilon},
\]
where $T_m(n)$ denotes the (normalized) $n^{\mathrm{th}}$ Hecke operator acting on $S_m(\SL_2(\ZZ))$, the space of holomorphic cusp forms of weight $m$ for the modular group $\Gamma=\SL_2(\ZZ)$.

\subsection{Main result in this paper}
In this paper, we will prove an analogous formula for \cite{venkatesh2004} using the simple trace formula on $\GL_2$,
partially answering the question raised by Sarnak in the case over $\QQ$ with ramification and  $n$ such that $\gcd(n,S)=1$ and $d=2$ for a type of function.
We consider $S=\{\infty,q_1,\dots,q_r\}$ for primes $q_1,\dots,q_r$ such that $2\in S$, and a corresponding function $f^{n}=\bigotimes_{v\in \mf{S}}'f^n_v$ for the standard representation such that the local components at places in $S$ are arbitrary. 
Specifically, for $\gcd(n,S)=1$, the function $f^n$ is defined as follows:
\begin{enumerate}[itemsep=0pt,parsep=0pt,topsep=2pt,leftmargin=0pt,labelsep=3pt,itemindent=9pt,label=\textbullet]
  \item If $v=p\notin S$,  $f^{n}_v=p^{-n_p}\triv_{\cX_p^{n_p}}$, where $n_p=v_p(n)$.
  \item If $v=q_i\in S$ and is a prime,  $f^{n}_v=f_{q_i}$ is an arbitrary function in $C_c^\infty(\G(\QQ_{q_i}))$.
  \item If $v=\infty$,  $f^{n}_\infty=f_\infty\in C^\infty(Z_+\bs \G(\RR))$ such that the orbital integral is compactly supported modulo $Z_+$, and other than this condition it is arbitrary.
\end{enumerate}
Note that in this case, $f_v^{n}$ is spherical for $v\notin S$, and $\Tr(\pi_p(f_p^{n}))=a_{\pi}(p^{n_p})$ for all $p\notin S$. 

We want to find analogous functions $f^{n,2}$ for the symmetric square case and wish to give an asymptotic expansion of
\[
\sum_{\substack{n<X\\ \gcd(n,S)=1}}I_{\cusp}(f^{n,2})
\]
so that by the following identity
\[
\lim_{X\to +\infty}\frac{1}{X}\sum_{\substack{n<X\\\gcd(n,S)=1}}I_\cusp(f^{n,2})=\sum_{\pi}m_\pi \prod_{v\in S}\Tr(\pi_v(f_v))\res_{s=1}L^S(s,\pi,\Sym^2),
\]
we can detect the automorphic representations $\pi$ such that $L^S(s,\pi,\Sym^2)$ has a pole at $s=1$. However, we do not use such functions in this paper. Instead, we will consider the following asymptotic formula, which is actually considered in \cite{venkatesh2002} and \cite{venkatesh2004}:
\begin{equation}\label{eq:relationstandardtosym2}
\lim_{X\to +\infty}\frac{1}{X}\sum_{\substack{n<X\\\gcd(n,S)=1}}I_\cusp(f^{n^2})=\sum_{\pi}m_\pi \prod_{v\in S}\Tr(\pi_v(f_v))\res_{s=1}\frac{L^S(s,\pi,\Sym^2)}{L^S(2s,\chi_\pi^2)},
\end{equation} 
(The formula will be proved in \autoref{sec:sym2std}.)
We can also use this formula to detect the automorphic representations $\pi$ such that $L^S(s,\pi,\Sym^2)$ has a pole at $s=1$. This case is easy to deal with since we have many tools and results \cite{cheng2025,cheng2025b,cheng2025c} for the standard representation case.

Let $v$ be a place of $\QQ$. We say an absolutely integrable function $f_v$ on $\G(\QQ_v)$ \emph{supercuspidal} if for any $g,h\in \G(\QQ_v)$, 
\[
\int_{\N(\QQ_v)}f_v(gnh)\rmd n=0,
\]
where $\N$ denotes the unipotent radical of the standard Borel subgroup $\B$ of $\G$.

We will prove the main result \autoref{thm:totalestimatefinal} in this paper under an additional assumption. We only consider the simple trace formula case. Namely,

\begin{assumption}\label{ass:nonarchimedean}
There exist two places $v_1\neq v_2\in S$ such that $f_{v_1}$ is supercuspidal, and $f_{v_2}$ is supercuspidal or supported on the elliptic elements. For simplicity, we assume that $v_1=\infty$ and $v_2=q_1$. The other cases are similar but with different results.
\end{assumption}

\begin{theorem}\label{thm:totalestimatefinal}
Suppose that $\varrho<1/8$ is a bound towards the Ramanujan conjecture. $f^{n}=\bigotimes_{v\in \mf{S}}f^{n}_v$ satisfying \autoref{ass:nonarchimedean}. Then for any $\varepsilon>0$, we have
\[
\sum_{\substack{n<X\\\gcd(n,S)=1}}I_{\mathrm{cusp}}(f^{n^2})=\ms{A}X+\ms{B}X+\ms{C}X+O (X^{\frac43\varrho+\frac56+\varepsilon}),
\]
where $\ms{A}$ is defined in \autoref{thm:contributionxi0} and $\ms{B},\ms{C}$ are defined in \autoref{thm:contributionxinot0}. The implied constant depends only on $f_\infty$, $f_{q_i}$ and $\varepsilon$.
\end{theorem}
\begin{remark}
We can also get the asymptotic formula of the elliptic part without \autoref{ass:nonarchimedean} by dealing with the proof of \autoref{thm:contributionxi0}, especially examining the residues more delicately. However, the formula may be extremely complicated.
\end{remark}
\begin{remark}
The sharpest bound towards the Ramanujan conjecture is $\varrho=7/64<1/8$ now (see \cite{sarnak2005}), so \autoref{thm:totalestimatefinal} is valid with
\[
\sum_{\substack{n<X\\\gcd(n,S)=1}}I_{\mathrm{cusp}}(f^{n^2})=\ms{A}X+\ms{B}X+\ms{C}X+ (X^{\frac{47}{48}+\varepsilon}).
\]
Maybe one can improve the bound such that $\varrho=1/4$ proved in \cite{cheng2025b} can be used.
\end{remark}

By \eqref{eq:relationstandardtosym2} we thus obtain a formula for the average of the residues at $s=1$ of the symmetric square case.

\begin{corollary}
Suppose that $\varrho<1/8$ is a bound towards the Ramanujan conjecture and \autoref{ass:nonarchimedean} holds. Then
\[
\sum_{\pi}m_\pi \prod_{v\in S}\Tr(\pi_v(f_v))\res_{s=1}\frac{L^S(s,\pi,\Sym^2)}{L^S(2s,\chi_\pi^2)}=\ms{A}+\ms{B}+\ms{C},
\]
where $\ms{A}$ is defined in \autoref{thm:contributionxi0} and $\ms{B},\ms{C}$ are defined in \autoref{thm:contributionxinot0}.
\end{corollary}

\subsection{An overview of the proof}
By the simple trace formula we have (see \cite[Chapter 16]{getz2024} for example)
\[
\sum_{\substack{n<X\\\gcd(n,S)=1}}I_{\mathrm{cusp}}(f^{n^2})= \sum_{\substack{n<X\\\gcd(n,S)=1}}I_{\mathrm{ell}}(f^{n^2}),
\]
and by \cite{cheng2025} we have
\[
\sum_{\substack{n<X\\\gcd(n,S)=1}}I_{\mathrm{ell}}(f^{n^2})= \sum_{\substack{n<X\\\gcd(n,S)=1}}\Sigma^{n^2}(\xi)-\sum_{\substack{n<X\\\gcd(n,S)=1}}\Sigma^{n^2}(\square).
\]
The square term vanishes under \autoref{ass:nonarchimedean} (see the proof of \autoref{prop:ramanujan}). Hence it suffices to consider the term after Poisson summation, namely
\[
S(X)=\sum_{\substack{n<X\\\gcd(n,S)=1}}\Sigma^{n^2}(\xi).
\]
We first consider the asymptotic formula of the smooth variant
\[
S_G(X)=\sum_{\substack{n=1\\\gcd(n,S)=1}}^{+\infty}G\legendresymbol{n}{X}\Sigma^{n^2}(\xi).
\]
We split the sum into two parts: the $\xi=0$ part and the $\xi\neq 0$ part.

For the $\xi=0$ part, we will use the analytic continuation of Kloosterman-type series
\[
\sum_{n\in \ZZ_{(S)}^{>0}}\frac{\chi(n)}{n^{u-s}}\sum_{k,f\in \ZZ_{(S)}^{>0}}\frac{\Kl_{k,f}^S(0,\pm n^2q^\nu)}{k^{1+s}f^{1+2s}},
\]
which will be proved in \autoref{lem:kloostermandirichlet}. The proof of this lemma uses Corollary 5.6 of \cite{cheng2025}, which is based on the computation of \cite[Section 5]{altug2015}. Then we use the contour shifting method to estimate the sum $S_G^{\xi=0}(X)$. The main term comes from the residues and the error term comes from the integral over the new contour. The method is similar to that in the proof of \cite[Section 4.1.2]{altug2020}.

Next we consider the $\xi\neq 0$ term. We will do a second Poisson summation with respect to the determinant and obtain 
\begin{align*}
S_G^{\xi\neq 0}(X)&=\frac12\sum_{c\in \{\pm 1\}\times q^{\FF_2^r}}\sum_{k,f\in \ZZ_{(S)}^{>0}}\frac{1}{k^3f^5}\sum_{\substack{\xi,\alpha\in \ZZ^S\\\xi\neq 0}} \prescript{}{c} {\widehat{J}}_{k,f}(X,\xi,\alpha)\prescript{2}{c}{\widehat{\Kl}}_{k,f}^S(\xi,\alpha)\\
&=\prod_{i=1}^{r}(1-q_i^{-1})^{-1}\sum_{\chi}\int_{(\sigma)} \widetilde{G}(s)\sum_{\substack{\xi,\alpha\in \ZZ^S\\\xi\neq 0}}\prescript{2}{c}D_{\xi,\alpha}^S(s,\chi)\prescript{}{c}\Phi_{\xi,\alpha}(s,\chi)X^s\rmd s,
\end{align*}
where the integral transform $_c\widehat{J}_{k,f}(X,\xi,\alpha)$ and the transformed Kloosterman sum $\prescript{2}{c}{\widehat{\Kl}}_{k,f}^S(\xi,\alpha)$ defined in \autoref{sec:secondpoisson},
\[  
\prescript{2}{c}D_{\xi,\alpha}^S(s,\chi)=\sum_{k,f\in \ZZ_{(S)}^{>0}}\frac{\prescript{2}{c}{\widehat{\Kl}}_{k,f}^{S}(\xi,\alpha)\chi(kf^2)}{k^{1+s}f^{1+2s}}.
\]
and the integral $\prescript{}{c}\Phi_{\xi,\alpha}(s,\chi)$ is defined in \autoref{sec:mainanalysis}.

Then we compute $\prescript{2}{c}{\widehat{\Kl}}_{k,f}^S(\xi,\alpha)$ explicitly and analyze the poles and the bounds of the Dirichlet series $\prescript{2}{c}D_{\xi,\alpha}^S(s,\chi)$ and estimate the integral. Finally, we use contour shift method, computing the main term from the residues and the error term from the integral over the new contour. Thus we get an asymptotic formula of $S_G(X)$.

The last thing to do is to compare $S(X)$ and $S_G(X)$, which is similar to the method of \cite{cheng2025c}. Namely, we control their difference due to the bound $\varrho<1/8$ towards the Ramanujan conjecture. Hence we obtain an asymptotic formula for $S(X)$ and finally we get the result by proving \eqref{eq:relationstandardtosym2} in \autoref{sec:sym2std}.

\subsection{Notations}
\begin{enumerate}[itemsep=0pt,parsep=0pt,topsep=0pt,leftmargin=0pt,labelsep=3pt,itemindent=9pt,label=\textbullet]
  \item $\# X$ denotes the number of elements in a set $X$.
  \item For $A\subseteq X$, $\triv_A$ denotes the characteristic function on $X$, defined by $\triv_A(x)=1$ for $x\in A$ and $\triv_A(x)=0$ for $x\notin A$.
  \item $\triv$ also denotes the trivial character or the trivial representation.
  \item For $x\in \RR$, $\lfloor x\rfloor$ denotes the greatest integer that is less than or equal to $x$.
  \item We often use the notation $a\equiv b\,(n)$ to denote $a\equiv b\pmod n$.
  \item For $D\equiv 0,1\pmod 4$, $\legendresymbol{D}{\cdot}$ denotes the Kronecker symbol.
  \item If $R$ is a ring (which we \emph{always} assume to be commutative with $1$), $R^\times$ denotes its group of units.  
  \item $\mf{S}$ denotes the set of all places of $\QQ$.
  \item For $S=\{\infty, q_1,\dots,q_r\}$, $\ZZ^S$ denotes the ring of $S$-integers in $\QQ$. That is
    \[
    \ZZ^S=\{\alpha\in \QQ\ |\ v_p(\alpha)\geq 0\ \text{for all}\ p\notin S\}.
    \]
    Additionally, we define
    \[
    \QQ_S=\prod_{v\in S}\QQ_v=\RR\times \QQ_{q_1}\times\dots\times\QQ_{q_r}\quad\text{and}\quad\QQ_{S_\fin}=\prod_{v\in S_\fin}\QQ_v=\QQ_{q_1}\times\dots\times\QQ_{q_r}
    \]
  and
    \[
    \ZZ_{S_\fin}=\prod_{v\in S_\fin}\ZZ_v=\ZZ_{q_1}\times\dots\times\ZZ_{q_r}.
    \]
  \item For $n\in \ZZ$, we write $\gcd(n,S)=1$, or $n\in \ZZ_{(S)}$, if 
  $p\nmid n$ for all $p\in S$. 
  We write $n\in \ZZ_{(S)}^{>0}$ if additionally $n>0$.
  \item Let $p$ be a prime. For $a\in\QQ$, we define $a_{(p)}$ to be the $p$-part of $a$, that is, 
\[
a_{(p)}=p^{v_p(a)},
\]
and we define $a^{(p)}=a/a_{(p)}$.
Moreover, we define 
  \[
    a_{(q)}=\prod_{i=1}^{n}a^{v_{q_i}(a)}\quad\text{and}\quad a^{(q)}=\frac{a}{a_{(q)}}=\prod_{p\notin S}a^{v_{p}(a)}.
  \]
  \item Let $p$ be a prime and $k\in \ZZ_{\geq 0}$. $p^k\parallel a$ means that $v_p(a)=k$.
  \item For $n\in \ZZ_{>0}$, ${\bm d}(n)$ denotes the number of divisors of $n$, ${\bm \sigma}(n)$ denotes the sum of divisors of $n$, ${\bm \phi}(n)$ denotes the Euler totient function. $\bm{\mu}(n)$ denotes the M\"obius function, defined by
  \[
  \bm{\mu}(n)=\begin{cases}
    1, & \text{if $n=1$}, \\
    (-1)^m, & \text{if $n$ is a product of $m$ distinct primes}, \\
    0, & \text{otherwise}.
  \end{cases}
  \]
  \item $\Gamma(s)$ denotes the gamma function, defined by
  \[
  \Gamma(s)=\int_{0}^{+\infty}\rme^{-x}x^s\frac{\rmd x}{x},
  \]
  for $\Re s>0$, analytically continued to $\CC$. $\zeta(s)$ denotes the Riemann zeta function, defined by
  \[
    \zeta(s)=\sum_{n=1}^{+\infty}\frac{1}{n^s}=\prod_{p}\frac{1}{1-p^{-s}}
  \]
  for $\Re s>1$, analytically continued to $\CC$.
  \item For any Dirichlet character $\chi$, we define
  \[
  L(s,\chi)=\sum_{n=1}^{+\infty}\frac{\chi(n)}{n^s}=\prod_{p}\frac{1}{1-\chi(p)p^{-s}}.
  \]
  for $\Re s>1$, analytically continued to $\CC$. Moreover, we define
  \[
  L^S(s,\chi)=\sum_{\substack{n=1\\\gcd(n,S)=1}}^{+\infty}\frac{\chi(n)}{n^s}=\prod_{p\notin S}\frac{1}{1-\chi(p)p^{-s}}.
  \]
  for $\Re s>1$, analytically continued to $\CC$. Finally, we set
  \[
  \zeta^S(s)=L^S(s,\triv)=\sum_{\substack{n=1\\\gcd(n,S)=1}}^{+\infty}\frac{1}{n^s}=\prod_{p\notin S}\frac{1}{1-p^{-s}}=\prod_{i=1}^{r}(1-q_i^{-s})\zeta(s).
  \]
  \item $(\sigma)$ denotes the vertical contour from $\sigma-\rmi\infty$ to $\sigma+\rmi\infty$.
  \item Let $f$ be a meromorphic function near $z=z_0$ and suppose that 
  \[
  f(z)=\sum_{n\in \ZZ}a_n(z-z_0)^n
  \]
  is its Laurent expansion near $z_0$, we denote $\fp_{z=z_0}f(z)=a_0$.
  \item We define $\cS$ be the set of smooth functions $\Phi$ on $\lopen 0,+\infty\ropen$ such that for any $k\in \ZZ_{\geq 0}$, $\Phi^{(k)}(x)$ are of rapid decay as $x\to +\infty$.
  \item For $G\in C_c^\infty(\RR)$, $p\in [1,+\infty]$ and $k\in \ZZ_{\geq 0}$, $\|G\|_p$ denotes the $L^p$-norm of $G$, and $\|G\|_{k,p}$ denotes the $W^{k,p}$-Sobolev norm of $G$, namely
    \[  
    \|G\|_{k,p}=\sum_{j=0}^{k}\|G^{(j)}\|_{p}.
    \]
  \item We define $\rme(x)=\rme_\infty(x)=\rme^{\dpii x}$. For a prime $p$ and $x\in \QQ_p$, we define $\langle x\rangle_p$ to be the "fractional part" of $x$. Namely, if $x=\sum_{n\geq -N}a_np^n\in \QQ_p$, then 
      \[
      \langle x\rangle_p=\sum_{n=-N}^{-1}a_np^n\in \QQ.
      \]
      We then define $\rme_p(x)=\rme(-\langle x\rangle_p)$ for $x\in \QQ_p$. \emph{Note the minus sign}.
  \item We use $f(x)=O(g(x))$ or $f(x)\ll g(x)$ to denote that there exists a constant $C$ such that $|f(x)|\leq C|g(x)|$ for all $x$ in a specified set. If the constant depends on other variables, they will be subscripted under $O$ or $\ll$. 
  \item The notation $f(x)\asymp g(x)$ indicates that $f(x)\ll g(x)$ and $g(x)\ll f(x)$. If the constant depends on other variables, they will be subscripted under $\asymp$.
  \item We use $\square$ as the end of the proof and use $\blacksquare$ as the end of the proof of a lemma that is inserted in the middle of another proof.
\end{enumerate}

\section{Preliminaries}\label{sec:preliminaries}

\subsection{The elliptic part of the trace formula and measure normalizations}
Let $\mf{S}$ be the set of places of $\QQ$ and let $\AA$ be the ad\`ele ring of $\QQ$. We want to compute the elliptic part of the trace formula for certain functions of the form $f=\bigotimes_{v\in \mf{S}}' f_v$ on $\G(\AA)$. The elliptic part of the trace formula is
\begin{equation}\label{eq:elliptictrace}
I_\el(f)=\sum_{\gamma\in \G(\QQ)^\#_\el}\vol(\gamma)\prod_{v\in \mathfrak{S}}\orb(f_v;\gamma).
\end{equation}
where $ \G(\QQ)^\#_\el$ denotes the set of elliptic conjugacy classes in $\G(\QQ)$, and
\[
\orb(f_v;\gamma)=\int_{\G_\gamma(\QQ_v)\bs \G(\QQ_v)} f_v(g^{-1}\gamma g)\rmd g,
\qquad
\vol(\gamma)=\int_{Z_+\G_\gamma(\QQ)\bs \G_\gamma(\AA)} \rmd g.
\]
with $Z_+$ the connected component of the identity matrix in the center $\Z(\RR)$ of $\G(\RR)$. The measures of $\G_\gamma(\QQ_v)$ and $\G(\QQ_v)$ are normalized as follows. For $\G(\QQ_v)$, 
\begin{enumerate}[itemsep=0pt,parsep=0pt,topsep=0pt,leftmargin=0pt,labelsep=3pt,itemindent=9pt,label=\textbullet]
  \item If $v=p$ is a prime, we normalize the Haar measure on $\G(\QQ_p)$ such that the volume of $\G(\ZZ_p)$ is $1$.
  \item If $v=\infty$, we choose an arbitrary Haar measure on $\G(\RR)$ (which does not affect the arguments below).
\end{enumerate}

For $\G_\gamma(\QQ_v)\cong E_v^\times$ (by Proposition 2.1 of \cite{cheng2025}),
\begin{enumerate}[itemsep=0pt,parsep=0pt,topsep=0pt,leftmargin=0pt,labelsep=3pt,itemindent=9pt,label=\textbullet]
  \item If $v=p$ is a prime, we know that $E_p$ is either $\QQ_p^2$ or a quadratic extension of $\QQ_p$. In the first case, we let $\G_\gamma(\ZZ_p)=\ZZ_p^\times \times \ZZ_p^\times$. In the second case, we let $\G_\gamma(\ZZ_p)=\cO_{E_p}^\times$
  In each case, we normalize the Haar measure on $\G_\gamma(\QQ_p)=E_p^\times$ such that the volume of $\G_\gamma(\ZZ_p)$ is $1$.
  \item If $v=\infty$, we know that $E_v$ is either $\RR^2$ or $\CC$. If $E_v=\RR^2$, we define the measure on $E_v^\times=\RR^\times \times \RR^\times$ as $\rmd x\rmd y/|xy|$, where $(x,y)$ is the coordinate of $\RR^2$. The action of $Z_+$ on $E_v$ is $a\cdot(x,y)=(\sqrt{a}x,\sqrt{a}y)$, and we define the measure on $Z_+\bs E_v$ to be the measure of the quotient of $E_v$ by the measure $\rmd a/a$ on $Z_+$. If $E_v=\CC$, we choose the  measure on $E_v^\times=\CC^\times$ to be $2\rmd r\rmd \theta/r$, where we use the polar coordinate $(r,\theta)$ on $\CC^\times$. The action of $Z_+$ on $E_v$ is $a\cdot z=\sqrt{a}z$, and we define the measure on $Z_+\bs E_v$ to be the measure of the quotient of $E_v$ by the measure $\rmd a/a$ on $Z_+$.
\end{enumerate} 
These measure normalizations are taken from \cite{langlands2004}.

\subsection{The modified norm}
For any prime $p$, the modified norm $|\cdot|_p'$ is as follows. For $y\in \QQ_p$, if $p\neq 2$, we have
\[
|y|_p'=p^{-2\lfloor v_p(y)/2\rfloor},
\]
that is, $|y|_p'=|y|_p$ if $v_p(y)$ is even and $|y|_p'=p|y|_p$ if $v_p(y)$ is odd. If $p=2$, then for odd $v_p(y)$ we have $|y|_p'=p^{-v_p(y)+3}=p^3|y|_p$, while for even $v_p(y)$, if $y=p^{v_p(y)}y_0$, then
\[
|y|_p'=\begin{cases}
  p^{-v_p(y)}=|y|_p, & \text{if $y_0\equiv 1\pmod 4$}, \\
  p^{-v_p(y)+2}=p^2|y|_p, & \text{if $y_0\equiv 3\pmod 4$}.
\end{cases}
\]

For any regular element $\gamma\in\G(\QQ_p)$, we denote $T_\gamma=\Tr\gamma$ and $N_\gamma=\det\gamma$. $k_\gamma$ is defined such that $p^{k_\gamma}=|T_\gamma^2-4N_\gamma|_p'^{-1/2}$. This coincides with the original definition of \cite{cheng2025} (see Proposition 2.6 of loc. cit.).

\subsection{Singularities of the orbital integrals}\label{subsec:singularities}
We define
\[
\omega_\infty(x)=\begin{cases}
             0, & x>0, \\
             1, & x<0
           \end{cases}
\]
for $x\in \RR$ with $x\neq 0$ and
\[
\omega_p(y)=\legendresymbol{y|y|_{p}'}{p}
\]
for $p\in S$ and $y\in \QQ_p$ with $y\neq 0$. When $p=q_i$, we often write $\omega_p=\omega_i$.
For $\iota\in \{0,1\}$ we define
\[
X_{\iota}=\{x\in \RR\ |\ \omega_\infty(x)=\iota\}
\]
and for $\epsilon_i\in \{0,\pm 1\}$, we define
\[
Y_{\epsilon_i}=\{y_i\in \QQ_{q_i}\ |\ \omega_i(y_i)=\epsilon_i\}. 
\]
The notation here is different from \cite{cheng2025b}. $x$ and $y$ here represent \emph{discriminants} but not traces in this paper. In loc. cit., the determinant $\pm nq^\nu$ is almost fixed but we will vary it in this paper. So it is better to use discriminant as the domain of $\omega_\infty$ and $\omega_p$. Also, we set $\omega_p(T,N)=\omega_p(T^2-4N)$ so that $\omega_p(y)$ defined in \cite{cheng2025b} is $\omega_p(y,\pm nq^\nu)$.

\subsubsection{Nonarchimedean case}
For any prime $p\in S$, we define
\[
\theta_{p}(\gamma)=\frac{1}{\mathopen{|}\det\gamma\mathclose{|}_p^{1/2}}\left(1-\frac{\chi(p)}{p}\right)^{-1}p^{-{k_\gamma}}\orb(f_{p};\gamma),
\]
where $\chi(p)=\omega_p(\Tr\gamma,\det\gamma)$. By Corollary 2.11 of \cite{cheng2025}, the local behavior of $\theta_{p}$ at $z=aI$ is
\begin{equation}\label{eq:shalikalocal}
\theta_{p}(\gamma)=\lambda_1\left(1-\frac{\chi(p)}{p}\right)^{-1}p^{-{k_\gamma}} \frac{1-\chi(p)}{1-p}+\lambda_2.
\end{equation}

Clearly $\theta_{p}(\gamma)$ is invariant under conjugation. Thus $\theta_{p}(\gamma)$ can be parametrized by $T=\Tr\gamma$ and $N=\det\gamma$, i.e., $\theta_{p}(\gamma)=\theta_p(T,N)$. 
Since $\theta_p(\gamma)$ is smooth away from the center, $\theta_p(T,N)$ is smooth except at $T^2=4N$. Moreover, for any $\epsilon\in \{0,\pm 1\}$, by \eqref{eq:shalikalocal} we have
\[
\theta_p(T,N)=\theta_{p,1}(T,N)+\theta_{p,0}(T,N),
\]
where $\theta_{p,\tau}(T,N)={|T^2-4N|_p'}^{\tau/2}\psi_\tau(T,N)$ for $\tau\in \{0,1\}$, where $\psi_\tau(T,N)$ is a smooth function on $\{(T,N)\in \QQ_p^2\,|\,\omega_p(T^2-4N)=\epsilon\}$ and up to boundary, with bounded support, for any $\epsilon\in \{0,\pm 1\}$.

For the invariance with respect to center we have the following lemma:
\begin{lemma}\label{lem:orbitalintegralsmooth}
Let $p\in S$. Then there exists $L\in \ZZ_{>0}$ such that for any $c\in 1+p^{L}\ZZ_p$ and any $T\in \QQ_p$ and $N\in \QQ_p^\times$, we have $\theta_p(cT,c^2N)=\theta_p(T,N)$.
\end{lemma}
\begin{proof}
Since $f_p$ is smooth and compactly supported, there exists $L\geq 2$ such that $f_p(bg)=f_p(g)$ for any $c\in 1+p^{L}\ZZ_p$ and $g\in \G(\QQ_p)$. Hence $\orb(f_p;c\gamma)=\orb(f_p;\gamma)$ for any $c\in 1+p^{L}\ZZ_p$ and regular elements $\gamma$. Clearly $c\gamma$ is split (resp. inert, ramified) if and only if $\gamma$ is. Hence $\theta_p(c\gamma)=\theta_p(\gamma)$ for any $c\in 1+p^{L}\ZZ_p$ and regular elements $\gamma$. Thus for any $T\in \QQ_p$ and $N\in \QQ_p^\times$ with $T^2-4N\neq 0$, we have
\[
\theta_p(cT,c^2N)=\theta_p(T,N).
\]
By continuity, the above equation also holds for $T^2-4N=0$.
\end{proof}

\subsubsection{Archimedean case}
Recall the following theorem  (see Theorem 2.12 in \cite{cheng2025}):
\begin{theorem}\label{thm:archimedeanintegral}
For any $f_\infty\in C^\infty(\G(\RR))$, any maximal torus $\T(\RR)$ in $\G(\RR)$ and any $z$ in the center of $\G(\RR)$, there exists a neighborhood $N$ in $\T(\RR)$ of $z$ and smooth functions $g_1,g_2\in C^\infty(N)$ (depending on $f_\infty$ and $z$) such that
\begin{equation}\label{eq:archimedeanintegral}
\orb(f_\infty;\gamma)=g_1(\gamma)+\frac{|\gamma_1\gamma_2|^{1/2}}{|\gamma_1-\gamma_2|}g_2(\gamma)
\end{equation}
for any $\gamma\in \T(\RR)$, where $\gamma_1$ and $\gamma_2$ are the eigenvalues of $\gamma$.  Moreover, $g_1$ and $g_2$ can be extended smoothly to all split and elliptic elements, remaining invariant under conjugation, with $g_1(\gamma)=0$ if $\T(\RR)$ is split, and $g_2$ can further be extended smoothly to the center. If $f_\infty$ is $Z_+$-invariant, then $g_1$ and $g_2$ are also $Z_+$-invariant.
\end{theorem}

We define
\[
\theta_\infty(\gamma)=\frac{|\gamma_1-\gamma_2|}{|\gamma_1\gamma_2|^{1/2}}\orb(f_\infty;\gamma)= \frac{|\gamma_1-\gamma_2|}{|\gamma_1\gamma_2|^{1/2}}g_1(\gamma)+g_2(\gamma).
\]
Since $g_1,g_2$ and $\theta_\infty$ are invariant under conjugation, we can parametrize them by $T_\gamma$ and $N_\gamma$ as in the nonarchimedean case.

Since $T_{z\gamma}=aT_\gamma$ and  $N_{z\gamma}=a^2N_\gamma$ for $z=aI$ with $a>0$, we have $g_i(T_\gamma,N_\gamma)=g_i(aT_\gamma,a^2N_\gamma)$ and $\theta_\infty(T_\gamma,N_\gamma)=\theta_\infty(aT_\gamma,a^2N_\gamma)$ for $i=1,2$ and any $a>0$. We also set 
\[
\theta_{\infty,0}(T,N)=g_2(T,N)\quad\text{and}\quad\theta_{\infty,1}(T,N)=\left|\frac{T^2-4N}{N}\right| ^{\frac{1}{2}}g_1(T,N).
\]
Hence $\theta_\infty(T,N)=\theta_{\infty,1}(T,N)+\theta_{\infty,0}(T,N)$. Moreover, $\theta_{\infty,\sigma}(T,N)=|T^2-4N|^{\sigma/2}\varphi_\sigma(T,N)$ for $\sigma\in \{0,1\}$, where $\varphi_\sigma(T,N)$ is a smooth function on $\{(T,N)\in \RR^2\,|\,\omega_\infty(T^2-4N)=\iota\}$ up to boundary, with bounded support, for any $\iota\in \{0,1\}$.

Also, we set $\theta_\infty^\pm(x)=\theta_\infty(x,\pm 1/4)$, which coincides with the notation in \cite{cheng2025} and \cite{cheng2025b}.

\subsection{Global notations} For $\nu\in \ZZ^r$, we usually denote $\nu_i$ by its $i^{\mathrm{th}}$ component. We define
\[
q^\nu=q_1^{\nu_1}\dots q_r^{\nu_r}.
\]

For any $y=(y_1,\dots,y_n)\in \QQ_{S_\fin}=\QQ_{q_1}\times\dots\times\QQ_{q_r}$ and $N\in \QQ$, we define
\[
\theta_{q}(y,N)=\prod_{i=1}^{r}\theta_{q_i}(y_i,N),\qquad|y|_q'=\prod_{i=1}^{r}|y_i|_{q_i}',\qquad \rme_q(y)=\prod_{i=1}^{r}\rme_{q_i}(y_i).
\]
Moreover, for $\tau=(\tau_1,\dots,\tau_n)\in \{0,1\}^n$, we define
\[
\theta_{q,\tau}(y,N)=\prod_{i=1}^{r}\theta_{q_i,\tau_i}(y_i,N).
\]
We usually embed $\QQ$ in $\QQ_S$ or $\QQ_{S_\fin}$ diagonally. 

\subsection{The partial Zagier $L$-function and the approximate functional equation}\label{subsec:deffv}
 For $\delta\in \ZZ^S$ that is not a square, we define the \emph{partial Zagier $L$-function} as
\[
L^{S}(s,\delta)=\sum_{f^2\mid \delta}\frac{1}{f^{2s-1}}L^{S}\left(s,\legendresymbol{\delta/f^2}{\cdot}\right),
\]
where the sum of $f$ is over all $\ZZ_{\geq 0}$ such that $f\in \ZZ_{(S)}$ and $\delta/f^2\in \ZZ^S$, and
\[
L^{S}\left(s,\legendresymbol{\delta/f^2}{\cdot}\right)=\sum_{k\in \ZZ_{(S)}^{>0}}\frac{1}{k^s}\legendresymbol{\delta/f^2}{k}=\prod_{p\notin S}\left(1-\legendresymbol{\delta/f^2}{p}p^{-s}\right)^{-1}
\]
for $s$ such that the series converges absolutely, extended to $\CC$ by analytic continuation. 

Let
\[
F(x)=\frac{1}{2K_0(2)}\int_{x}^{+\infty}\rme^{-t-1/t}\frac{\rmd t}{t},
\]
where $K_s(y)$ is the modified Bessel function of the second kind. The Mellin transform of $F$ is
\begin{equation}\label{eq:mellinf}
\widetilde{F}(s)=\frac{1}{s}\frac{K_s(2)}{K_0(2)}.
\end{equation}
From this one can show that (see Proposition 3.7 in \cite{cheng2025} for example) $\widetilde{F}$ is an odd function, and for $s=\sigma+\rmi t\in \CC$ such that $\sigma_1\leq \sigma\leq \sigma_2$, we have
\begin{equation}\label{eq:frapiddecay}
\widetilde{F}(s)\ll_{\sigma_1,\sigma_2} |s|^{|\sigma|-1}\rme^{-\uppi|t|/2}.
\end{equation}

\begin{theorem}[\cite{cheng2025}, Corollary 3.10]
For $\delta\in \ZZ^S$ that is not a square, and any $A>0$, we have
\begin{equation}\label{eq:approximatefeat1}
L^{S}(1,\delta)=\sum_{f^2\mid \delta}\sum_{k\in \ZZ_{(S)}^{>0}}\frac{1}{kf}\legendresymbol{\delta/f^2}{k}\left(F\legendresymbol{kf^2}{A}+\frac{kf^2}{\sqrt{|\delta|_{\infty,q}'}}V_{\iota,\epsilon}\legendresymbol{kf^2A}{|\delta|_{\infty,q}'}\right),
\end{equation}
where 
\begin{equation}\label{eq:defv}
V_{\iota,\epsilon}(y)=V_{\iota,\epsilon,1}(y)=\frac{\uppi^{1/2}}{2\uppi \rmi}\int_{(1)}\widetilde{F}(s)\prod_{i=1}^{r}\frac{1-\epsilon_i q_i^{s-1}}{1-\epsilon_i q_i^{-s}}\frac{\Gamma((\iota+s)/2)}{\Gamma((\iota+1-s)/2)}(\uppi y)^{-s}\rmd s.
\end{equation}
\end{theorem}

They following propositions are proved in \cite{cheng2025b}:
\begin{proposition}\label{prop:propertyf}
We have $F\in \cS$. Moreover, $\widetilde{F}(s)$ has a meromorphic continuation to the whole complex plane, holomorphic except for a simple pole at $s=0$, and is of rapid decay for $t$ when $\sigma$ is fixed.
\end{proposition}

\begin{proposition}\label{prop:propertyh}
For any $\iota$ and $\epsilon$, we have $V_{\iota,\epsilon}\in \cS$. Moreover, $\widetilde{V_{\iota,\epsilon}}(s)$ is holomorphic on $\Re s>0$, and is of rapid decay for $t$ when $\sigma>0$ is fixed.
\end{proposition}

\subsection{The generalized Kloosterman sum} The \emph{partial generalized Kloosterman sum} is defined by
\[
\Kl_{k,f}^S(\xi,m)=
  \sum_{\substack{a \bmod kf^2\\ a^2-4m\equiv 0\,(f^2)}}\legendresymbol{(a^2-4m)/f^2}{k}\rme\legendresymbol{a\xi}{kf^2}\rme_{q}\legendresymbol{a\xi}{kf^2},
\]
for $f,k\in \ZZ_{(S)}$ and $\xi,m\in \ZZ^S$. 
For any prime $p\notin S$ and $\xi,m\in \ZZ_p$, we define the \emph{local generalized Kloosterman sum} to be
\begin{equation}\label{eq:localgeneralizedkloosterman}
\Kl_{p^u,p^v}^{(p)}(\xi,m)=
\sum_{\substack{a \bmod p^{u+2v}\\ a^2-4m\equiv 0\,(p^{2v})}}\legendresymbol{(a^2-4m)/p^{2v}}{p^u}\rme_p\legendresymbol{-a\xi}{p^{u+2v}}.
\end{equation}

As shown in \cite[Section 2.6]{cheng2025b}, for $\xi\in \ZZ^S$, we have
\begin{equation}\label{eq:localgeneralizedkloostermaneq}
\Kl_{p^u,p^v}^{(p)}(\xi,m)=
\sum_{\substack{a \bmod p^{u+2v}\\ a^2-4m\equiv 0\,(p^{2v})}}\legendresymbol{(a^2-4m)/p^{2v}}{p^u} \rme\legendresymbol{a\xi}{p^{u+2v}}\rme_{q}\legendresymbol{a\xi}{p^{u+2v}}.
\end{equation}

\subsection{Norm with respect to the Mellin transform} Fix $\delta>0$. We consider functions $G$ in $C_c^\infty(\lopen \delta^{-1},\delta\ropen)$. Then we have the following results proved in \cite{cheng2025c}:

\begin{lemma}\label{lem:grapiddeacy}
Let $\widetilde{G}(u)$ be the Mellin transform of $G$, namely
\[
\widetilde{G}(u)=\int_{0}^{+\infty}G(x)x^{u}\frac{\rmd x}{x}.
\]
Then $\widetilde{G}(u)$ is well defined on the whole complex plane. Moreover $G$ has rapid decay vertically. More precisely, for any $\sigma\in \RR$ and $A\geq 0$,
\[
\widetilde{G}(\sigma+\rmi t)\ll_{\sigma,A,\delta}\|G\|_{\lceil A\rceil,1}(1+|t|)^{-A}.
\]
\end{lemma}

For $G\in C_c^\infty(\lopen\delta^{-1},\delta\ropen)$ and $K,\sigma\in \RR$, we define the norm
\[
\|G\|_{M^K_\sigma}=\int_{(\sigma)}|\widetilde{G}(u)|(1+|u|)^{K}\rmd |u|.
\]
\begin{corollary}\label{cor:mellinnorm}
For any $G\in C_c^\infty(\lopen\delta^{-1},\delta\ropen)$ and $K\geq -1$, we have
\[
\|G\|_{M^K_\sigma}\ll \|G\|_{\lfloor K\rfloor +2,1},
\]
and for $K<-1$ we have
\[
\|G\|_{M^K_\sigma}\ll \|G\|_{1},
\]
where the implied constants depend only on $K,\sigma$ and $\delta$.
\end{corollary}

\section{A second Poisson summation}\label{sec:secondpoisson}
In \cite[Section 4]{cheng2025} with $\vartheta=1/2$, we computed that
\[
I_{\el}(f^n)+\Sigma^n(\square)=\Sigma^n(\xi)=\sum_{\pm}\sum_{\nu\in \ZZ^r}\sum_{k,f\in \ZZ_{(S)}^{>0}}\frac{1}{k^2f^3}\sum_{\xi\in \ZZ^S}\Kl_{k,f}^S(\xi, \pm nq^\nu)I_{k,f}(\xi, \pm nq^\nu),
\]
where (cf. \cite[Section 3]{cheng2025c})
\begin{align*}
&I_{k,f}(\xi,m)=2\int_{x\in\RR}\int_{y\in\QQ_{S_\fin}}\theta_\infty(x,m)\theta_{q}(y, m)\left[F\legendresymbol{kf^2}{|x^2-4m|_\infty^{1/2}|y^2-4m|_q'^{1/2}}\right.\\
+&\left.\frac{kf^2}{\sqrt{|x^2-4m|_\infty|y^2-4m|_q'}} V\legendresymbol{kf^2}{|x^2-4m|_\infty^{1/2} |y^2- 4m|_q'^{1/2}}\right]\rme\legendresymbol{-x\xi }{kf^2}\rme_{q}\legendresymbol{-y\xi}{kf^2}\rmd x\rmd y.
\end{align*}

Our goal in the following sections is to an asymptotic formula of
\[
S(X):=\sum_{\substack{n<X \\ \gcd(n,S)=1}}\Sigma^{n^2}(\xi).
\]

Let $K=\lfloor\log_2 X\rfloor$. Then we have
\begin{align*}
  S(X) & =\sum_{j=1}^{K}\sum_{\substack{2^{j-1}\leq n<2^{j}\\n\in \ZZ_{(S)}^{>0}}}\sum_{\pm}\sum_{\nu\in \ZZ^r}\sum_{k,f\in \ZZ_{(S)}^{>0}}\sum_{\xi\in \ZZ^S}\frac{1}{k^2f^3}\Kl_{k,f}^S(\xi,\pm n^2q^\nu)I_{k,f}(\xi,\pm n^2q^\nu)\\
  &+\sum_{\substack{2^{K}\leq n<X\\n\in \ZZ_{(S)}^{>0}}}\sum_{\pm}\sum_{\nu\in \ZZ^r}\sum_{k,f\in \ZZ_{(S)}^{>0}}\sum_{\xi\in \ZZ^S}\frac{1}{k^2f^3}\Kl_{k,f}^S(\xi,\pm n^2q^\nu)I_{k,f}(\xi,\pm n^2q^\nu).
\end{align*}
Now we define
\[
S(X/2,X)=\sum_{\substack{X/2\leq n<X\\n\in \ZZ_{(S)}^{>0}}}\sum_{\pm}\sum_{\nu\in \ZZ^r}\sum_{k,f\in \ZZ_{(S)}^{>0}}\sum_{\xi\in \ZZ^S}\frac{1}{k^2f^3}\Kl_{k,f}^S(\xi,\pm n^2q^\nu)I_{k,f}(\xi,\pm n^2q^\nu).
\]

Observe that $\ZZ^S-\{0\}=\{\pm 1\}\times \ZZ_{(S)}^{>0}\times q^{\ZZ^r}$ Hence the quotient
\[
(\{\pm 1\}\times (\ZZ_{(S)}^{>0})^2\times q^{\ZZ^r})/(\ZZ^S-\{0\})^2
\]
is $\{\pm 1\}\times q^{\FF_2^r}$. Moreover, the map $\ZZ^S-\{0\}\to(\ZZ^S-\{0\})^2, a\mapsto a^2$ is a $2:1$ map. Hence
\[
S(X/2,X)=\frac12\sum_{c\in \{\pm 1\}\times q^{\FF_2^r}}\sum_{\substack{X/2\leq |n^{(q)}|<X\\n\in \ZZ^S}}\sum_{k,f\in \ZZ_{(S)}^{>0}}\sum_{\xi\in \ZZ^S}\frac{1}{k^2f^3}\Kl_{k,f}^S(\xi,cn^2)I_{k,f}(\xi,cn^2).
\]
Similarly, the last term can be written as
\[
S(2^K,X)=\frac12\sum_{c\in \{\pm 1\}\times q^{\FF_2^r}}\sum_{\substack{2^{K}\leq |n^{(q)}|<X\\n\in \ZZ^S}}\sum_{k,f\in \ZZ_{(S)}^{>0}}\sum_{\xi\in \ZZ^S}\frac{1}{k^2f^3}\Kl_{k,f}^S(\xi,cn^2)I_{k,f}(\xi,cn^2).
\]

We will approximate the following two sums by using 
\[
S_G(X)=\sum_{\substack{n=1\\\gcd(n,S)=1}}^{+\infty}G\legendresymbol{n}{X}\Sigma^{n^2}(\xi)= \frac12 \sum_{c\in \{\pm 1\}\times q^{\FF_2^r}}\sum_{n\in \ZZ^S}G\legendresymbol{|n^{(q)}|}{X}\sum_{k,f\in \ZZ_{(S)}^{>0}}\sum_{\xi\in \ZZ^S}\frac{1}{k^2f^3}\Kl_{k,f}^S(\xi,cn^2)I_{k,f}(\xi,cn^2)
\]
that is roughly $\triv_{[1/2,1]}$ in the first case and is roughly $\triv_{[2^K/X,1]}$ in the second case. 

The main task of the following sections is to estimate $S_G(X)$. (See \autoref{thm:contributionxig} for the final result.) We assume that $G\in C_c^\infty(\lopen 1/4,5/4\ropen)$ in the following.

\subsection{A second Poisson summation}\label{subsec:secondpoisson}
By the definition of the Kloosterman sum, we have
\[
\Kl_{k,f}^S(\xi,cn^2)=
  \sum_{\substack{a \bmod kf^2\\ a^2-4cn^2\equiv 0\,(f^2)}}\legendresymbol{(a^2-4cn^2)/f^2}{k}\rme\legendresymbol{a\xi}{kf^2}\rme_{q}\legendresymbol{a\xi}{kf^2}.
\]
Hence $\Kl_{k,f}^S(\xi,cn^2)$ is $kf^2$-periodic with respect to the variable $n$. Also, by the definition of $I_{k,f}(\xi,n)$, we have
\begin{align*}
&I_{k,f}(\xi,cn^2)=2\int_{x\in\RR}\int_{y\in\QQ_{S_\fin}}\theta_\infty(x,cn^2)\theta_{q}(y, cn^2)\left[F\legendresymbol{kf^2}{|x^2-4cn^2|_\infty^{1/2}|y^2-4cn^2|_q'^{1/2}}\right.\\
+&\left.\frac{kf^2}{\sqrt{|x^2-4cn^2|_\infty|y^2- 4cn^2|_q'}}V\legendresymbol{kf^2}{|x^2-4cn^2|_\infty^{1/2}|y^2- 4cn^2|_q'^{1/2}}\right]\rme\legendresymbol{-x\xi}{kf^2}\rme_{q}\legendresymbol{-y\xi}{kf^2}\rmd x\rmd y.
\end{align*}

For any $c\in \{\pm 1\}\times q^{\FF_2^r}$, we consider the following function defined on $(a,b)=(a,b_1,\dots,b_r)\in \QQ_S=\RR\times \QQ_{q_1}\times\dots\times \QQ_{q_r}$:
\begin{align*}
&{}_cJ_{k,f}(X,\xi,\eta,a,b)=2G\legendresymbol{|a|_\infty|b|_q}{X}\int_{x\in\RR}\int_{y\in\QQ_{S_\fin}} \theta_\infty(x,ca^2)\theta_{q}(y,cb^2)\left[F\legendresymbol{kf^2}{|x^2-4ca^2|_\infty^{1/2}|y^2- 4cb^2|_q'^{1/2}}\right.\\
+&\left.\frac{kf^2}{\sqrt{|x^2-4ca^2|_\infty|y^2- 4cb^2|_q'}}V\legendresymbol{kf^2}{|x^2-4ca^2|_\infty^{1/2}|y^2- 4cb^2|_q'^{1/2}}\right]\rme\legendresymbol{-x\xi}{kf^2}\rme_{q}\legendresymbol{-y\eta}{kf^2}\rmd x\rmd y.
\end{align*}
\begin{lemma}\label{lem:secondsmooth}
${}_cJ_{k,f}(X,\xi,\eta,a,b)$ is a smooth and compactly supported function for $(a,b)\in \QQ_S$.
\end{lemma}
\begin{proof}
The proof is similar to that of \cite[Lemma 3.1]{cheng2025c}. We leave it to the reader.
\end{proof}
Clearly $|n^{(q)}|=|n|_\infty|n|_q$. Hence
\[
S_G(X)=\frac12\sum_{c\in \{\pm 1\}\times q^{\FF_2^r}}\sum_{k,f\in \ZZ_{(S)}^{>0}}\frac{1}{k^2f^3}\sum_{\xi,n\in \ZZ^S}J_{k,f}(X,\xi,n)\Kl_{k,f}^S(\xi,cn^2),
\]
where $J_{k,f}(X,\xi,n)=J_{k,f}(X,\xi,\xi,n,n)$.

By \autoref{lem:secondsmooth} we can apply the Poisson summation formula to the above sum. Since $\Kl_{k,f}^S(\xi,cn^2)$ is $kf^2$-periodic for $n$, by Corollary B.4 of \cite{cheng2025} we obtain
\[
S_G(X)=\frac12\sum_{c\in \{\pm 1\}\times q^{\FF_2^r}}\sum_{k,f\in \ZZ_{(S)}^{>0}}\frac{1}{k^3f^5}\sum_{\xi,\alpha\in \ZZ^S}{}_c\widehat{J}_{k,f}(X,\xi,\alpha){}^2_c\widehat{\Kl}_{k,f}^S(\xi,\alpha)
\]
and the sum converges absolutely, where
\begin{align*}
&\prescript{}{c}{\widehat{J}}_{k,f}(X,\xi,\eta,\alpha,\beta)=\int_{(x,a)\in\RR^2}\int_{(y,b)\in\QQ_{S_\fin}^2}G\legendresymbol{|a|_\infty|b|_q}{X} \theta_\infty(x,ca^2)\theta_{q}(y,cb^2)\left[F\legendresymbol{kf^2}{|x^2-4ca^2|_\infty^{1/2}|y^2- 4cb^2|_q'^{1/2}}\right.\\
&+\left.\frac{kf^2}{\sqrt{|x^2-4ca^2|_\infty|y^2- 4cb^2|_q'}}V\legendresymbol{kf^2}{|x^2-4ca^2|_\infty^{1/2}|y^2- 4cb^2|_q'^{1/2}}\right]\rme\legendresymbol{-x\xi-a\alpha}{kf^2}\rme_{q}\legendresymbol{-y\eta-b\beta}{kf^2}\rmd x\rmd a\rmd y\rmd b,
\end{align*}
$_c\widehat{J}_{k,f}(X,\xi,\alpha)=\prescript{}{c}{\widehat{J}}_{k,f}(X,\xi,\xi,\alpha,\alpha)$, and the \emph{transformed Kloosterman sum} is defined as
\begin{equation}\label{eq:transformedkloosterman}
\prescript{2}{c}{\Kl}_{k,f}^S(\xi,\alpha)=\sum_{m\bmod kf^2}\Kl_{k,f}^S(\xi,m)\rme\legendresymbol{m\alpha}{kf^2}\rme_q\legendresymbol{m\alpha}{kf^2}.
\end{equation}

\subsection{Simplification of the integral $\prescript{}{c}{\widehat{J}}_{k,f}(X,\xi,\eta,\alpha,\beta)$}
As in the standard representation case \cite{cheng2025c}, we need to a transformation on the $\QQ_S$-point of the Hitchin-Steinberg base $\mf{g}\sslash \G$. In our cases, the transformation is simply multiplying $kf^2$. More precisely, we make the transformation
\[
x\mapsto kf^2 x,\qquad y\mapsto kf^2y,\qquad a\mapsto kf^2 a,\qquad b\mapsto kf^2b. 
\]
The Jacobian of this transformation is $|kf^2|_\infty^2|kf^2|_q^2=k^2f^4$. Hence we obtain

\begin{align*} 
&\prescript{}{c}{\widehat{J}}_{k,f}(X,\xi,\eta,\alpha,\beta)=\int_{(x,a)\in\RR^2}\int_{(y,b)\in\QQ_{S_\fin}^2} k^2f^4 G\legendresymbol{kf^2|a|_\infty|b|_q}{X} \theta_\infty(kf^2x,k^2f^4ca^2)\\
\times&\theta_{q}(kf^2y,k^2f^4cb^2)\left[F\legendresymbol {1}{|x^2-4ca^2|_\infty^{1/2}|y^2- 4cb^2|_q'^{1/2}}+\frac{1}{|x^2-4ca^2|_\infty^{1/2}|y^2- 4cb^2|_q'^{1/2}}\right.\\
\times&\left.V\legendresymbol{1}{|x^2-4ca^2|_\infty^{1/2}|y^2- 4cb^2|_q'^{1/2}}\right]\rme(-x\xi-a\alpha)\rme_{q}(-y\eta-b\beta)\rmd x\rmd a\rmd y\rmd b.
\end{align*}

Since $f_\infty$ is $Z_+$-invariant, we have $\theta_\infty(zT,z^2N)=\theta_\infty(T,N)$ for $z\in \RR_{>0}$. Hence $\theta_\infty(kf^2x,k^2f^4ca^2)=\theta_\infty(x,ca^2)$, which is independent of $k$ and $f$. 

The nonarchimedean case is a bit different from the archimedean one. 
By Lemma 2.2 of \cite{cheng2025c}, for each $i\in \{1,\dots,n\}$ there exists $L_i\in \ZZ_{>0}$ such that for any $z\in 1+q_i^{L_i}\ZZ_{q_i}$ we have
$\theta_{q_i}(zT,z^2N)=\theta_{q_i}(T,N)$.
Hence
\[
\theta_{q}(zT,z^2N)=\theta_{q}(T,N)
\]
for any $z\in U=\prod_{i=1}^{r} (1+q_i^{L_i}\ZZ_{q_i})$, where the multiplication on $\QQ_{S_\fin}$ is defined componentwise.  By Fourier inversion formula,
\[
\theta_q(zy,z^2cb^2)=\prod_{i=1}^{r}(1-q_i^{-1})^{-1}\sum_{\chi} \widehat{\theta}_q(\chi,y,cb^2)\chi(z),
\]
where the Fourier transform is defined as
\[
\widehat{\theta}_q(\chi,T,N)=\int_{\ZZ_{S_\fin}^\times}\theta_q(zT,z^2N)\overline{\chi(z)}\rmd z.
\]
Moreover,  $\chi$ runs over all characters on $\ZZ_{S_\fin}^\times$ that is trivial on $U$. Since $\ZZ_{S_\fin}^\times/U$ is finite, the sum over $\chi$ is finite (which is independent of $T$ and $N$). We also consider them as Dirichlet characters such that $\chi(m)=0$ if $\gcd(m,S)\neq 1$. Therefore we obtain
\begin{equation}\label{eq:transformedj}
\begin{split} 
&\prescript{}{c}{\widehat{J}}_{k,f}(X,\xi,\eta,\alpha,\beta)=\prod_{i=1}^{r}(1-q_i^{-1})^{-1} \int_{(x,a)\in\RR^2}\int_{(y,b)\in\QQ_{S_\fin}^2} k^2f^4 G\legendresymbol{kf^2|a|_\infty|b|_q}{X} \theta_\infty(x,ca^2)\\
\times&\sum_{\chi}\widehat{\theta}_{q}(\chi,y,cb^2)\chi(kf^2)\left[F\legendresymbol {1}{|x^2-4ca^2|_\infty^{1/2}|y^2- 4cb^2|_q'^{1/2}}+\frac{1}{|x^2-4ca^2|_\infty^{1/2}|y^2- 4cb^2|_q'^{1/2}}\right.\\
\times&\left.V\legendresymbol{1}{|x^2-4ca^2|_\infty^{1/2}|y^2- 4cb^2|_q'^{1/2}}\right]\rme(-x\xi-a\alpha)\rme_{q}(-y\eta-b\beta)\rmd x\rmd a\rmd y\rmd b.
\end{split}
\end{equation}

\section{Contribution of the $\xi=0$ term}
Our main goal in this section is to give an asymptotic formula for
\[
S_G^{\xi=0}(X)=\frac12\sum_{c\in \{\pm 1\}\times q^{\FF_2^r}}\sum_{k,f\in \ZZ_{(S)}^{>0}}\frac{1}{k^3f^5}\sum_{\alpha\in \ZZ^S}\prescript{}{c}{J}_{k,f}(X,0,\alpha)\prescript{2}{c}{\Kl}_{k,f}^S(0,\alpha).
\]
For the $\xi=0$ term, we actually do not need the second Poisson summation. We write $S_G^{\xi=0}(X)$ as the original definition.
\[
S_G^{\xi=0}(X)=\sum_{\substack{n=1\\\gcd(n,S)=1}}^{+\infty}G\legendresymbol{n}{X}\Sigma^{n^2}(\xi=0).
\]

For simplicity we only consider the case that $v_1=\infty$ and $v_2=q_1$ in \autoref{ass:nonarchimedean}. The other cases are similar but with different results. We leave the computation to the reader.
\begin{theorem}\label{thm:contributionxi0}
Under \autoref{ass:nonarchimedean}, we have
\[
S_G^{\xi=0}(X)=\ms{A}\widetilde{G}(1) X+O(\|G\|_{3,1}X^{\frac12+\varepsilon})
\]
for any $\varepsilon>0$,
where
\begin{align*}
\ms{A}&=-\sum_{\chi^2=\triv}\sum_{\pm}\sum_{\nu\in \ZZ^r}\frac{2^{r+1}\uppi}{4\zeta^S(2)}\frac{\log q_1}{1-q_1^{-1}}\prod_{i=1}^{r}(1-q_i^{-1})\int_{X_1}\frac{\theta_\infty^\pm(x)}{|x^2\mp 1|_\infty^{1/2}}\rmd x \int_{\ZZ_{q}^\times}\int_{Y_{\epsilon^0}} \frac{\theta_{q}(y,\pm zq^\nu)}{|y^2\mp 4zq^\nu|_q'^{1/2}}\overline{\chi(z)}\rmd z\rmd y \\
&-\sum_{\chi^2=\triv}\sum_{\pm}\sum_{\nu\in \ZZ^r}\frac{2^{r}\uppi}{4\zeta^S(2)}\frac{(1-q_1^{-2})\log q_1}{1-q_1^{-1}}\prod_{i=1}^{r}(1-q_i^{-1})\int_{X_1}\frac{\theta_\infty^\pm(x)}{|x^2\mp 1|_\infty^{1/2}}\rmd x \int_{\ZZ_{q}^\times}\int_{Y_{\epsilon^{-1}}} \frac{\theta_{q}(y,\pm zq^\nu)}{|y^2\mp 4zq^\nu|_q'^{1/2}}\overline{\chi(z)}\rmd z\rmd y,
\end{align*}
where $\epsilon^0, \epsilon^{-1}\in \{0,\pm 1\}^r$ satisfy $\epsilon^{0}_1=0$, $\epsilon_1^{-1}=-1$, and $\epsilon_i^0=\epsilon_i^{-1}=1$ for $i=2,3,\dots,n$. The sum over $\chi$ is over all the quadratic characters such that it is unramified at places outside $S$. The implied constant depends only on $f_\infty$, $f_{q_i}$ and $\varepsilon$.
\end{theorem}

\subsection{Preparations}

\begin{lemma}\label{lem:kloostermandirichlet}
Let $\chi$ be a Dirichlet character which is unramified at the primes that do not belong to $S$, then for any complex numbers $s,u$,
\[
\sum_{n\in \ZZ_{(S)}^{>0}}\frac{\chi(n)}{n^{u-s}}\sum_{k,f\in \ZZ_{(S)}^{>0}}\frac{\Kl_{k,f}^S(0,\pm n^2q^\nu)}{k^{1+s}f^{1+2s}}=\frac{\zeta^S(2s)}{\zeta^S(s+1)} \frac{L^S(u-s,\chi)L^S(u+s,\chi)L^S(u,\chi)}{L^S(2u,\chi^2)}
\]
whenever the series converges absolutely. 
\end{lemma}
\begin{proof}
By Corollary 5.6 of \cite{cheng2025}, the left hand side equals
\[
\sum_{n\in \ZZ_{(S)}^{>0}}\frac{\chi(n)}{n^{u- s}}\frac{\zeta^S(2s)}{\zeta^S(s+1)}\prod_{\substack{p\mid n\\p\notin S}}\frac{1-p^{-s(v_p(n^2)+1)}}{1-p^{-s}}=\frac{\zeta^S(2s)}{\zeta^S(s+1)}\prod_{p\notin S}\sum_{m=0}^{+\infty}\frac{\chi(p^m)}{p^{m(u-s)}}\frac{1-p^{-s(2m+1)}}{1-p^{-s}}.
\]
Since
\begin{align*}
\sum_{m=0}^{+\infty}\frac{\chi(p^m)}{p^{m(u-s)}}\frac{1-p^{-s(2m+1)}}{1-p^{-s}}&= \frac{1}{1-p^{-s}}\left[\sum_{m=0}^{+\infty}\legendresymbol{\chi(p)}{p^{u-s}}^m-p^{-s}\sum_{m=0}^{+\infty} \legendresymbol{\chi(p)}{p^{u+s}}^m\right]\\
&=\frac{1}{1-p^{-s}}\left[\frac{1}{1-\chi(p)p^{s-u}}-p^{-s}\frac{1}{1-\chi(p)p^{-s-u}}\right]\\
&=\frac{1+\chi(p)p^{-u}}{(1-\chi(p)p^{s-u})(1-\chi(p)p^{-s-u})}\\\
&= \frac{1-\chi(p)^2p^{-2u}}{(1-\chi(p)p^{s-u})(1-\chi(p)p^{-s-u})(1-\chi(p)p^{-u})},
\end{align*}
we obtain
\begin{align*}
\sum_{n\in \ZZ_{(S)}^{>0}}\frac{\chi(n)}{n^{u-s}}\sum_{k,f\in \ZZ_{(S)}^{>0}}\frac{\Kl_{k,f}^S(0,\pm n^2q^\nu)}{k^{1+s}f^{1+2s}}&=\frac{\zeta^S(2s)}{\zeta^S(s+1)} \prod_{p\notin S} \frac{1-\chi(p)^2p^{-2u}}{(1-\chi(p)p^{s-u})(1-\chi(p)p^{-s-u})(1-\chi(p)p^{-u})}\\
&=\frac{\zeta^S(2s)}{\zeta^S(s+1)} \frac{L^S(u-s,\chi)L^S(u+s,\chi)L^S(u,\chi)}{L^S(2u,\chi^2)}.\qedhere
\end{align*}
\end{proof}

\begin{lemma}\label{lem:supercuspidaltrivialvanish}
Suppose that $v\in S$ such that $f_v$ is supercuspidal, then
\begin{enumerate}[itemsep=0pt,parsep=0pt,topsep=0pt, leftmargin=0pt,labelsep=2.5pt,itemindent=15pt,label=\upshape{(\arabic*)}]
  \item If $v=\infty$ is archimedean, then 
  \[
  \sum_{\pm}\int_{\RR}\theta_\infty^\pm (x)\rmd x=0.
  \]
\item If $v=q_i$ is nonarhcimedean, then
\[
\sum_{\nu\in \ZZ}q_i^{\nu}\int_{z\in\ZZ_{q_i}^\times}\int_{y\in\QQ_{q_i}}\theta_{q}(y,z q_i^\nu)\rmd y\rmd z=0.
\]
\end{enumerate}
\end{lemma}
\begin{proof}
(1) By Proposition 7.1 of \cite{cheng2025}, we have
\[
\sum_{\pm}\int_{\RR}\theta_\infty^\pm (x)\rmd x=\frac12\int_{Z_+\bs \G(\RR)}f_\infty(g)\rmd g.
\]
By Iwasawa decomposition, we have
\[
\int_{Z_+\bs \G(\RR)}f_\infty(g)\rmd g=\int_{Z_+\bs \A(\RR)}\int_{K}\int_{\N(\RR)}f(tnk)\rmd n\rmd k\rmd t,
\]
which vanishes by supercuspidality.

(2) By Proposition 7.3 of \cite{cheng2025}, we have
\[
\sum_{\nu\in \ZZ}q_i^{\nu}\int_{z\in\ZZ_{q_i}^\times}\int_{y\in\QQ_{q_i}}\theta_{q}(y,z q_i^\nu)\rmd y\rmd z=\sum_{\nu\in \ZZ}q_i^{2\nu}\int_{y\in \QQ_{q_i}}\int_{|z|_{q_i}= q_i^{-\nu}}\theta_{q_i}(y,z)\rmd y\rmd z=\left(1-\frac{1}{q_i}\right)\int_{\G(\QQ_{q_i})}f_{q_i}(g)\rmd g,
\]
which vanishes by a same argument.
\end{proof}

\begin{lemma}\label{lem:hyperbolicvanish}
Suppose that $v\in S$ such that all the hyperbolic orbital integral of $f_v$ vanishes (in particular, if $f_v$ is supercuspidal). Then
\begin{enumerate}[itemsep=0pt,parsep=0pt,topsep=0pt, leftmargin=0pt,labelsep=2.5pt,itemindent=15pt,label=\upshape{(\arabic*)}]
  \item If $v=\infty$ is archimedean, then for any sign $\pm$, $\theta_\infty^{\pm}(x)$ vanishes on $X_0$.
\item If $v=q_i$ is nonarhcimedean, then for any $N\in \QQ_{q_i}$, $\theta_{q_i}(y,N)$ vanishes on $Y_1$.
\end{enumerate}
\end{lemma}

\begin{proof}
(1) Recall that
\[
\orb(f_\infty;\gamma)\left|\frac{T_\gamma^2-4N_\gamma}{N_\gamma}\right|^{1/2} 
     = \theta_\infty^{\sgn N_\gamma}\left(\frac{T_\gamma}{2\sqrt{|N_\gamma|}}\right).
\]
Also, it is easily checked that $\gamma$ is hyperbolic if and only if 
\[
\omega_\infty\left(\frac{T_\gamma}{2\sqrt{|N_\gamma|}}\right)=0.
\]
Hence the conclusion follows.

(2) For $y\in Y_1$, let $\gamma\in \G(\QQ_{q_i})$ with $\Tr \gamma=y$ and $\det\gamma=N$. Since
\[
\omega_i(y)=\legendresymbol{(y^2-4N)|y^2-4N|_{q_i}'}{q_i}=1
\]
is equivalent to $\gamma$ being hyperbolic, we obtain the desired result.
\end{proof}

\begin{lemma}\label{lem:integralsmooth}
Let $p\in S$ be a prime, $c\in \QQ_p$. Consider the integral
\[
\Phi(z,s)=\int_{\QQ_p}\theta_p(y,c z)|y^2-4c z|_{p}'^{s}\rmd y
\]
for $z\in\ZZ_{p}^\times$. Then $\Phi(z,s)$ converges absolutely and uniformly in $z$, uniformly on each compact subset of $\Re s>-1$. $\Phi(z,s)$ is holomorphic and bounded on each vertical strips on $\Re s>-1$. Moreover, there exists $M>0$ such that for any $a\in 1+p^M\ZZ_p$, $\Phi(z,s)=\Phi(az,s)$ for any $z\in\ZZ_{p}^\times$ and any $\Re s>-1$.

The same formula holds if we replace the domain of integration by $\{y\in \QQ_p\ |\ \omega(y)=\varepsilon\}$, where
\[
\omega(y)=\legendresymbol{(y^2-4c z)|y^2-4c z|_p'}{p},
\]
and $\epsilon\in \{0,\pm 1\}$.
\end{lemma}
\begin{proof}
Since $y$ only has singularities of order $1$ at the square roots of $4cz$, the integral converges on each compact subset of $\Re s>-1$, holomorphic and bounded on each vertical strips there for \emph{fixed} $z$.

Since $f_p$ is smooth and compactly supported, there exists $L\geq 2$ such that $f_p(bg)=f_p(g)$ for any $b\in 1+p^{L}\ZZ_p$ and $g\in \G(\QQ_p)$. Hence $\orb(f_p;b\gamma)=\orb(f_p;\gamma)$ for any $b\in 1+p^{L}\ZZ_p$ and regular elements $\gamma$. Clearly $b\gamma$ is split (resp. inert, ramified) if and only if $\gamma$ is. Hence $\theta_p(b\gamma)=\theta_p(\gamma)$ for any $b\in 1+p^{L}\ZZ_p$ and regular elements $\gamma$. Thus for any $T\in \QQ_p$ and $N\in \QQ_p^\times$ with $T^2-4N\neq 0$, we have
\[
\theta_p(bT,b^2N)=\theta_p(T,N).
\]
By continuity, the above equation also holds for $T^2-4N=0$.

Since $b^2\in 1+p^{L}\ZZ_p$ we obtain $|y^2- 4c z|_p'=|b^2y^2- 4c b^2z|_p'$. Hence 
\[
\Phi(z,s)=\int_{\QQ_p}\theta_p(by,c b^2z)|b^2y^2-4c b^2z|_{p}'^{s}\rmd y= \int_{\QQ_p}\theta_p(y,c b^2z)|y^2-4c b^2z|_{p}'^{s}\rmd y=\Phi(b^2z,s),
\]
where in the second equality we made change of variable $y\mapsto b^{-1}y$.

Now for $M$ sufficiently large, we can always write $a=b^2$ with $b\in 1+p^{L}\ZZ_p$ for $a\in 1+p^{M}\ZZ_p$. Hence $\Phi(z,s)=\Phi(az,s)$ for any $z\in\ZZ_{p}^\times$. Finally, the uniformity in $z$ follows from the fact that $\ZZ_p^\times /(1+p^M\ZZ_p)$ is finite.

The same proof also holds if we replace the domain of integration by $\{y\in \QQ_p\ |\ \omega(y)=\varepsilon\}$ since the set is invariant under $y\mapsto b^{-1}y$ for $b\in 1+p^{L}\ZZ_p$ with $L\geq 2$.
\end{proof}

\subsection{Proof of \autoref{thm:contributionxi0}}

Now we prove \autoref{thm:contributionxi0}.
The sum $S_G^{\xi=0}(X)$ can be split into the following two terms
\begin{equation}\label{eq:sigmaxi01}
\begin{split}
&\sum_{n\in \ZZ_{(S)}^{>0}}G\legendresymbol{n}{X}4n\sum_{\pm}\sum_{\nu\in \ZZ^r}q^{\nu/2}\sum_{f\in \ZZ_{(S)}^{>0}}\sum_{k\in \ZZ_{(S)}^{>0}}\frac{1}{k^2f^3}\Kl_{k,f}^S(0,\pm n^2q^\nu)\int_{x\in\RR}\int_{y\in\QQ_{S_\fin}}\theta_\infty^\pm(x)\theta_{q}(y,\pm n^2q^\nu) \\
  &\times F\legendresymbol{kf^2(4n^2q^\nu)^{-1/2}}{|x^2\mp 1|_\infty^{1/2}|y^2\mp 4n^2q^\nu|_q'^{1/2}}\rmd x\rmd y
\end{split}
\end{equation}
and
\begin{equation}\label{eq:sigmaxi02}
\begin{split}
&\sum_{n\in \ZZ_{(S)}^{>0}}G\legendresymbol{n}{X}4n\sum_{\pm}\sum_{\nu\in \ZZ^r}q^{\nu/2}\sum_{f\in \ZZ_{(S)}^{>0}}\sum_{k\in \ZZ_{(S)}^{>0}}\frac{1}{k^2f^3}\Kl_{k,f}^S(0,\pm n^2q^\nu)\int_{x\in\RR}\int_{y\in\QQ_{S_\fin}}\theta_\infty^\pm(x)\theta_{q}(y,\pm n^2q^\nu) \\
  &\times \frac{kf^2n^{-1}q^{-\nu/2}}{2\sqrt{|x^2\mp 1|_\infty|y^2\mp 4n^2q^\nu|_q'}}V\legendresymbol{kf^2(4n^2q^\nu)^{-1/2}}{|x^2\mp 1|_\infty^{1/2}|y^2\mp 4n^2q^\nu|_q'^{1/2}}\rmd x\rmd y.
\end{split}
\end{equation}

\begin{proposition}\label{prop:sigmaxi01}
Assume that \autoref{ass:nonarchimedean} holds. Then for any $\varepsilon>0$, we have 
\begin{align*}
\eqref{eq:sigmaxi01}=&2\widetilde{G}\legendresymbol{3}{2}\sum_{\chi^2=\triv}\sum_{\pm}\sum_{\nu\in \ZZ^r}\widetilde{F}\!\left(-\frac12\right)\sqrt{2}q^{\nu/4} \int_{\RR}\frac{\theta_\infty^\pm(x)}{|x^2\mp 1|_\infty^{1/4}}\rmd x\\
\times&\prod_{i=1}^{r}(1-q_i^{-1})\int_{\ZZ_{q_i}^\times}\int_{\QQ_{q_i}}\frac{\theta_{q_i}(y_i,\pm z_i q^\nu)}{|y_i^2\mp 4z_iq^\nu|_{q_i}'^{1/4}}\overline{\chi_i(z_i)}\rmd z_i\rmd y_i \frac{\zeta^S(2)}{\zeta^S(3)}X^{\frac32}+O(X^\varepsilon).
\end{align*}
The implied constant depends only on $m,f_{q_i}$ and $\varepsilon$.
\end{proposition}
\begin{proof}
By Mellin inversion formula on $G$ we have
\begin{align*}
\eqref{eq:sigmaxi01}=&\frac{4}{\dpii} \int_{(6)}\widetilde{G}(u)\sum_{n\in \ZZ_{(S)}^{>0}}\frac{1}{n^{u-1}}\sum_{\pm}\sum_{\nu\in \ZZ^r}q^{\nu/2}\sum_{k,f\in \ZZ_{(S)}^{>0}}\frac{1}{k^2f^3}\Kl_{k,f}^S(0,\pm n^2q^\nu)\\
  \times&\int_{x\in\RR}\int_{y\in\QQ_{S_\fin}}\theta_\infty^\pm(x)\theta_{q}(y,\pm n^2q^\nu)  F\legendresymbol{kf^2(4n^2q^\nu)^{-1/2}}{|x^2\mp 1|_\infty^{1/2}|y^2\mp 4n^2q^\nu|_q'^{1/2}}\rmd x\rmd y X^u\rmd u.
\end{align*}
Since $\Kl_{k,f}^S(0,\pm n^2q^\nu)\ll n$ (by Proposition B.2 in \cite{cheng2025b}) and $\theta_\infty^\pm(x)$, $\theta_{q_i}(y,\pm n^2q^\nu)$ are compactly supported (independent of $n\in \ZZ_{(S)}^{>0}$), the integrand over $u$ converges absolutely for $\Re u>3$ and thus the right hand side of the above equation is well-defined. By Mellin inversion formula on $F$, the above equation equals
\begin{align*}
\eqref{eq:sigmaxi01}=&\frac{4}{\dpii}\frac{1}{\dpii} \int_{(6)}\widetilde{G}(u)\sum_{\pm}\sum_{\nu\in \ZZ^r}q^{\nu/2}\int_{(2)}\widetilde{F}(s)(4q^\nu)^{s/2}\int_{x\in\RR}\theta_\infty^\pm(x)|x^2\mp 1|_\infty^{s/2}\rmd x\\
  \times& \sum_{n\in \ZZ_{(S)}^{>0}}\int_{y\in\QQ_{S_\fin}}\theta_{q}(y,\pm n^2q^\nu)|y^2\mp 4n^2q^\nu|_q'^{s/2}\rmd y \frac{1}{n^{u-1-s}}\sum_{k,f\in \ZZ_{(S)}^{>0}}\frac{\Kl_{k,f}^S(0,\pm n^2q^\nu)}{k^{2+s}f^{3+2s}}\rmd s X^u\rmd u.
\end{align*}
Note that the above sum for $f$ and $k$ converges absolutely for $\Re s>-1$ and the sum for $n$ converges absolutely for $\Re (u-1-s)>2$ since $\Kl_{k,f}^S(0,\pm n^2q^\nu)\ll n$. Since $\widetilde{F}(s)$ has rapid decay vertically and $\theta_\infty^\pm(x)$, $\theta_{q_i}(y,\pm n^2q^\nu)$ are compactly supported (independent of $n\in \ZZ_{(S)}^{>0}$), all the sums and integrals converge absolutely.
Hence we can change the order of the integrals and the sums. 

Now we define 
\[
\Phi_i(z,s)=\int_{\QQ_{q_i}}\theta_{q_i}(y,\pm z q^\nu)|y^2\mp 4zq^\nu|_{q_i}'^{s}\rmd y.
\]
For any $\chi_i\in \widehat{\ZZ_{q_i}^\times}$, we define the Fourier transform
\[
\widehat{\Phi_i}(\chi_i,s)=\int_{\ZZ_{q_i}^\times}\Phi_i(z,s)\overline{\chi_i(z)}\rmd z.
\]
By \autoref{lem:integralsmooth}, there exists $L_i>0$ such that $\Phi_i(az,s)=\Phi_i(z,s)$ for any $a\in 1+q_i^{L_i}\ZZ_{q_i}$. If $\chi_i$ is not trivial on $1+q_i^{L_i}\ZZ_{q_i}$, then there exists $a\in 1+q_i^{L_i}\ZZ_{q_i}$ such that $\chi_i(a)\neq 1$. Hence
\[
\widehat{\Phi_i}(\chi_i,s)=\int_{\ZZ_{q_i}^\times}\Phi_i(az,s)\overline{\chi_i(az)}\rmd z=\overline{\chi_i(a)}\int_{\ZZ_{q_i}^\times}\Phi_i(z,s)\overline{\chi_i(z)}\rmd z=\overline{\chi_i(a)}\widehat{\Phi_i}(\chi_i,s).
\]
Hence $\widehat{\Phi_i}(\chi_i,s)\neq 0$ only if $\chi_i\in (\ZZ_{q_i}/1+q_i^{L_i}\ZZ_{q_i})\sphat\,$. Thus such $\chi_i$ is finite. Hence by Fourier inversion formula
\[
\Phi_i(z,s)=(1-q_i^{-1})^{-1}\sum_{\chi_i}\widehat{\Phi_i}(\chi_i,s)\chi_i(z).
\] 
Thus for any $z\in \ZZ_{S_\fin}^\times$ we have
\begin{align*}
\int_{y\in\QQ_{S_\fin}}\theta_{q}(y,\pm zq^\nu)|y^2\mp 4zq^\nu|_q'^{s/2}\rmd y&=\prod_{i=1}^{r}\Phi_i(z_i, s/2)=\prod_{i=1}^{r}(1-q_i^{-1})^{-1}\sum_{\chi_i}\widehat{\Phi_i}(\chi_i,s/2) \chi_i(z_i)\\
&=\sum_{\chi}\chi(z)\prod_{i=1}^{r}(1-q_i^{-1})^{-1}\widehat{\Phi_i} (\chi_i,s/2),
\end{align*}
where $\chi$ runs over all characters on $\ZZ_{S_\fin}^\times$ and the sum over $\chi$ is actually \emph{finite}. We also consider $\chi$ as Dirichlet characters so that $\chi(m)=0$ if $\gcd(m,S)\neq 1$.

Hence we obtain
\begin{align*}
\eqref{eq:sigmaxi01}=&\prod_{i=1}^{r}(1-q_i^{-1})^{-1}\frac{4}{\dpii}\frac{1}{\dpii} \int_{(6)}\widetilde{G}(u)\sum_{\pm}\sum_{\nu\in \ZZ^r}q^{\nu/2}\int_{(2)}\widetilde{F}(s)(4q^\nu)^{s/2}\int_{x\in\RR}\theta_\infty^\pm(x)|x^2\mp 1|_\infty^{s/2}\rmd x\\
  \times& \sum_{n\in \ZZ_{(S)}^{>0}}\sum_{\chi}\chi(n^2)\prod_{i=1}^{n}\widehat{\Phi_i}(\chi_i,s/2)\frac{1} {n^{u-1-s}}\sum_{k,f\in \ZZ_{(S)}^{>0}}\frac{\Kl_{k,f}^S(0,\pm n^2q^\nu)}{k^{2+s}f^{3+2s}} X^u\rmd u\rmd s.
\end{align*}
By \autoref{lem:kloostermandirichlet}, we obtain
\begin{align*}
\eqref{eq:sigmaxi01}=&\prod_{i=1}^{r}(1-q_i^{-1})^{-1}\frac{4}{\dpii}\frac{1}{\dpii} \int_{(6)}\widetilde{G}(u)\sum_{\pm}\sum_{\nu\in \ZZ^r}q^{\nu/2}\int_{(2)}\widetilde{F}(s)(4q^\nu)^{ s/2}\int_{x\in\RR}\theta_\infty^\pm(x)|x^2\mp 1|_\infty^{s/2}\rmd x\\
  \times&\sum_{\chi}\prod_{i=1}^{n}\widehat{\Phi_i}(\chi_i,s/2) \frac{\zeta^S(2s+2)}{\zeta^S(s+2)} \frac{L^S(u-s-1,\chi^2)L^S(u+s+1,\chi^2)L^S(u,\chi^2)}{L^S(2u,\chi^4)}X^u\rmd u\rmd s.
\end{align*}
Since
\[
\widehat{\Phi_i}(\chi_i,s)=\int_{\ZZ_{q_i}^\times}\Phi_i(z,s)\overline{\chi_i(z)}\rmd z
\]
and $\Phi_i(z,s)$ converges uniformly in $z$ and any compact subset of $\Re s>-1$ (\autoref{lem:integralsmooth}), $\widehat{\Phi_i}(\chi_i,s)$ is holomorphic on $\Re s>-1$.

Now we move the $s$-contour from $(2)$ to $(6)$. The integrand is holomorphic on $\Re s>2$ and $\Re u>5$ except for the function $L^S(u-1-s,\chi^2)$ when $\chi^2=\triv$, which has a pole at $s=u-2$ with residue
\[
\res_{s=u-2}L^S\left(u-1-s,\triv\right)=\res_{s=u-2}\prod_{i=1}^{r}\left(1-\frac{1}{q_i^{ u-1-s}}\right)\zeta(u-1-s)=-\prod_{i=1}^{r}(1-q_i^{-1}).
\]
Hence by residue formula, we know that \eqref{eq:sigmaxi01} equals
\begin{align*}
&\prod_{i=1}^{r}(1-q_i^{-1})^{-1}\frac{2}{\dpii}\frac{1}{\dpii} \int_{(6)}\widetilde{G}(u)\sum_{\pm}\sum_{\nu\in \ZZ^r}q^{\nu/2}\int_{(6)}\widetilde{F}(s)(4q^\nu)^{s/2}\int_{x\in\RR}\theta_\infty^\pm(x)|x^2\mp 1|_\infty^{s/2}\rmd x\\
  \times&\sum_{\chi}\prod_{i=1}^{n}\widehat{\Phi_i}(\chi_i, s/2)\frac{\zeta^S(2s+2)}{\zeta^S(s+2)} \frac{L^S(u-s-1,\chi^2)L^S(u+s+1,\chi^2)L^S(u,\chi^2)}{L^S(2u,\chi^4)}\frac{X^u}{u}\rmd u\rmd s\\
  +&\frac{2}{\dpii}\sum_{\chi^2=\triv}\int_{(6)}\widetilde{G}(u)\sum_{\pm}\sum_{\nu\in \ZZ^r}q^{\nu/2}\widetilde{F}(u-2)(4q^\nu)^{u/2-1} \int_{x\in\RR}\theta_\infty^\pm(x)|x^2\mp 1|_\infty^{u/2-1}\rmd x\\
  \times&\prod_{i=1}^{n}\widehat{\Phi_i}(\chi_i,u/2-1)\frac{\zeta^S(2u-2)}{\zeta^S(u)} \frac{L^S(2u-1,\triv)L^S(u,\triv)}{L^S(2u,\triv)} X^u\rmd u.
\end{align*}
Note that since we move the contour from the left to the right, the residue formula has a minus sign.

First we consider the first term above. Since $\widetilde{G}(u)$ and $\widetilde{F}(s)$ have rapid decay vertically, the sums and integrals converge absolutely. Also, the integrand is holomorphic on $0<\Re u\leq 6$ when $\Re s=6$ is fixed. Hence we can change the order of the integration over $s$ and $u$ and then move the $u$-contour to $(1/2)$ and then move the $s$-contour to $(\epsilon)$. Thus the first term is bounded by
\begin{align*}
&\sum_{\pm}\sum_{\nu\in \ZZ^r}q^{\nu/2}\int_{(\varepsilon)}|\widetilde{F}(s)|(4q^\nu)^{\Re s/2}\int_{x\in\RR}|\theta_\infty^\pm(x)||x^2\mp 1|_\infty^{\Re s/2}\rmd x\sum_{\chi}\prod_{i=1}^{n}|\widehat{\Phi_i}(\chi_i,s/2)|\\
  \times&\int_{(\frac12)}|\widetilde{G}(u)|\left|\frac{\zeta^S(2s+2)}{\zeta^S(s+2)} \frac{L^S(u-s-1,\chi^2)L^S(u+s+1,\chi^2)L^S(u,\chi^2)}{L^S(2u,\chi^4)}\right|\rmd |s||X^u|\rmd |u|.
\end{align*}
Using the functional equation of Dirichlet $L$-functions and the Stirling formula, we have
\begin{equation}\label{eq:boundlvertical}
L^S(\sigma+\rmi t,\chi)\ll_\sigma (1+|t|)^{\frac12-\sigma}
\end{equation}
if $\sigma<0$.

For $\Re u=1/2$ and $\Re s=\varepsilon$, we have
\begin{enumerate}[itemsep=0pt,parsep=0pt,topsep=0pt,leftmargin=0pt,labelsep=3pt,itemindent=9pt,label=\textbullet]
  \item $L^S(u,\chi^2)\ll (1+|u|)^{1/6+\varepsilon}$ by the Weyl bound of the Dirichlet $L$-functions \cite{petrow2023}.
  \item Since $\Re(u+s+1)>1$, we have $L^S(u+s+1,\chi^2)\ll 1$.
  \item Since $\Re(u-s-1)=-1/2-\varepsilon$, we have $L^S(u+s+1,\chi^2)\ll (1+|u|)^{1+\varepsilon}$.
  \item The function $1/L^S(2u,\chi^4)$ can be bounded by $(1+|u|)^\varepsilon$ for $\Re(2u)=1$ \cite[Chapter II of \SSec 8.3]{tenenbaum2015analytic}. 
\end{enumerate} 
Hence the first term is bounded by
\[
\int_{(\frac12)}|\widetilde{G}(u)|(1+|u|)^{7/6+\varepsilon} \rmd |u|X^{1/2}\ll \|G\|_{3,1}X^{1/2}
\]
by \autoref{cor:mellinnorm}.

The second term equals
\begin{equation}\label{eq:secondterm1}
\begin{split}
&\frac{4}{\dpii}\sum_{\chi^2=\triv}\int_{(6)}\widetilde{G}(u)\sum_{\pm}\sum_{\nu\in \ZZ^r}q^{\nu/2}\widetilde{F}(u-2)(4q^\nu)^{u/2-1}\\
\times&\int_{\RR}\theta_\infty^\pm(x)|x^2\mp 1|_\infty^{u/2-1}\rmd x\widehat{\Phi_i}(\chi_i,u/2-1) \frac{\zeta^S(2u-2)\zeta^S(2u-1)}{\zeta^S(2u)}X^u\rmd u.
\end{split}
\end{equation}

$\widehat{\Phi_i}(\triv,u/2-1)$ is holomorphic on $\Re u>0$ and bounded on vertical strips by \autoref{lem:integralsmooth}. Moreover, the archimedean integral
\[
\int_{x\in\RR}\theta_\infty^\pm(x)|x^2\mp 1|_\infty^{u-3/2}\rmd x
\]
is holomorphic on $\Re u>1/2$ and bounded on vertical strips. For any $\varepsilon>0$, we move the contour from $(6)$ to $(\varepsilon)$. The integrand is holomorphic on $0<\Re u\leq 6$ except for the poles of $\widetilde{F}(u-2)$ at $u=2$, $\zeta^S(2u-2)$ at $u=3/2$ and $\zeta^{S}(2u-1)$ at $u=1$ with residues
\[
\res_{u=2}\widetilde{F}(u-2)=\lim_{u\to 2}(u-2)\widetilde{F}(u-2)=1,
\]
\[
\res_{u=3/2}\zeta^S(2u-2)=\lim_{u\to 3/2}(u-3/2)\prod_{i=1}^{r}\left(1-\frac{1}{q_i^{2u-2}}\right) \zeta(2u-2)=\frac{1}{2}\prod_{i=1}^{r}\left(1-\frac{1}{q_i}\right),
\]
and
\[
\res_{u=1}\zeta^S(2u-1)=\frac{1}{2}\prod_{i=1}^{r}\left(1-\frac{1}{q_i}\right).
\]
By moving the contour from $(6)$ to $(1/2+\varepsilon)$, \eqref{eq:secondterm1} equals
\begin{align*}
&\frac{4}{\dpii}\int_{(\frac12+\varepsilon)}\sum_{\chi^2=\triv}\widehat{G}(u)\sum_{\pm}\sum_{\nu\in \ZZ^r}q^{\nu/2}\widetilde{F}(u-2)(4q^\nu)^{\frac u2-1}\int_{\RR}\theta_\infty^\pm(x)|x^2\mp 1|_\infty^{\frac u2-1}\rmd x\\
\times&\prod_{i=1}^{r}\widehat{\Phi_i}(\chi_i,u/2-1) \frac{\zeta^S(2u-2)\zeta^S(2u-1)}{\zeta^S(2u)}\frac{X^u}{u}\rmd u\\
+&4\res_{u=2}\cdots+4\res_{u=\frac{3}{2}}\cdots+4\res_{u=1}\cdots\\
=&\frac{4}{\dpii}\int_{(\frac12+\varepsilon)}\sum_{\chi^2=\triv}\widehat{G}(u)\sum_{\pm}\sum_{\nu\in \ZZ^r}q^{\nu/2}\widetilde{F}(u-2)(4q^\nu)^{\frac u2-1}\int_{\RR}\theta_\infty^\pm(x)|x^2\mp 1|_\infty^{\frac u2-1}\rmd x\\
\times&\prod_{i=1}^{r}\widehat{\Phi_i}(\chi_i,u/2-1) \frac{\zeta^S(2u-2)\zeta^S(2u-1)}{\zeta^S(2u)}\frac{X^u}{u}\rmd u\\
+&4\widetilde{G}(2)\sum_{\chi^2=\triv}\sum_{\pm}\sum_{\nu\in \ZZ^r}q^{\nu/2}\int_{\RR}\theta_\infty^\pm(x)\rmd x\prod_{i=1}^{r}\left(1-\frac{1}{q_i}\right)\prod_{i=1}^{r}\widehat{\Phi_i}(\chi_i,0) \frac{\zeta^S(2)\zeta^S(3)}{\zeta^S(4)}X^2\\
+&4\widetilde{G}\legendresymbol{3}{2}\sum_{\chi^2=\triv}\sum_{\pm}\sum_{\nu\in \ZZ^r}q^{\nu/2}\widetilde{F}\!\left(-\frac12\right)(4q^\nu)^{-\frac14} \int_{\RR}\frac{\theta_\infty^\pm(x)}{|x^2\mp 1|_\infty^{1/4}}\rmd x\prod_{i=1}^{r}\widehat{\Phi_i}(\chi_i,-1/4)(1-q_i^{-1}) \frac{\zeta^S(2)}{\zeta^S(3)}X^{\frac32}\\
+&4\widetilde{G}(1)\sum_{\chi^2=\triv}\sum_{\pm}\sum_{\nu\in \ZZ^r}q^{\nu/2}\widetilde{F}(-1)(4q^\nu)^{-\frac12}\int_{\RR}\frac{\theta_\infty^\pm(x)}{|x^2\mp 1|_\infty^{1/2}}\rmd x\prod_{i=1}^{r}\widehat{\Phi_i}(\chi_i,-1/2)(1-q_i^{-1}) \frac{\zeta^S(0)}{\zeta^S(2)}X,
\end{align*}
where $\cdots$ denotes the integrand of the first term.

The first term is bounded by
\begin{align*}
&\int_{(\frac12+\varepsilon)}|\widetilde{G}(u)|\sum_{\chi^2=\triv}\sum_{\pm}\sum_{\nu\in \ZZ^r}q^{\nu/2}|\widetilde{F}(u-2)|(4q^\nu)^{\frac {\varepsilon}{2}-1}\int_{\RR}|\theta_\infty^\pm(x)||x^2\mp 1|_\infty^{\frac{\varepsilon}{2}-1}\rmd x\\
\times&\prod_{i=1}^{r}\widehat{\Phi_i}(\chi_i,\varepsilon/2-1) \left|\frac{\zeta^S(2u-2)\zeta^S(2u-1)}{\zeta^S(2u)}\right|\frac{X^{\frac12+\varepsilon}}{|u|}\rmd u. 
\end{align*}
which is 
\[
\ll \int_{(\frac12+\varepsilon)}|\widetilde{G}(u)|(1+|u|)^{-2}\rmd u X^{\frac12+\varepsilon}\ll \|G\|_{M_{1/2+\varepsilon}^{-2}}X^{\frac12+\varepsilon}\ll \|G\|_1X^{\frac12+\varepsilon}
\]
by \autoref{cor:mellinnorm} and that $\widetilde{F}$ is of rapid decay.

The second term vanishes by \autoref{lem:supercuspidaltrivialvanish}. By definition we have
\[
\widehat{\Phi_i}(\chi_i,-1/4)=\int_{\ZZ_{q_i}^\times}\int_{\QQ_{q_i}}\frac{\theta_{q_i}(y_i,\pm z_i q^\nu)}{|y_i^2\mp 4z_iq^\nu|_{q_i}'^{1/4}}\overline{\chi_i(z_i)}\rmd z_i\rmd y_i.
\]
Hence the third term above is precisely the main term in the theorem.
The last term vanishes since $r\geq 1$ and
\[
\zeta^S(0)=\prod_{i=1}^{r}\left(1-\frac{1}{q_i^0}\right)\zeta(0)=0.
\]
Hence we obtain the result.
\end{proof}

\begin{proposition}\label{prop:sigmaxi02}
Assume that \autoref{ass:nonarchimedean} holds. Then for any $\varepsilon>0$, we have 
\begin{align*}
\eqref{eq:sigmaxi02}=&2\widetilde{G}\legendresymbol{3}{2}\sum_{\chi^2=\triv}\sum_{\pm}\sum_{\nu\in \ZZ^r}q^{\nu/2}\widetilde{F}\legendresymbol{1}{2}\sqrt{2}q^{\nu/4} \int_{\RR}\frac{\theta_\infty^\pm(x)}{|x^2\mp 1|_\infty^{1/4}}\rmd x\\
\times&\prod_{i=1}^{r}(1-q_i^{-1})\int_{\ZZ_{q_i}^\times}\int_{\QQ_{q_i}}\frac{\theta_{q_i}(y_i,\pm z_i q^\nu)}{|y_i^2\mp 4z_iq^\nu|_{q_i}'^{1/4}}\overline{\chi_i(z_i)}\rmd z_i\rmd y_i \frac{\zeta^S(2)}{\zeta^S(3)}X^{\frac32}+\ms{A}X+O(X^\varepsilon).
\end{align*}
The implied constant depends only on $m,f_{q_i}$ and $\varepsilon$.
\end{proposition}
\begin{proof}
For any $\iota\in \{0,1\}$ and $\epsilon\in \{0,\pm 1\}^r$, we consider 
\begin{equation}\label{eq:sigmaxi02single}
\begin{split}
&G\legendresymbol{n}{X} 4n\sum_{\pm}\sum_{\nu\in \ZZ^r}q^{\nu/2}\sum_{k,f\in \ZZ_{(S)}^{>0}}\frac{1}{k^2f^3}\Kl_{k,f}^S(0,\pm n^2q^\nu)\int_{X_\iota}\int_{Y_\epsilon}\theta_\infty^\pm(x)\theta_{q}(y,\pm 4n^2q^\nu) \\
  \times&\frac{kf^2n^{-1}q^{-\nu/2}}{\sqrt{|x^2\mp 1|_\infty|y^2\mp 4n^2q^\nu|_q'}}V\legendresymbol{kf^2(4n^2q^\nu)^{-1/2}}{|x^2\mp 1|_\infty^{1/2}|y^2\mp 4n^2q^\nu|_q'^{1/2}}\rmd x\rmd y.
\end{split}
\end{equation}

By Mellin inversion formula on $G$ we have
\begin{align*}
\eqref{eq:sigmaxi02single}=&\frac{2}{\dpii} \int_{(4)}\widetilde{G}(u)\sum_{n\in \ZZ_{(S)}^{>0}}\frac{1}{n^u}\sum_{\pm}\sum_{\nu\in \ZZ^r}\sum_{k,f\in \ZZ_{(S)}^{>0}}\frac{1}{kf}\Kl_{k,f}^S(0,\pm n^2q^\nu)\\
  \times&\int_{X_\iota}\int_{Y_\epsilon}\frac{\theta_\infty^\pm(x)\theta_{q}(y,\pm n^2q^\nu)}{\sqrt{|x^2\mp 1|_\infty|y^2\mp 4n^2q^\nu|_q'}} V_{\iota,\epsilon}\legendresymbol{kf^2(4n^2q^\nu)^{-1/2}}{|x^2\mp 1|_\infty^{1/2}|y^2\mp 4nq^\nu|_q'^{1/2}}\rmd x\rmd yX^u\rmd u.
\end{align*}
Since $\Kl_{k,f}^S(0,\pm nq^\nu)\ll \sqrt{n}$, $\theta_\infty^\pm(x)$, $\theta_{q_i}(y,\pm nq^\nu)$ are compactly supported (independent of $n\in \ZZ_{(S)}^{>0}$), the integrand over $u$ converges absolutely for $\Re u>2$ and thus the right hand side of the above equation is well-defined. By definition \eqref{eq:defv} of $V_{\iota,\epsilon}$, the above equation equals
\begin{align*}
&\frac{2}{\dpii} \int_{(4)}\widetilde{G}(u)\sum_{\pm}\sum_{\nu\in \ZZ^r}\frac{1}{\dpii}\int_{(1)}\widetilde{F}(s)\uppi^{-s+\frac12}(4q^\nu)^{s/2} \frac{\Gamma(\frac{\iota+s}{2})}{\Gamma(\frac{\iota+1-s}{2})}\prod_{i=1}^{r}\frac{1-\epsilon_i q_i^{s-1}}{1-\epsilon_iq_i^{-s}} \int_{X_\iota}\theta_\infty^\pm(x)|x^2\mp 1|_\infty^{s/2-1/2}\rmd x\\
  \times& \sum_{n\in \ZZ_{(S)}^{>0}}\int_{Y_\epsilon}\theta_{q}(y,\pm n^2q^\nu)|y^2\mp 4n^2q^\nu|_q'^{s/2-1/2}\rmd y\frac{1}{n^{u-s}}\sum_{k,f\in \ZZ_{(S)}^{>0}}
  \frac{\Kl_{k,f}^S(0,\pm n^2q^\nu)}{k^{1+s}f^{1+2s}}\rmd s X^u\rmd u.
\end{align*}
The above sum for $f$ and $k$ converges absolutely for $\Re s>0$ and the sum for $n$ converges absolutely for $\Re (u-s)>2$ since $\Kl_{k,f}^S(0,\pm n^2q^\nu)\ll n$. Since $\widetilde{F}(s)$ has rapid decay vertically and the functions $\theta_\infty^\pm(x)$, $\theta_{q_i}(y,\pm n^2q^\nu)$ are compactly supported (independent of $n\in \ZZ_{(S)}^{>0}$), all the sums and integrals converge absolutely.
Hence we can change the order of the integrals and the sums. Now we define
\[
\Psi_i(z,s)=\int_{Y_\epsilon}\theta_{q}(y,\pm zq^\nu)|y^2\mp 4zq^\nu|_q'^{s}.
\]
Similarly, by \autoref{lem:integralsmooth}, we have the Fourier inversion formula
\[
\Psi_i(z,s)=(1-q_i^{-1})^{-1}\sum_{\chi_i}\widehat{\Psi_i}(\chi_i,s)\chi_i(z),
\] 
where 
\[
\widehat{\Psi_i}(\chi_i,s)=\int_{\ZZ_{q_i}^\times}\Psi_i(z,s)\overline{\chi_i(z)}\rmd z.
\]
$\chi_i$ runs over $\widehat{\ZZ_{q_i}^\times}$ and it is actually a finite sum. Hence for any $z\in \ZZ_{S_\fin}^\times$ we have
\[
\int_{Y_\epsilon}\theta_{q}(y,\pm zq^\nu)|y^2\mp 4zq^\nu|_q'^{s/2-1/2}\rmd y=\sum_{\chi}\chi(z)\prod_{i=1}^{r}(1-q_i^{-1})^{-1}\widehat{\Psi_i} (\chi_i,s/2-1/2),
\]
where $\chi$ runs over all characters on $\ZZ_{S_\fin}^\times$ and the sum over $\chi$ is actually finite. We also consider $\chi$ as Dirichlet characters.

Hence we have
\begin{align*}
\eqref{eq:sigmaxi02single}=&2\prod_{i=1}^{r}(1-q_i^{-1})^{-1}\frac{1}{\dpii} \int_{(4)}\widetilde{G}(u)\sum_{\pm}\sum_{\nu\in \ZZ^r}\frac{1}{\dpii}\int_{(1)}\widetilde{F}(s)\uppi^{-s+\frac12}(4q^\nu)^{\frac s2} \frac{\Gamma(\frac{\iota+s}{2})}{\Gamma(\frac{\iota+1-s}{2})}\prod_{i=1}^{r}\frac{1-\epsilon_i q_i^{s-1}}{1-\epsilon_iq_i^{-s}}\\
  \times& \int_{X_\iota}\theta_\infty^\pm(x)|x^2\mp 1|_\infty^{\frac s2-\frac12}\rmd x \sum_{\chi}\prod_{i=1}^{r}\widehat{\Psi_i}(\chi_i,s/2-1/2)\sum_{n\in \ZZ_{(S)}^{>0}}\frac{\chi(n)}{n^{u+s}}\sum_{k,f\in \ZZ_{(S)}^{>0}}\frac{\Kl_{k,f}^S(0,\pm 4n^2q^\nu)}{f^{1+2s}k^{1+s}}\rmd s X^u\rmd u.
\end{align*}
By \autoref{lem:kloostermandirichlet}, we obtain
\begin{align*}
\eqref{eq:sigmaxi02single}=&2\prod_{i=1}^{r}(1-q_i^{-1})^{-1}\frac{1}{\dpii}\frac{1}{\dpii}\int_{(4)} \widetilde{G}(u)\sum_{\pm}\sum_{\nu\in \ZZ^r}\int_{(1)}\widetilde{F}(s)\uppi^{-s+\frac12}(4q^\nu)^{\frac s2} \frac{\Gamma(\frac{\iota+s}{2})}{\Gamma(\frac{\iota+1-s}{2})}\\
  \times&\prod_{i=1}^{r}\frac{1-\epsilon_i q_i^{s-1}}{1-\epsilon_iq_i^{-s}} \int_{X_\iota}\theta_\infty^\pm(x)|x^2\mp 1|_\infty^{\frac s2-\frac12}\rmd x\sum_{\chi}\prod_{i=1}^{n}\widehat{\Psi_i}(\chi_i,s/2-1/2)\\
  \times&\frac{\zeta^S(2s)}{\zeta^S(s+1)} \frac{L^S(u-s,\chi^2)L^S(u+s,\chi^2)L^S(u,\chi^2)}{L^S(2u,\chi^4)} X^u\rmd u\rmd s.
\end{align*}

Now we move the $s$-contour from $(1)$ to $(8)$. The integrand is holomorphic on $\Re s>2$ and $\Re u>4$ except for the function $L^S(u-s,\chi^2)$ when $\chi^2=\triv$, which has a pole at $s=u-1$ with residue
\[
\res_{s=u-1}L^S\left(u-s,\triv\right)=\res_{s=u-1}\prod_{i=1}^{r}\left(1-\frac{1}{q_i^{u-s}}\right) \zeta(u-s)=-\prod_{i=1}^{r}\left(1-\frac{1}{q_i}\right).
\]
By residue formula, \eqref{eq:sigmaxi02single} becomes
\begin{align*}
&2\prod_{i=1}^{r}(1-q_i^{-1})^{-1}\frac{1}{\dpii}\frac{1}{\dpii}\int_{(4)}\widehat{G}(u)\sum_{\pm}\sum_{\nu\in \ZZ^r}\int_{(8)}\widetilde{F}(s)\uppi^{-s+\frac12}(4q^\nu)^{s/2} \frac{\Gamma(\frac{\iota+s}{2})}{\Gamma(\frac{\iota+1-s}{2})}\prod_{i=1}^{r}\frac{1-\epsilon_i q_i^{s-1}}{1-\epsilon_iq_i^{-s}} \\
  \times&\int_{X_\iota}\theta_\infty^\pm(x)|x^2\mp 1|_\infty^{\frac s2-\frac12}\rmd x\sum_{\chi}\prod_{i=1}^{n}\widehat{\Psi_i}(\chi_i, s/2-1/2)\frac{\zeta^S(2s)}{\zeta^S(s+1)} \frac{L^S(u-s,\chi^2)L^S(u+s,\chi^2)L^S(u,\chi^2)}{L^S(2u,\chi^4)}X^u\rmd s\rmd u\\ 
   +&\frac{2}{\dpii}\int_{(4)}\widehat{G}(u)\sum_{\pm}\sum_{\nu\in \ZZ^r}\sum_{\chi^2=\triv}\uppi^{  -u+\frac32}\widetilde{F}(u-1)(4q^\nu)^{\frac{u-1}{2}} \frac{\Gamma(\frac{\iota}{2}+\frac{u-1}{2})}{\Gamma(\frac{\iota+1}{2}-\frac{u-1}{2})} \prod_{i=1}^{r}\frac{1-\epsilon_i q_i^{u-2}}{1-\epsilon_iq_i^{-u+1}}\\
\times&\int_{X_\iota}\theta_\infty^\pm(x)|x^2\mp 1|_\infty^{\frac u2-1}\rmd x\prod_{i=1}^{n}\widehat{\Psi_i}(\triv,u/2-1)\frac{\zeta^S(2u-2)L^S(2u-1,\triv)L^S(u,\triv)}
{\zeta^S(u)L^S(2u,\triv)}X^u\rmd u.
\end{align*}
Note that since we move the contour from the left to the right, the residue formula has a minus sign.

We begin with considering the first term. Since $\widetilde{F}(s)$ has rapid decay vertically and the functions
\[
\frac{\Gamma(\frac{\iota+s}{2})}{\Gamma(\frac{\iota+1-s}{2})},\qquad\frac{1-\epsilon_i q_i^{s-1}}{1-\epsilon_iq_i^{-s}} 
\] 
are bounded when $\Re s=8$ is fixed (for the first one, use the Stirling formula), the sums and integrals converge absolutely. Also, the integrand is holomorphic on $0<\Re u\leq 4$ when $\Re s=8$ is fixed. Hence we can change the order of the integration over $s$ and $u$ and then move the $u$-contour to $(\frac12+\varepsilon)$ and then the $s$-contour to $(\varepsilon)$. Thus the first term is bounded by
\begin{align*}
&\int_{(\frac12+\varepsilon)}|\widetilde{G}(u)|\sum_{\pm}\sum_{\nu\in \ZZ^r}\int_{(\varepsilon)}|\widetilde{F}(s)|\uppi^{-\Re s+\frac12}(4q^\nu)^{\frac{\Re s}{2}} \left|\frac{\Gamma(\frac{\iota+s}{2})}{\Gamma(\frac{\iota+1-s}{2})}\right|\prod_{i=1}^{r} \left|\frac{1-\epsilon_i q_i^{s-1}}{1-\epsilon_iq_i^{-s}}\right|\int_{X_\iota}|\theta_\infty^\pm(x)||x^2\mp 1|_\infty^{\frac{\Re s}{2}-\frac12}\rmd x \\
  &\times\sum_{\chi}\prod_{i=1}^{n}|\widehat{\Psi_i}(\chi_i, s/2-1/2)|\left|\frac{\zeta^S(2s)}{\zeta^S(s+1)} \frac{L^S(u-s,\chi^2)L^S(u+s,\chi^2)L^S(u,\chi^2)}{L^S(2u,\chi^4)}\right|\rmd |s|\frac{|X^u|}{|u|}\rmd u,
\end{align*}
which is $\ll \|G\|_{3,1}X^{\frac12+\varepsilon}$ as in the previous case.

The second term equals
\begin{equation}\label{eq:secondterm2}
\begin{split}
&\frac{2}{\dpii}\int_{(4)}\widetilde{G}(u)\sum_{\chi^2=\triv}\sum_{\pm}\sum_{\nu\in \ZZ^r}\uppi^{  -u+\frac32}\widetilde{F}(u-1)(4q^\nu)^{\frac{u-1}{2}} \frac{\Gamma(\frac{\iota}{2}+\frac{u-1}{2})}{\Gamma(\frac{\iota+1}{2}-\frac{u-1}{2})} \prod_{i=1}^{r}\frac{(1-\epsilon_i q_i^{u-2})(1-q_i^{-2u+2})}{1-\epsilon_iq_i^{-u+1}}\\
\times&\int_{X_\iota}\theta_\infty^\pm(x)|x^2\mp 1|_\infty^{\frac u2-1}\rmd x\prod_{i=1}^{n}\widehat{\Psi_i}(\chi_i,u/2-1)\frac{\zeta(2u-2)\zeta^S(2u-1)}
{\zeta^S(2u)}X^u\rmd u.
\end{split}
\end{equation}

Now we move the $u$-contour from $(4)$ to $(\frac12+\varepsilon)$. The only poles in the region $0<\Re u\leq 4$ are $u=1$ and $u=3/2$. Hence \eqref{eq:secondterm2} equals
\begin{align*}
&\frac{2}{\dpii}\int_{(\frac12+\varepsilon)}\widetilde{G}(u)\sum_{\chi^2=\triv}\sum_{\pm}\sum_{\nu\in \ZZ^r}\uppi^{  -u+\frac32}\widetilde{F}(u-1)(4q^\nu)^{\frac{u-1}{2}} \frac{\Gamma(\frac{\iota}{2}+\frac{u-1}{2})}{\Gamma(\frac{\iota+1}{2}-\frac{u-1}{2})} \prod_{i=1}^{r}\frac{(1-\epsilon_i q_i^{u-2})(1-q_i^{-2u+2})}{1-\epsilon_iq_i^{-u+1}}\\
\times&\int_{X_\iota}\theta_\infty^\pm(x)|x^2\mp 1|_\infty^{\frac u2-1}\rmd x\prod_{i=1}^{n}\widehat{\Psi_i}(\chi_i,u/2-1)\frac{\zeta(2u-2)\zeta^S(2u-1)}
{\zeta^S(2u)}X^u\rmd u\\
+&\res_{u=3/2}\cdots
+\res_{u=1}\cdots,
\end{align*}
where $\cdots$ denotes the integrand of the first term.

The first term above is bounded by
\begin{align*}
&\int_{(\frac12+\varepsilon)}|\widetilde{G}(u)|\sum_{\chi^2=\triv}\sum_{\pm}\sum_{\nu\in \ZZ^r}\uppi^{-\varepsilon+\frac32}|\widetilde{F}(u-1)|(4q^\nu)^{\frac{-\varepsilon+1}{2}} \left|\frac{\Gamma(\frac{\iota}{2}+\frac{u-1}{2})}{\Gamma(\frac{\iota+1}{2}-\frac{u-1}{2})}\right| \prod_{i=1}^{r}\left|\frac{(1-\epsilon_i q_i^{u-2})(1-q_i^{-2u+2})}{1-\epsilon_iq_i^{-u+1}}\right|\\
\times&\left|\int_{X_\iota}\theta_\infty^\pm(x)|x^2\mp 1|_\infty^{\frac u2-1}\rmd x\right| \prod_{i=1}^{n}|\widehat{\Psi_i}(\chi_i,u/2-1)|\left|\frac{\zeta(2u-2)\zeta^S(2u-1)}
{\zeta^S(2u)}\right|\frac{|X^u|}{|u|}\rmd |u|\ll \|G\|_1X^{\frac12+\varepsilon}
\end{align*}
as in the previous case.

The second term becomes
\[
2\widetilde{G}\legendresymbol{3}{2}\sum_{\chi^2=\triv}\sum_{\pm}\sum_{\nu\in \ZZ^r}\widetilde{F}\legendresymbol{1}{2}(4q^\nu)^{1/4} \prod_{i=1}^{r}(1-q_i^{-1})\int_{X_\iota}\frac{\theta_\infty^\pm(x)}{|x^2\mp 1|_\infty^{1/4}}\rmd x\prod_{i=1}^{n}\widehat{\Psi_i}(\chi_i,-1/4)\frac{\zeta^S(2)}
{\zeta^S(3)}X^{\frac32}.
\]
By definition we have
\[
\widehat{\Psi_i}(\chi_i,-1/4)=\int_{\ZZ_{q_i}^\times}\int_{Y_\epsilon}\frac{\theta_{q}(y,\pm zq^\nu)}{|y^2\mp 4zq^\nu|_q'^{1/4}}\overline{\chi_i(z)}\rmd z.
\]
By summing over $\iota\in \{0,1\}$ and $\epsilon\in \{0,\pm 1\}^r$, we obtain the contribution
\begin{align*}
&2\widetilde{G}\legendresymbol{3}{2}\sum_{\iota\in \{0,1\}}\sum_{\epsilon\in \{0,\pm 1\}^r}\sum_{\chi^2=\triv}\sum_{\pm}\sum_{\nu\in \ZZ^r}\widetilde{F}\legendresymbol{1}{2}(4q^\nu)^{1/4} \prod_{i=1}^{r}(1-q_i^{-1})\\
\times&\int_{X_\iota}\frac{\theta_\infty^\pm(x)}{|x^2\mp 1|_\infty^{1/4}}\rmd x\prod_{i=1}^{n}\int_{\ZZ_{q_i}^\times}\int_{Y_\epsilon}\frac{\theta_{q}(y,\pm zq^\nu)}{|y^2\mp 4zq^\nu|_q'^{1/4}}\overline{\chi_i(z)}\rmd z\frac{\zeta^S(2)}
{\zeta^S(3)}X^{\frac32}\\
=&2\widetilde{G}\legendresymbol{3}{2}\sum_{\chi^2=\triv}\sum_{\pm}\sum_{\nu\in \ZZ^r}\widetilde{F}\legendresymbol{1}{2}\sqrt{2} q^{\nu/4} \int_{\RR}\frac{\theta_\infty^\pm(x)}{|x^2\mp 1|_\infty^{1/4}}\rmd x\\
\times&\prod_{i=1}^{r}(1-q_i^{-1})\int_{\ZZ_{q_i}^\times}\int_{\QQ_{q_i}}\frac{\theta_{q_i}(y_i,\pm z_i q^\nu)}{|y_i^2\mp 4z_iq^\nu|_{q_i}'^{1/4}}\overline{\chi_i(z_i)}\rmd z_i\rmd y_i \frac{\zeta^S(2)}{\zeta^S(3)}X^{\frac32}.
\end{align*}

Finally we deal with the last term, namely
\begin{equation}\label{eq:lasttermres1}
\begin{split}
&2\res_{u=1}\widetilde{G}(u)\sum_{\chi^2=\triv}\sum_{\pm}\sum_{\nu\in \ZZ^r}\uppi^{  -u+\frac32}\widetilde{F}(u-1)(4q^\nu)^{\frac{u-1}{2}} \frac{\Gamma(\frac{\iota}{2}+\frac{u-1}{2})}{\Gamma(\frac{\iota+1}{2}-\frac{u-1}{2})} \prod_{i=1}^{r}\frac{(1-\epsilon_i q_i^{u-2})(1-q_i^{-2u+2})}{1-\epsilon_iq_i^{-u+1}}\\
\times&\int_{X_\iota}\theta_\infty^\pm(x)|x^2\mp 1|_\infty^{\frac u2-1}\rmd x\prod_{i=1}^{n}\widehat{\Psi_i}(\chi_i,u/2-1)\frac{\zeta(2u-2)\zeta^S(2u-1)}
{\zeta^S(2u)}X^u.
\end{split}
\end{equation}
It is complicated in general since the function may have a \emph{triple} pole at $u=1$. However, under the additional hypothesis \autoref{ass:nonarchimedean}, it is much more simple.

We claim that if $\iota=0$ or $\epsilon_1=1$, then \eqref{eq:lasttermres1} vanishes. If $\iota=0$, then by \autoref{lem:hyperbolicvanish} (1), $\theta_\infty^\pm(x)$ is identically $0$ on $X_\iota$.
Hence the integral must be $0$. If $\epsilon_i=1$, then by \autoref{ass:nonarchimedean} for $i=1$ and \autoref{lem:hyperbolicvanish} (2) we have
\[
\int_{Y_1} \frac{\theta_{q_1}(y_1,\pm z_1q^\nu)}{|y_1^2\mp 4z_1q^\nu|_q'^{1/4}}\overline{\chi_1(z_1)}\rmd z_1\rmd y_1=0.
\]
Hence
\[
\widehat{\Psi_i}(\chi_i,-1/4)=\prod_{i=1}^{n}\int_{\ZZ_{q_i}^\times}\int_{Y_{\epsilon_i}} \frac{\theta_{q_i}(y_i,\pm z_iq^\nu)}{|y_i^2\mp 4z_iq^\nu|_q'^{1/4}}\overline{\chi_i(z_i)}\rmd z_i\rmd y_i=0.
\]
Thus $\eqref{eq:lasttermres1}=0$ in both cases.

Now we assume that $\iota\neq 0$ and $\epsilon_1\neq 1$. Then
\begin{enumerate}[itemsep=0pt,parsep=0pt,topsep=0pt,leftmargin=0pt,labelsep=2.5pt,itemindent=12pt,label=\textbullet]
  \item $\widetilde{F}(u-1)$ has a simple pole at $u=1$ with residue $1$ and is holomorphic otherwise.
  \item Since $\iota\neq 0$, the function 
  \[
  \frac{\Gamma(\frac{\iota}{2}+\frac{u-1}{2})}{\Gamma(\frac{\iota+1}{2}-\frac{u-1}{2})}
  \]
  is regular at $u=1$. 
  \item Since $\epsilon_1\neq 1$, the function 
  \[
  \frac{(1-\epsilon_i q_i^{u-2})(1-q_i^{-2u+2})}{1-\epsilon_iq_i^{-u+1}}
  \]
  is zero for $i=1$. For $i=2,\dots,n$ it is regular at $u=1$ and is zero unless $\epsilon_i\neq 1$.
  \item The function
  \[
  \frac{\zeta(2u-2)\zeta^S(2u-1)}{\zeta^S(2u)}
  \]
  has a simple pole at $u=1$ with residue
  \[
  \frac{1}{2}\frac{\zeta(0)}{\zeta^S(2)}=-\frac{1}{4\zeta^S(2)}.
  \]
\end{enumerate}
By the above analysis we know that  \eqref{eq:lasttermres1} is regular at $u=1$ unless $\epsilon_i=1$ for all $i=2,3,\dots,n$. In the exceptional case, the function has a simple pole at $u=1$ with residue
\begin{align*}
&2\sum_{\chi^2=\triv}\sum_{\pm}\sum_{\nu\in \ZZ^r}\res_{u=1}\widetilde{G}(u)\widetilde{F}(u-1)\res_{u=1}\frac{\zeta(2u-2)\zeta^S(2u-1)}
{\zeta^S(2u)} \left.\frac{\rmd}{\rmd u}\right|_{u=1}\frac{(1-\epsilon_1 q_1^{u-2})(1-q_1^{-2u+2})}{1-\epsilon_1q_1^{-u+1}}  \\
\times&\left[(4q^\nu)^{\frac{u-1}{2}} \frac{\Gamma(\frac{1}{2}+\frac{u-1}{2})}{\Gamma(1-\frac{u-1}{2})} \prod_{i=2}^{r}[(1-q_i^{u-2})(1+q_i^{-u+1})]\uppi^{-u+\frac32}\int_{X_1}\theta_\infty^\pm(x)|x^2\mp 1|_\infty^{\frac u2-1}\rmd x\prod_{i=1}^{n}\widehat{\Psi_i}(\chi_i,u/2-1)\frac{X^u}{u}\right]_{u=1}\\
=&-2\widetilde{G}(1)\sum_{\chi^2=\triv}\sum_{\pm}\sum_{\nu\in \ZZ^r}\frac{1}{4\zeta^S(2)}\left.\frac{\rmd}{\rmd u}\right|_{u=1}\frac{(1-\epsilon_1 q_1^{u-2})(1-q_1^{-2u+2})}{1-\epsilon_1q_1^{-u+1}} \frac{\Gamma(1/2)}{\Gamma(1)}2^{r-1}\prod_{i=2}^{r}(1-q_i^{-1})\\
\times& \uppi^{\frac12}\int_{X_1}\frac{\theta_\infty^\pm(x)}{|x^2\mp 1|_\infty^{1/2}}\rmd x \prod_{i=1}^{n}\int_{\ZZ_{q_i}^\times}\int_{Y_{\epsilon_i}} \frac{\theta_{q_i}(y_i,\pm z_iq^\nu)}{|y_i^2\mp 4z_iq^\nu|_q'^{1/2}}\overline{\chi_i(z_i)}\rmd z_i\rmd y_i X.
\end{align*}

We have
\[
\left.\frac{\rmd}{\rmd u}\right|_{u=1}\frac{(1-\epsilon_1 q_1^{u-2})(1-q_1^{-2u+2})}{1-\epsilon_1q_1^{-u+1}}=2\log q_1
\]
if $\epsilon_1=0$ and
\[
\left.\frac{\rmd}{\rmd u}\right|_{u=1}\frac{(1-\epsilon_1 q_1^{u-2})(1-q_1^{-2u+2})}{1-\epsilon_1q_1^{-u+1}}=(1+q_1^{-1})\log q_1
\]
if $\epsilon_1=-1$. Hence the residue becomes
\begin{align*}
&-2\widetilde{G}(1)\sum_{\chi^2=\triv}\sum_{\pm}\sum_{\nu\in \ZZ^r}\frac{2^{r}\uppi}{4\zeta^S(2)}\frac{\log q_1}{1-q_1^{-1}}\prod_{i=1}^{r}(1-q_i^{-1})\int_{X_1}\frac{\theta_\infty^\pm(x)}{|x^2\mp 1|_\infty^{1/2}}\rmd x \int_{\ZZ_{q}^\times}\int_{Y_{\epsilon}} \frac{\theta_{q}(y,\pm zq^\nu)}{|y^2\mp 4zq^\nu|_q'^{1/2}}\overline{\chi(z)}\rmd z\rmd y X,
\end{align*}
if $\epsilon_1=0$ and
\begin{align*}
&-2\widetilde{G}(1)\sum_{\chi^2=\triv}\sum_{\pm}\sum_{\nu\in \ZZ^r}\frac{2^{r-1}\uppi}{4\zeta^S(2)}\frac{(1-q_1^{-2})\log q_1}{1-q_1^{-1}}\prod_{i=1}^{r}(1-q_i^{-1})\\
\times&\int_{X_1}\frac{\theta_\infty^\pm(x)}{|x^2\mp 1|_\infty^{1/2}}\rmd x \int_{\ZZ_{q}^\times}\int_{Y_{\epsilon}} \frac{\theta_{q}(y,\pm zq^\nu)}{|y^2\mp 4zq^\nu|_q'^{1/2}}\overline{\chi(z)}\rmd z\rmd y X
\end{align*}
if $\epsilon_1=-1$.

Finally, we sum over $\iota\in \{1\}$, $\epsilon_1\in \{0,-1\}$ and $\epsilon_i\in \{1\}$ for $i\geq 2$, we obtain the total formula in this proposition.
\end{proof}

Combining \autoref{prop:sigmaxi01} and \autoref{prop:sigmaxi02}, and observing that the main term of \autoref{prop:sigmaxi01} cancels with the first term of \autoref{prop:sigmaxi02}, we obtain \autoref{thm:contributionxi0}.

\section{Analysis of the transformed Kloosterman sum}
In the following few sections we will treat the $\xi\neq 0$ case. We don't need \autoref{ass:nonarchimedean} for dealing with the $\xi\neq 0$ term.

In this section, we analyze the transformed Kloosterman sum $\prescript{2}{c}{\widehat{\Kl}}_{k,f}^S(\xi,\alpha)$ for $\xi$, $\alpha$, $c\in \ZZ^S$. Recall that
\begin{align*}
\prescript{2}{c}{\widehat{\Kl}}_{k,f}^S(\xi,\alpha)&=\sum_{m \bmod kf^2}\Kl_{k,f}^S(\xi,cm^2)\rme\legendresymbol{m\alpha}{kf^2} \rme_q\legendresymbol{m\alpha}{kf^2}\\
&=\sum_{\substack{a,b \bmod kf^2\\f^2\mid a^2-4cb^2}}\legendresymbol{(a^2-4cb^2)/f^2}{k}\rme\legendresymbol{a\xi}{kf^2} \rme_q\legendresymbol{a\xi}{kf^2}\rme\legendresymbol{b\alpha}{kf^2} \rme_q\legendresymbol{b\alpha}{kf^2}.
\end{align*}

For any prime $p\notin S$ and $c$, $\xi$, $\alpha\in \ZZ_p$, we define the \emph{local transformed Kloosterman sum} to be
\begin{equation}\label{eq:deflocalkloostermancharacter}
\prescript{2}{c}{\widehat{\Kl}}_{p^u,p^v}^{(p)}(\xi,\alpha)=\sum_{m\bmod p^{u+2v}}\Kl_{p^u,p^v}^{(p)}(\xi,cm^2)\rme_p\legendresymbol{-m\alpha}{p^{u+2v}}.
\end{equation}

As in \cite[Section 4]{cheng2025c}, for $\alpha\in \ZZ^S$ we have
\begin{equation}\label{eq:localgeneralizedkloostermaneq2}
\prescript{2}{c}{\widehat{\Kl}}_{p^u,p^v}^{(p)}(\xi,\alpha)=
\sum_{m \bmod p^{u+2v}}\Kl_{p^u,p^v}^{(p)}(\xi,cm^2) \rme\legendresymbol{m\alpha}{p^{u+2v}}\rme_{q}\legendresymbol{m\alpha}{p^{u+2v}}.
\end{equation}

To analyze the transformed Kloosterman sum, we first factor it into local data.
\begin{proposition}\label{prop:prodkloostermanlocal}
We have
\[
\prescript{2}{c}{\widehat{\Kl}}_{k,f}^S(\xi,\alpha)=\prod_{p\notin S}\prescript{2}{c}{\widehat{\Kl}}_{k_{(p)},f_{(p)}}^{(p)}(\xi,\alpha).
\]
\end{proposition}
\begin{proof}
Let
\[
kf^2=\prod_{j=1}^{t}(kf^2)_{(p_j)}
\]
be the prime factorization. By Chinese remainder theorem, we have an isomorphism
\[
\begin{split}
   \varphi: \prod_{j=1}^{t}\ZZ/(kf^2)_{(p_j)}&\to \ZZ/kf^2 \\
     (a_1,\dots,a_t) &\mapsto \sum_{j=1}^{t} a_j(kf^2)^{(p_j)}((kf^2)^{(p_j)})^{-1},
\end{split}
\]
where $(a^{(p)})^{-1}$ denotes the inverse of $a^{(p)}$ modulo $a_{(p)}$. Note that $\varphi(a_1,\dots,a_t)\equiv a_j\pmod{(kf^2)_{(p_j)}}$.

Since
\[
\frac{\varphi(a_1,\dots,a_t)\alpha}{kf^2}=\frac{\alpha}{kf^2}\sum_{j=1}^{t} a_j(kf^2)^{(p_j)}((kf^2)^{(p_j)})^{-1}=\sum_{j=1}^{t} \frac{a_j((kf^2)^{(p_j)})^{-1}\alpha}{(kf^2)_{(p_j)}},
\]
we obtain
\[
\rme\legendresymbol{\varphi(a_1,\dots,a_t)\alpha}{kf^2}=\prod_{j=1}^{t} \rme\legendresymbol{a_j((kf^2)^{(p_j)})^{-1}}{(kf^2)_{(p_j)}}.
\]
By Proposition B.1 in \cite{cheng2025b} we have
\begin{align*}
\Kl_{k,f}^S(\xi,c\varphi(a_1,\dots,a_t)^2 )&=\prod_{j=1}^{t}\Kl_{k_{(p_j)},f_{(p_j)}}^{(p^j)}(((kf^2)^{(p_j)})^{-1}\xi,c \varphi(a_1,\dots,a_t)^2)\\
&=\prod_{j=1}^{t}\Kl_{k_{(p_j)},f_{(p_j)}}^{(p^j)} (((kf^2)^{(p_j)})^{-1}\xi, ca_j^2).
\end{align*}
Hence by \eqref{eq:localgeneralizedkloostermaneq2},
\begin{align*}
  \prescript{2}{c}{\widehat{\Kl}}_{k,f}^S(\xi,\alpha)&=\sum_{\substack{a_j \bmod (kf^2)_{(p_j)}\\ 1\leq j\leq t}}\Kl_{k,f}^S(\xi,c\varphi(a_1,\dots,a_t)^2)\rme\legendresymbol{\varphi(a_1,\dots,a_t) \alpha}{kf^2} \rme_q\legendresymbol{\varphi(a_1,\dots,a_t) \alpha}{kf^2}\\
   & =\sum_{\substack{a_j \bmod (kf^2)_{(p_j)}\\ 1\leq j\leq t}} \prod_{j=1}^{t} \Kl_{k_{(p_j)},f_{(p_j)}}^{(p^j)}(((kf^2)^{(p_j)})^{-1}\xi, ca_j^2) \rme\legendresymbol{a_j((kf^2)^{(p_j)})^{-1}\alpha}{(kf^2)_{(p_j)}} \rme_q\legendresymbol{a_j((kf^2)^{(p_j)})^{-1}\alpha}{(kf^2)_{(p_j)}}\\
   & = \prod_{j=1}^{t}\prescript{2}{c}{\widehat{\Kl}}_{k_{(p_j)},f_{(p_j)}}^{(p)}\left(((kf^2)^{(p_j)})^{-1}\xi, ((kf^2)^{(p_j)})^{-1}\alpha\right).
\end{align*}

Now it suffices to prove that for any $\kappa\in \ZZ_p^\times$ we have $\prescript{2}{c}{\widehat{\Kl}}_{p^u,p^v}^{(p)}(\kappa\xi,\kappa\alpha)= \prescript{2}{c}{\widehat{\Kl}}_{p^u,p^v}^{(p)}(\xi,\alpha)$.
 By definition of the local generalized Kloosterman sum we have
\begin{equation}\label{eq:deflocalkloostermandoublecharacter}
\prescript{2}{c}{\widehat{\Kl}}_{p^u,p^v}^{(p)}(\xi,\alpha)=\sum_{a\bmod p^{u+2v}}\sum_{\substack{b\bmod p^{u+2v}\\ p^{2v}\mid a^2-4cb^2}}\legendresymbol{(a^2-4cb^2)/p^{2v}}{p^u}\rme_p\legendresymbol{-a\xi}{p^{u+2v}} \rme_p\legendresymbol{-b\alpha}{p^{u+2v}}.
\end{equation}
Hence by making change of variable $a\mapsto \kappa^{-1}a$ and $b\mapsto \kappa^{-1}b$ we obtain
\[
  \prescript{2}{c}{\widehat{\Kl}}_{p^u,p^v}^{(p)}(\kappa\xi,\kappa\alpha)=\sum_{\substack{a,b\bmod p^{u+2v}\\ p^{2v}\mid \kappa^{-2}(a^2-4cb^2)}}\legendresymbol{\kappa^{-2}(a^2-4cb^2)/p^{2v}}{p^u}\rme_p\legendresymbol{-a\xi}{p^{u+2v}} \rme_p\legendresymbol{-b\alpha}{p^{u+2v}}= \prescript{2}{c}{\widehat{\Kl}}_{p^u,p^v}^{(p)}(\xi,\alpha)
\]
and thus the conclusion follows.
\end{proof}

We will mainly analyze the local transformed Kloosterman sum \eqref{eq:deflocalkloostermancharacter} in the remaining part of this section.
We will consider $\legendresymbol{c}{p}=1$ and $\legendresymbol{c}{p}=-1$ separately. The stories are quite different.

\subsection{Case $\legendresymbol{c}{p}=-1$}
In this subsection we assume that $\legendresymbol{c}{p}=-1$. First we present a well-known lemma which will be used in the computation below.
\begin{lemma}
Suppose that $p$ is an odd prime and $a\in \FF_p^\times$. Then we have
\[
\sum_{x\in \FF_p}\legendresymbol{x^2-a}{p}=-1.
\]
\end{lemma}
\begin{proof}
See, for example, \cite[Lemma 2 of Appendix A]{langlands2004}.
\end{proof}

We first consider the case $u=1$ and $v=0$.
\begin{proposition}\label{prop:fouriertransformfinitefield}
Let $p$ be an odd prime. Suppose that $d\in \FF_p$ such that $\legendresymbol{d}{p}=-1$ and $\xi,\eta\in \ZZ_p$. Then
\begin{equation}\label{eq:quadraticextensionfinitefield}
\sum_{x,y\in \FF_p}\legendresymbol{x^2-dy^2}{p}\rme\legendresymbol{x\xi+y\eta}{p} =p\legendresymbol{\eta^2-d\xi^2}{p}.
\end{equation}
\end{proposition}
\begin{proof}
Since $\legendresymbol{d}{p}=-1$ we have $\FF_{p^2}=\FF_p(\sqrt d)$. Moreover, $1$ and $\sqrt{d}$ form a basis of $\FF_{p^2}$ over $\FF_p$. Suppose that $\alpha=a+b\sqrt{d}\in\FF_{p^2}$. Then we have
\[
\norm(\alpha)=a^2-db^2\qquad\text{and}\qquad \Tr(\alpha)=2a,
\]
where the trace and the norm are taken in  $\FF_{p^2}$ over $\FF_p$. Therefore
\[
\Tr((x+\sqrt{d}y)(\xi+\eta/\sqrt{d}))=2(x\xi+y\eta)
\]
and hence
\[
\sum_{x,y\in \FF_p}\legendresymbol{x^2-dy^2}{p}\rme\legendresymbol{x\xi+dy\eta}{p} =\sum_{\alpha\in \FF_{p^2}}\legendresymbol{\norm(\alpha)}{p}\rme\legendresymbol{\Tr(\alpha\beta)/2}{p},
\]
where $\beta=\xi+\eta/\sqrt{d}$. 

If $\beta=0$, i.e., $\xi=\eta=0$, then both sides in \eqref{eq:quadraticextensionfinitefield} are zero. Now we assume that $\beta\neq 0$.

We make change of variable $\alpha\mapsto\alpha\beta^{-1}$ and then the above becomes
\[
\sum_{\alpha\in \FF_{p^2}}\legendresymbol{\norm(\alpha\beta^{-1})}{p}\rme\legendresymbol{\Tr(\alpha)/2}{p}= \legendresymbol{\norm(\beta)}{p}\sum_{\alpha\in \FF_{p^2}}\legendresymbol{\norm(\alpha)}{p}\rme\legendresymbol{\Tr(\alpha)/2}{p}.
\]

Finally we compute the so-called \emph{Gauss sum}
\[
\sum_{\alpha\in \FF_{p^2}}\legendresymbol{\norm(\alpha)}{p}\rme\legendresymbol{\Tr(\alpha)/2}{p} =\sum_{x,y\in \FF_p}\legendresymbol{x^2-dy^2}{p}\rme\legendresymbol{x}{p}.
\]
The contribution when $x=0$ is
\[
\sum_{y\in \FF_p}\legendresymbol{-dy^2}{p}=(p-1)\legendresymbol{-d}{p}.
\]
For the contribution when $x\neq 0$, we use the above lemma and obtain
\[
\sum_{y\in \FF_p}\legendresymbol{x^2-dy^2}{p}\rme\legendresymbol{x}{p}=\legendresymbol{-d}{p} \rme\legendresymbol{x}{p}\sum_{y\in \FF_p}\legendresymbol{y^2+(-d)^{-1}x^2}{p}=-\legendresymbol{-d}{p} \rme\legendresymbol{x}{p}.
\]
Therefore
\[
\sum_{x,y\in \FF_p}\legendresymbol{x^2-dy^2}{p}\rme\legendresymbol{x}{p}=(p-1)\legendresymbol{-d}{p} -\legendresymbol{-d}{p} \sum_{x\in \FF_p^\times}\rme\legendresymbol{x}{p}=p\legendresymbol{-d}{p}.
\]
Since $\norm(\beta)=\xi^2-\eta^2/d$, we obtain 
\[
\legendresymbol{\norm(\beta)}{p}\sum_{\alpha\in \FF_{p^2}}\legendresymbol{\norm(\alpha)}{p}\rme\legendresymbol{\Tr(\alpha)/2}{p}= p\legendresymbol{-d}{p}\legendresymbol{\xi^2-\eta^2/d}{p}=p\legendresymbol{\eta^2-d\xi^2}{p}.
\]
Hence the conclusion follows.
\end{proof}
\begin{corollary}\label{cor:fouriertransformfinitefield}
We have
\[
\prescript{2}{c}{\widehat{\Kl}}_{p,1}^{(p)}(\xi,\alpha)=p\legendresymbol{\alpha^2-4c\xi^2}{p}.
\]
\end{corollary}
\begin{proof}
By definition we have
\[
\prescript{2}{c}{\widehat{\Kl}}_{p,1}^{(p)}(\xi,\alpha)=\sum_{a,b\bmod p}\legendresymbol{a^2-4cb^2}{p}\rme_p\legendresymbol{-a\xi}{p} \rme_p\legendresymbol{-b\alpha}{p}=\sum_{a,b\in \FF_p}\legendresymbol{a^2-4cb^2}{p}\rme\legendresymbol{a\xi+b\eta}{p}.
\]
Since $\legendresymbol{4c}{p}=\legendresymbol{c}{p}=-1$, by
\autoref{prop:fouriertransformfinitefield} we obtain the result.
\end{proof}

Now we consider the general case.
\begin{theorem}\label{thm:localkloostermaninert}
 For $p\notin S$ and $\legendresymbol{c}{p}=-1$, the value of the transformed Kloosterman sum $\prescript{2}{c}{\widehat{\Kl}}_{p^u,p^v}^{(p)}(\xi,\alpha)$ is given by the following:
\begin{enumerate}[itemsep=0pt,parsep=0pt,topsep=0pt, leftmargin=0pt,labelsep=2.5pt,itemindent=15pt,label=\upshape{(\arabic*)}]
  \item Suppose that $u$ is odd. Then $\prescript{2}{c}{\widehat{\Kl}}_{p^u,p^v}^{(p)}(\xi,\alpha)=0$ if $v_p(\xi)<u+v-1$ or $v_p(\alpha)<u+v-1$, and
      \[
      \prescript{2}{c}{\widehat{\Kl}}_{p^u,p^v}^{(p)}(\xi,\alpha)=p^{2u+2v-1} \legendresymbol{(\alpha/p^{u+v-1})^2-4c(\xi/p^{u+v-1})^2}{p}
      \]
      otherwise.
\item Suppose that $u$ is even. Then for $u=0$ we have
      \[
      \prescript{2}{c}{\widehat{\Kl}}_{p^u,p^v}^{(p)}(\xi,\alpha)=\begin{cases}
                                                                   p^{2v}, & v_p(\xi)\geq u+v\ \text{and}\ v_p(\alpha)\geq u+v, \\
                                                                   0, & \text{otherwise},
                                                                 \end{cases}
      \]
      and for $u\geq 2$ we have
      \[
      \prescript{2}{c}{\widehat{\Kl}}_{p^u,p^v}^{(p)}(\xi,\alpha)=\begin{cases}
                                                                   p^{2u+2v}-p^{2u+2v-2}, & v_p(\xi)\geq u+v\ \text{and}\ v_p(\alpha)\geq u+v, \\
                                                                   -p^{2u+2v-2},& v_{p}(\xi)=u+v-1,\ v_p(\alpha)\geq u+v-1,\\
                                                                   &\text{or}\ v_{p}(\alpha)=u+v-1,\ v_p(\xi)\geq u+v-1,\\
                                                                   0, & \text{otherwise}.
                                                                 \end{cases}
      \]
\end{enumerate}
\end{theorem}
\begin{proof}
Recall that
\[
\prescript{2}{c}{\widehat{\Kl}}_{p^u,p^v}^{(p)}(\xi,\alpha)=\sum_{\substack{a,b\bmod p^{u+2v}\\ p^{2v}\mid a^2-4cb^2}}\legendresymbol{(a^2-4cb^2)/p^{2v}}{p^u}\rme_p\legendresymbol{-a\xi}{p^{u+2v}} \rme_p\legendresymbol{-b\alpha}{p^{u+2v}}.
\]
Since $\legendresymbol{c}{p}=-1$, we know that $ p^{2v}\mid a^2-4cb^2$ if and only if $p^v\mid a$ and $p^v\mid b$. Hence by making change of variable $a\mapsto p^{-v}a$ and $b\mapsto p^{-v}b$ we obtain
\[
\prescript{2}{c}{\widehat{\Kl}}_{p^u,p^v}^{(p)}(\xi,\alpha)=\sum_{a,b\bmod p^{u+v}}\legendresymbol{a^2-4cb^2}{p^u} \rme_p\legendresymbol{-a\xi}{p^{u+v}} \rme_p\legendresymbol{-b\alpha}{p^{u+v}}.
\]

(1) For $u$ odd we have
\[
\prescript{2}{c}{\widehat{\Kl}}_{p^u,p^v}^{(p)}(\xi,\alpha)=\sum_{a,b\bmod p^{u+v}}\legendresymbol{a^2-4cb^2}{p} \rme_p\legendresymbol{-a\xi}{p^{u+v}} \rme_p\legendresymbol{-b\alpha}{p^{u+v}}.
\]
If $v_p(\xi)<u+v-1$, we can find $a'\in p\ZZ_p$ such that $\rme_p(-a'\xi/p^{u+v})\neq 1$. By making change of variable $a\mapsto a+a'$ we obtain
\[
\prescript{2}{c}{\widehat{\Kl}}_{p^u,p^v}^{(p)}(\xi,\alpha)=\sum_{a,b\bmod p^{u+v}}\legendresymbol{(a+a')^2-4cb^2}{p} \rme_p\legendresymbol{-(a+a')\xi}{p^{u+v}} \rme_p\legendresymbol{-b\alpha}{p^{u+v}}=\rme_p\legendresymbol{-a'\xi}{p^{u+v}} \prescript{2}{c}{\widehat{\Kl}}_{p^u,p^v}^{(p)}(\xi,\alpha).
\]
Hence we must have $\prescript{2}{c}{\widehat{\Kl}}_{p^u,p^v}^{(p)}(\xi,\alpha)=0$. Similarly, if $v_p(\alpha)<u+v-1$, we also have $\prescript{2}{c}{\widehat{\Kl}}_{p^u,p^v}^{(p)}(\xi,\alpha)=0$.

Now we assume that $v_p(\xi)\geq u+v-1$ and $v_p(\alpha)\geq u+v-1$. In this case, we can use \autoref{cor:fouriertransformfinitefield} and obtain
\begin{align*}
\prescript{2}{c}{\widehat{\Kl}}_{p^u,p^v}^{(p)}(\xi,\alpha)&=\sum_{a,b\bmod p^{u+v}}\legendresymbol{a^2-4cb^2}{p} \rme_p\legendresymbol{-a\xi}{p^{u+v}} \rme_p\legendresymbol{-b\alpha}{p^{u+v}}=p^{2(u+v-1)} \prescript{2}{c}{\widehat{\Kl}}_{p,1}^{(p)}(\xi/p^{u+v-1},\alpha/p^{u+v-1})\\
&=p^{2u+2v-1} \legendresymbol{(\alpha/p^{u+v-1})^2-4c(\xi/p^{u+v-1})^2}{p}.
\end{align*}

(2) First we assume that $u=0$. In this case $\legendresymbol{x}{p^u}=1$. Hence
\[
\prescript{2}{c}{\widehat{\Kl}}_{1,p^v}^{(p)}(\xi,\alpha)=\sum_{a,b\bmod p^{u+v}} \rme_p\legendresymbol{-a\xi}{p^{u+v}} \rme_p\legendresymbol{-b\alpha}{p^{u+v}},
\]
which is $p^{2(u+v)}$ if $v_p(\xi),v_p(\eta)\geq u+v$ and is $0$ otherwise.

Now we assume that $u\geq 2$. In this case $\legendresymbol{x}{p^u}=1$ if $p\nmid x$ and $\legendresymbol{x}{p^u}=0$ if $p\mid x$. Hence we obtain
\begin{align*}
\prescript{2}{c}{\widehat{\Kl}}_{p^u,p^v}^{(p)}(\xi,\alpha)&=\sum_{a,b\bmod p^{u+v}} \rme_p\legendresymbol{-a\xi}{p^{u+v}} \rme_p\legendresymbol{-b\alpha}{p^{u+v}}-\sum_{\substack{a,b\bmod p^{u+v}\\ p\mid a^2-4cb^2}} \rme_p\legendresymbol{-a\xi}{p^{u+v}} \rme_p\legendresymbol{-b\alpha}{p^{u+v}}\\
&=\sum_{a,b\bmod p^{u+v}} \rme_p\legendresymbol{-a\xi}{p^{u+v}} \rme_p\legendresymbol{-b\alpha}{p^{u+v}}-\sum_{\substack{a,b\bmod p^{u+v}\\ p\mid a,p\mid b}} \rme_p\legendresymbol{-a\xi}{p^{u+v}} \rme_p\legendresymbol{-b\alpha}{p^{u+v}}\\
&=\sum_{a,b\bmod p^{u+v}} \rme_p\legendresymbol{-a\xi}{p^{u+v}} \rme_p\legendresymbol{-b\alpha}{p^{u+v}}-\sum_{a,b\bmod p^{u+v-1}} \rme_p\legendresymbol{-a\xi}{p^{u+v-1}} \rme_p\legendresymbol{-b\alpha}{p^{u+v-1}},
\end{align*}
The first term above is $p^{2(u+v)}$ if $v_p(\xi),v_p(\eta)\geq u+v$ and is $0$ otherwise, and the second term is $-p^{2(u+v-2)}$ if $v_p(\xi),v_p(\eta)\geq u+v-1$ and is $0$ otherwise. Hence we obtain the desired result.
\end{proof}

\subsection{Case $\legendresymbol{c}{p}=1$}
In this subsection we assume that $\legendresymbol{c}{p}=1$. We fix an element $\sqrt{c}$ that is a square root of $c$ in $\ZZ_p$ (which exists by Hensel's lemma). The key point in this case is the following linear transformation.

\begin{proposition}\label{prop:kloostermanlocalsplit}
We have
\[
\prescript{2}{c}{\widehat{\Kl}}_{p^u,p^v}^{(p)}(\xi,\alpha)=\sum_{\substack{a,b\bmod p^{u+2v}\\ p^{2v}\mid ab}}\legendresymbol{(ab)/p^{2v}}{p^u}\rme_p\legendresymbol{-a\xi'}{p^{u+2v}} \rme_p\legendresymbol{-b\alpha'}{p^{u+2v}},
\]
where
\[
\xi'=\frac{\xi}{2}+\frac{\alpha}{4\sqrt{c}}\qquad\text{and}\qquad \alpha'=\frac{\xi}{2}-\frac{\alpha}{4\sqrt{c}}.
\]
\end{proposition}
\begin{proof}
We have $a^2-4cb^2=(a+2\sqrt{c}b)(a-2\sqrt{c}b)$. By making change of variable
\[
a\mapsto\frac{a+b}{2}\qquad\text{and}\qquad b\mapsto\frac{a-b}{4\sqrt{c}},
\]
we have $a+2\sqrt{c}\mapsto a$ and $a-2\sqrt{c}\mapsto b$. Hence by \eqref{eq:deflocalkloostermandoublecharacter} we conclude that
\begin{align*}
\prescript{2}{c}{\widehat{\Kl}}_{p^u,p^v}^{(p)}(\xi,\alpha)&=\sum_{\substack{a,b\bmod p^{u+2v}\\ p^{2v}\mid ab}}\legendresymbol{(ab)/p^{2v}}{p^u}\rme_p\legendresymbol{-\frac{a+b}{2}\xi}{p^{u+2v}} \rme_p\legendresymbol{-\frac{a-b}{4\sqrt{c}}\alpha}{p^{u+2v}}\\
&=\sum_{\substack{a,b\bmod p^{u+2v}\\ p^{2v}\mid ab}}\legendresymbol{(ab)/p^{2v}}{p^u}\rme_p\legendresymbol{-a\xi'}{p^{u+2v}} \rme_p\legendresymbol{-b\alpha'}{p^{u+2v}}.\qedhere
\end{align*}
\end{proof}

Now we compute the sum in this case. First we assume that $u$ is odd.
\begin{lemma}\label{lem:gausssumsplit}
Let $p$ be an odd prime, $\xi\in \ZZ_p$ and $n\in \ZZ_{>0}$. Then we have
\[
\sum_{a\bmod p^n}\legendresymbol{a}{p}\rme_p\legendresymbol{-a\xi}{p^{n}}= p^{n-1} \legendresymbol{\xi/p^{n-1}}{p}\times \begin{cases}
                                              \sqrt{p}, & p\equiv 1\pmod 4,\\
                                              \rmi\sqrt{p}, & p\equiv 3\pmod 4 
                                            \end{cases}
\]
if $v_p(\xi)=n-1$, and is $0$ otherwise. 
\end{lemma}
\begin{proof}
If $v_p(\xi)\geq n$, then 
\[
\sum_{a\bmod p^n}\legendresymbol{a}{p}\rme_p\legendresymbol{-a\xi}{p^{n}}=\sum_{a\bmod p^n}\legendresymbol{a}{p}=0.
\]
If $v_p(\xi)\leq n-2$, then $\rme_p(p\xi/p^n)\neq 1$. Therefore from
\[
\sum_{a\bmod p^n}\legendresymbol{a}{p}\rme_p\legendresymbol{-a\xi}{p^{n}}=\sum_{a\bmod p^n}\legendresymbol{a+p}{p}\rme_p\legendresymbol{-a\xi-p\xi}{p^{n}}= \rme_p\legendresymbol{-p\xi}{p^{n}}\sum_{a\bmod p^n}\legendresymbol{a}{p}\rme_p\legendresymbol{-a\xi}{p^{n}}
\]
we conclude that 
\[
\sum_{a\bmod p^n}\legendresymbol{a}{p}\rme_p\legendresymbol{-a\xi}{p^{n}}=0.
\]

From now on, we assume that $v_p(\xi)=n-1$ and we assume that $\xi=\xi_0p^{n-1}$ so that $\xi_0\in \ZZ_p^\times$. In this case we have
\[
\sum_{a\bmod p^n}\legendresymbol{a}{p}\rme_p\legendresymbol{-a\xi}{p^{n}}= \sum_{a\bmod p^n}\legendresymbol{a}{p}\rme_p\legendresymbol{-a\xi_0}{p}=p^{n-1}\sum_{a\bmod p}\legendresymbol{a}{p}\rme_p\legendresymbol{-a\xi_0}{p}.
\]
The sum 
\[
\mf{g}=\sum_{a\bmod p}\legendresymbol{a}{p}\rme_p\legendresymbol{-a\xi_0}{p}
\]
is the Gauss sum and it is well-known (see \cite{shafarevich1966} for example) that
\[
\mf{g}=\legendresymbol{\xi/p^{n-1}}{p}\times \begin{cases}
                                              \sqrt{p}, & p\equiv 1\pmod 4,\\
                                              \rmi\sqrt{p}, & p\equiv 3\pmod 4 .
                                            \end{cases}
\]
Hence we obtain the desired conclusion.
\end{proof}

\begin{proposition}\label{prop:splitoddcase}
For $p\notin S$, $\legendresymbol{c}{p}=1$ and $u$ odd, we have
\[
\prescript{2}{c}{\widehat{\Kl}}_{p^u,p^v}^{(p)}(\xi,\alpha)=p^{2u+2v-1} \legendresymbol{(-\xi'\alpha')/p^{2u+2v-2}}{p}
\]
if $v_p(\xi'\alpha')=2u+2v-2$ and $v_p(\xi'),v_p(\alpha')\geq u-1$, and is $0$ otherwise.
\end{proposition} 
\begin{proof}
By \autoref{prop:kloostermanlocalsplit} and since $u$ is odd, we have
\begin{align*}
\prescript{2}{c}{\widehat{\Kl}}_{p^u,p^v}^{(p)}(\xi,\alpha)&=\sum_{\substack{a,b\bmod p^{u+2v}\\ p^{2v}\parallel ab}}\legendresymbol{(ab)/p^{2v}}{p}\rme_p\legendresymbol{-a\xi'}{p^{u+2v}} \rme_p\legendresymbol{-b\alpha'}{p^{u+2v}}\\
&=\sum_{w=0}^{2v}\sum_{\substack{a\bmod p^{u+2v}\\ p^{w}\parallel a}} \sum_{\substack{b\bmod p^{u+2v}\\ p^{2v-w}\parallel b}}\legendresymbol{a/p^{w}}{p}\legendresymbol{b/p^{2v-w}}{p}\rme_p\legendresymbol{-a\xi'}{p^{u+2v}} \rme_p\legendresymbol{-b\alpha'}{p^{u+2v}}\\
&=\sum_{w=0}^{2v}\sum_{\substack{a\bmod p^{u+2v-w}\\ p\nmid a}}\legendresymbol{a}{p}\rme_p\legendresymbol{-a\xi'}{p^{u+2v-w}}  \sum_{\substack{b\bmod p^{u+w}\\ p\nmid b}}\legendresymbol{b}{p} \rme_p\legendresymbol{-b\alpha'}{p^{u+w}}
\end{align*}
Clearly the assumptions $p\nmid a$ and $p\nmid b$ in the sum are redundant. Hence by \autoref{lem:gausssumsplit} we obtain
\[
\prescript{2}{c}{\widehat{\Kl}}_{p^u,p^v}^{(p)}(\xi,\alpha)=p^{2u+2v-1}\sum_{\substack{0\leq w\leq 2v\\ v_p(\xi')=u+2v-w-1,\,v_p(\alpha')=u+w-1}}\legendresymbol{\xi'/p^{u+2v-w-1}}{p} \legendresymbol{\alpha'/p^{u+w-1}}{p}\legendresymbol{-1}{p}.
\]
It is easy to see that there exists $w\in [0,2v]\cap \ZZ$ such that $v_p(\xi')=u+2v-w-1$ and $v_p(\alpha')=u+w-1$ if and only if $v_p(\xi')\geq u-1$, $v_p(\alpha')\geq u-1$, and $v_p(\xi'\alpha')=v_p(\xi')+v_p(\alpha')=2u+2v-2$. Hence we obtain the desired formula.
\end{proof}

Next we consider the case when $u=0$. In this case we have
\[
\prescript{2}{c}{\widehat{\Kl}}_{1,p^v}^{(p)}(\xi,\alpha)=\sum_{\substack{a,b\bmod p^{2v}\\ p^{2v}\mid ab}}\rme_p\legendresymbol{-a\xi'}{p^{2v}} \rme_p\legendresymbol{-b\alpha'}{p^{2v}}.
\]
The sum can be split as
\[
\sum_{\substack{a,b\bmod p^{2v}\\ p^{2v}\mid ab}}\rme_p\legendresymbol{-a\xi'}{p^{2v}} \rme_p\legendresymbol{-b\alpha'}{p^{2v}}=  \sum_{w=0}^{2v}\sum_{\substack{a\bmod p^{2v}\\ p^{w}\parallel a}}\rme_p\legendresymbol{-a\xi'}{p^{2v}} \sum_{\substack{b\bmod p^{2v}\\ p^{2v-w}\mid b}} \rme_p\legendresymbol{-b\alpha'}{p^{2v}}.
\]
For simplicity we introduce the following notations. Let $\xi\in \ZZ_p$ and $r,t\in \ZZ_{\geq 0}$, we define 
\begin{equation}\label{eq:defsrt}
S_{r,t}^{(p)}(\xi)=\begin{dcases}
                        \sum_{\substack{a\bmod p^t\\ p^r\mid a}}\rme_p\legendresymbol{-a\xi}{p^t}, & r\leq t, \\
                        0, & \text{otherwise}.
                      \end{dcases}
\end{equation}
Also, we define
\begin{equation}\label{eq:defdeltasrt}
\Delta S_{r,t}^{(p)}(\xi)=S_{r,t}^{(p)}(\xi)-S_{r+1,t}^{(p)}(\xi).
\end{equation}
It is easy to see that
\[
\Delta S_{r,t}^{(p)}(\xi)=\begin{dcases}
                        \sum_{\substack{a\bmod p^t\\ p^r\parallel a}}\rme_p\legendresymbol{-a\xi}{p^t}, & r\leq t, \\
                        0, & \text{otherwise}.
                      \end{dcases}
\]
Hence
\begin{equation}\label{eq:kloostermansumsrt}
\prescript{2}{c}{\widehat{\Kl}}_{1,p^v}^{(p)}(\xi,\alpha)= \sum_{w=0}^{2v}\Delta S_{w,2v}^{(p)}(\xi')S_{2v-w,2v}^{(p)}(\alpha').
\end{equation}
Hence we only need to compute $S_{r,t}^{(p)}(\xi)$.

\begin{proposition}\label{prop:srtexplicit}
Let $\xi\in \ZZ_p$. Then we have
\begin{equation}\label{eq:srtexplicit}
S_{r,t}^{(p)}(\xi)=\begin{dcases}
                        p^{t-r}, & v_p(\xi)\geq t-r, \\
                        0, & \text{otherwise}.
                      \end{dcases}
\end{equation}
\end{proposition}
\begin{proof}
Clearly we may assume that $r\leq t$. In this case we have
\[
S_{r,t}^{(p)}(\xi)=\sum_{\substack{a\bmod p^t\\ p^r\mid a}}\rme_p\legendresymbol{-a\xi}{p^t}=\sum_{a\bmod p^{t-r}}\rme_p\legendresymbol{-a\xi}{p^{t-r}},
\]
which is $p^{t-r}$ if $v_p(\xi)\geq t-r$ and is $0$ otherwise.
\end{proof}
By \autoref{prop:srtexplicit} and \eqref{eq:defdeltasrt}, \eqref{eq:kloostermansumsrt} we obtain a formula for $\prescript{2}{c}{\widehat{\Kl}}_{p^u,p^v}^{(p)}(\xi,\alpha)$ for $u=0$.

Next, we consider the case that $u\geq 2$ and is even. In this case we have
\begin{align*}
\prescript{2}{c}{\widehat{\Kl}}_{p^u,p^v}^{(p)}(\xi,\alpha)&=\sum_{\substack{a,b\bmod p^{u+2v}\\ p^{u+2v}\mid ab}}\legendresymbol{(ab)/p^{2v}}{p^u}\rme_p\legendresymbol{-a\xi'}{p^{u+2v}} \rme_p\legendresymbol{-b\alpha'}{p^{u+2v}}=\sum_{\substack{a,b\bmod p^{u+2v}\\ p^{u+2v}\parallel ab}}\rme_p\legendresymbol{-a\xi'}{p^{u+2v}} \rme_p\legendresymbol{-b\alpha'}{p^{u+2v}}\\
&=\sum_{w=0}^{2v}\sum_{\substack{a\bmod p^{u+2v}\\ p^{w}\parallel a}}\rme_p\legendresymbol{-a\xi'}{p^{u+2v}} \sum_{\substack{b\bmod p^{u+2v}\\ p^{2v-w}\parallel b}} \rme_p\legendresymbol{-b\alpha'}{p^{u+2v}}
\end{align*}
Hence
\begin{equation}\label{eq:kloostermansumsrt2}
\prescript{2}{c}{\widehat{\Kl}}_{p^u,p^v}^{(p)}(\xi,\alpha) \sum_{w=0}^{2v}\Delta S_{w,u+2v}^{(p)}(\xi')\Delta S_{2v-w,u+2v}^{(p)}(\alpha').
\end{equation}

To summarize, we obtain the following theorem.
\begin{theorem}\label{thm:localkloostermansplit}
For $p\notin S$ and $\legendresymbol{c}{p}=1$, we let $\sqrt{c}$ to be a square root of $c$ in $\ZZ_p$ and denote
\[
\xi'=\frac{\xi}{2}+\frac{\alpha}{4\sqrt{c}}\qquad\text{and}\qquad \alpha'=\frac{\xi}{2}-\frac{\alpha}{4\sqrt{c}}.
\]
Then the value of the transformed Kloosterman sum $\prescript{2}{c}{\widehat{\Kl}}_{p^u,p^v}^{(p)}(\xi,\alpha)$ is given by the following:
\begin{enumerate}[itemsep=0pt,parsep=0pt,topsep=0pt, leftmargin=0pt,labelsep=2.5pt,itemindent=15pt,label=\upshape{(\arabic*)}]
  \item Suppose that $u$ is odd. Then 
  \[
  \prescript{2}{c}{\widehat{\Kl}}_{p^u,p^v}^{(p)}(\xi,\alpha)=p^{2u+2v-1} \legendresymbol{(-\xi'\alpha')/p^{2u+2v-2}}{p}
  \]
  if $v_p(\xi'\alpha')=2u+2v-2$ and $v_p(\xi'),v_p(\alpha')\geq u-1$, and is $0$ otherwise.
\item Suppose that $u$ is even. Then for $u=0$ we have
      \[
      \prescript{2}{c}{\widehat{\Kl}}_{p^u,p^v}^{(p)}(\xi,\alpha)=\sum_{w=0}^{2v}\Delta S_{w,2v}^{(p)}(\xi')S_{2v-w,2v}^{(p)}(\alpha')
      \]
      and for $u\geq 2$ we have
      \[
      \prescript{2}{c}{\widehat{\Kl}}_{p^u,p^v}^{(p)}(\xi,\alpha)=\sum_{w=0}^{2v}\Delta S_{w,u+2v}^{(p)}(\xi')\Delta S_{2v-w,u+2v}^{(p)}(\alpha'),
      \]
\end{enumerate}
where $S_{r,t}^{(p)}(\xi)$ is given in \eqref{eq:srtexplicit} and $\Delta S_{r,t}^{(p)}(\xi)=S_{r,t}^{(p)}(\xi)-S_{r+1,t}^{(p)}(\xi)$.
\end{theorem}
\begin{proof}
The conclusion follows from \autoref{prop:splitoddcase}, \eqref{eq:kloostermansumsrt} and \eqref{eq:kloostermansumsrt2}.
\end{proof}

\section{Analysis of the Kloosterman-type series}
In this section we will analyze the following \emph{Kloosterman-type series}:
\begin{equation}\label{eq:kloostermantypeseries}
\prescript{2}{c}D_{\xi,\alpha}^S(s,\chi):=\sum_{k,f\in \ZZ_{(S)}^{>0}}\frac{\prescript{2}{c}{\widehat{\Kl}}_{k,f}^{S}(\xi,\alpha)\chi(kf^2)}{k^{1+s}f^{1+2s}}
\end{equation}
for fixed $\xi,\alpha\in \ZZ^{S}$, where $s$ is a complex variable and $\chi$ is a fixed Dirichlet character which is unramified outside $S$.

\begin{proposition}\label{prop:splitkloostermanseries}
We have
\[
\prescript{2}{c}D_{\xi,\alpha}^S(s,\chi)=\prod_{p\notin S}\prescript{2}{c}D_{\xi,\alpha}^{(p)}(s,\chi),
\]
where
\begin{equation}\label{eq:kloostermantypeserieslocal}
\prescript{2}{c}D_{\xi,\alpha}^{(p)}(s,\chi)=\sum_{u,v=0}^{+\infty} \frac{\prescript{2}{c}{\widehat{\Kl}}_{p^u,p^v}^{(p)}(\xi,\alpha)\chi(p^{u+2v})}{(p^u) ^{1+s}(p^v)^{1+2s}}.
\end{equation}
\end{proposition}
\begin{proof}
This follows from \autoref{prop:prodkloostermanlocal} and the multiplicativity of $\chi$. 
\end{proof}
By the above proposition it suffices to consider the \emph{local Kloosterman-type series} $\prescript{2}{c}D_{\xi,\alpha}^{(p)}(s,\chi)$.  

\subsection{Local computation}

As in the previous section, we consider the case $\legendresymbol{c}{p}=\pm 1$ separately. 
\subsubsection{Case $\legendresymbol{c}{p}=-1$}
In this subsubsection we assume that $\legendresymbol{c}{p}=-1$. Recall that the value of the transformed Kloosterman sum $\prescript{2}{c}{\widehat{\Kl}}_{p^u,p^v}^{(p)}(\xi,\alpha)$ is given in \autoref{thm:localkloostermaninert}.  We will consider cases depending on whether $\xi=0$ and $\alpha=0$.

We first consider the case $\xi\neq 0$ or $\alpha\neq 0$.
\begin{proposition}\label{prop:finitesuminert}
If $\xi\neq 0$ or $\alpha\neq 0$, then $\prescript{2}{c}D_{\xi,\alpha}^{(p)}(s,\chi)$ is a finite sum and defines an entire function for $s$.
\end{proposition}

\begin{proof}
By \autoref{thm:localkloostermaninert} we find that $\prescript{2}{c}{\widehat{\Kl}}_{p^u,p^v}^{(p)}(\xi,\alpha)\neq 0$ only if $v_p(\xi)\geq u+v-1$ and $v_p(\alpha)\geq u+v-1$.

By symmetry we may assume that $\xi\neq 0$. Under such assumption $v_p(\xi)$ is finite. Hence there are finitely many choices of $u$ and $v$ satisfying $v_p(\xi)\geq u+v-1$ and hence $\prescript{2}{c}{\widehat{\Kl}}_{p^u,p^v}^{(p)}(\xi,\alpha)\neq 0$ for finitely many $u$ and $v$. Therefore the series \eqref{eq:kloostermantypeserieslocal} is a finite sum. Since each term in \eqref{eq:kloostermantypeserieslocal} is an entire function, we conclude that $\prescript{2}{c}D_{\xi,\alpha}^{(p)}(s,\chi)$ is an entire function for $s$.
\end{proof}

Now we compute $\prescript{2}{c}D_{\xi,\alpha}^{(p)}(s,\chi)$ for "general" $p$, that is, the primes $p$ satisfying $v_p(\xi)=0$ if $\xi\neq 0$ and $v_p(\alpha)=0$ if $\alpha\neq 0$.

\begin{proposition}\label{prop:nonzeroinert}
Suppose that $\xi\neq 0$ or $\alpha\neq 0$. If the prime $p\notin S$ satisfies $p\nmid \delta$, then
\[
\prescript{2}{c}D_{\xi,\alpha}^{(p)}(s,\chi)=1+\frac{\chi(p)}{p^{s}}\legendresymbol{\delta}{p},
\]
where $\delta=\alpha^2-4c \xi^2$.
\end{proposition}
\begin{proof}
Since $\legendresymbol{c}{p}=-1$, $p\nmid \delta$ is equivalent to $p\nmid \xi$ or $p\nmid \alpha$.

If $u$ is odd, then by \autoref{thm:localkloostermaninert} and the assumption that $v_p(\xi)=0$ or $v_p(\alpha)=0$ we must have $v=0$. In this situation we have
\[
\prescript{2}{c}{\widehat{\Kl}}_{p^u,p^v}^{(p)}(\xi,\alpha)=p \legendresymbol{\alpha^2-4c \xi^2}{p}.
\]
Hence we get the contribution $\chi(p)\legendresymbol{\delta}{p}p^{-s}$.

If $u$ is even, then by \autoref{thm:localkloostermaninert} and that $v_p(\xi)=0$ or $v_p(\alpha)=0$ we know that $v=0$. In this case, $\prescript{2}{c}{\widehat{\Kl}}_{p^u,p^v}^{(p)}(\xi,\alpha)=1$ and hence we get the contribution $1$.

By combining the two cases above we obtain the desired result. 
\end{proof}

Finally, we consider the case $\xi=\alpha=0$.
\begin{proposition}\label{prop:zeroinert}
If $\xi=\alpha=0$, then
\[
\prescript{2}{c}D_{\xi,\alpha}^{(p)}(s,\chi)=\frac{1-\chi^2(p)p^{-2s}}{(1-\chi^2(p)p^{1-2s}) (1-\chi^2(p)p^{2-2s})}
\]
\end{proposition}
\begin{proof}
By \autoref{thm:localkloostermaninert} we obtain the following result: If $u$ is odd, then    $\prescript{2}{c}{\widehat{\Kl}}_{p^u,p^v}^{(p)}(0,0)=0$. If $u$ is even, then
\[
\prescript{2}{c}{\widehat{\Kl}}_{p^u,p^v}^{(p)}(0,0)=\begin{cases}
                                                      p^{2v}, & u=0 \\
                                                      p^{2u+2v}(1-p^{-2}), & u\geq 2.
                                                    \end{cases}
\]
Hence
\[
\prescript{2}{c}D_{\xi,\alpha}^{(p)}(s,\chi)=\sum_{v=0}^{+\infty}\frac{p^{2v}\chi(p^{2v})} {(p^v)^{1+2s}}+\sum_{\substack{u=2\\ u\,\text{even}}}^{+\infty}\sum_{v=0}^{+\infty}\frac{(1-p^{-2})p^{2u+2v}\chi(p^{u+2v})} {(p^u)^{1+s}(p^v)^{1+2s}}.
\]

By direct computation we have
\[
\sum_{v=0}^{+\infty}\frac{p^{2v}\chi(p^{2v})} {(p^v)^{1+2s}}=\frac{1}{1-\chi^2(p)p^{1-2s}}
\]\
and
\[
\sum_{\substack{u=2\\ u\,\text{even}}}^{+\infty}\sum_{v=0}^{+\infty}\frac{(1-p^{-2})p^{2u+2v}\chi(p^{u+2v})} {(p^u)^{1+s}(p^v)^{1+2s}}=(1-p^{-2})\frac{\chi^2(p)}{p^{-2+2s}}\frac{1}{1-\chi^2(p)} \frac{1}{1-\chi^2(p)p^{1-2s}}.
\]
Hence
\[
\prescript{2}{c}D_{\xi,\alpha}^{(p)}(s,\chi)=\frac{1}{1-\chi^2(p)p^{1-2s}}\left[1+ \frac{(1-p^{-2})\chi^2(p)}{p^{-2+2s}-\chi^2(p)}\right]=\frac{1-\chi^2(p)p^{-2s}}{(1-\chi^2(p)p^{1-2s}) (1-\chi^2(p)p^{2-2s})}.\qedhere
\]
\end{proof}

\subsubsection{Case $\legendresymbol{c}{p}=1$}
In this subsubsection we assume that $\legendresymbol{c}{p}=1$. The value of the transformed Kloosterman sum $\prescript{2}{c}{\widehat{\Kl}}_{p^u,p^v}^{(p)}(\xi,\alpha)$ is now given in \autoref{thm:localkloostermansplit}.  We will consider cases depending on whether $\xi'=0$ and $\alpha'=0$.

First we consider the case when $\xi'\neq 0$ and $\alpha'\neq 0$.
\begin{lemma}\label{lem:srtzero}
If $v_p(\xi)<t-r$, then $S_{r,t}(\xi)=0$. If $v_p(\xi)<t-r-1$, then $\Delta S_{r,t}(\xi)=0$. 
\end{lemma}
\begin{proof}
It follows directly from \autoref{prop:srtexplicit} and \eqref{eq:defdeltasrt}.
\end{proof}

\begin{proposition}\label{prop:finitesumsplit}
If $\xi'\neq 0$ and $\alpha'\neq 0$, then $\prescript{2}{c}D_{\xi,\alpha}^{(p)}(s,\chi)$ is a finite sum and defines an entire function for $s$.
\end{proposition}
\begin{proof}
We first consider the case that $u$ is odd. By \autoref{thm:localkloostermansplit}, $\prescript{2}{c}{\widehat{\Kl}}_{p^u,p^v}^{(p)}(\xi,\alpha)=0$ unless $v_p(\alpha')\geq u-1$. Since $v_p(\alpha')$ is a finite number, the choice of $u$ is finite. Finally, $v$ is determined  since $v_p(\xi'\alpha')=2u+2v-1$. 

Now we consider the case that $u$ is even. By \autoref{lem:srtzero} and \autoref{thm:localkloostermansplit} we conclude that $\prescript{2}{c}{\widehat{\Kl}}_{p^u,p^v}^{(p)}(\xi,\alpha)$ does not vanish only if there exists $w\in\{0,1,\dots,2v\}$ such that $u+w-1\leq v_p(\xi')$ and $u+2v-w-1\leq v_p(\alpha')$. In particular, $u+2v-2\leq v_p(\xi')+v_p(\alpha')$, which is a finite number. Hence the choice of $u$ and $v$ is finite and thus $\prescript{2}{c}D_{\xi,\alpha}^{(p)}(s,\chi)$ is a finite sum. Since each summand is entire, we know that $\prescript{2}{c}D_{\xi,\alpha}^{(p)}(s,\chi)$ is entire as well.
\end{proof}

\begin{proposition}\label{prop:bothnonzerosplit}
Suppose that $\xi'\neq 0$ and $\alpha'\neq 0$. If the prime $p\notin S$ satisfies $p\nmid\delta$, then
\[
\prescript{2}{c}D_{\xi,\alpha}^{(p)}(s,\chi)=1+\frac{\chi(p)}{p^{s}}\legendresymbol{\delta}{p},
\]
where $\delta=\alpha^2-4c \xi^2$.
\end{proposition}
\begin{proof}
Suppose that $\prescript{2}{c}{\widehat{\Kl}}_{p^u,p^v}^{(p)}(\xi,\alpha)$ is nonzero. 

If $u$ is odd, then by \autoref{thm:localkloostermansplit}, we have $0=v_p(\delta)=v_p(\xi'\alpha')=2u+2v-2$. Hence we must have $u=1$ and $v=0$. In this case, we have
\[
\frac{\widehat{\Kl}_{p^u,p^v}^{(p)}(\xi,\alpha)\chi(p^{u+2v})}{(p^{u})^{1+s}(p^{v}) ^{1+2s}}=\frac{\chi(p)}{p^{1+s}}p\legendresymbol{-\xi'\alpha'}{p}=\legendresymbol{\alpha^2/(16c)-\xi^2/4}{p}\frac{1}{p^s}.
\]
Since $\legendresymbol{c}{p}=1$, we have
\[
\legendresymbol{\alpha^2/(16c)-\xi^2/4}{p}=\legendresymbol{\alpha^2-4c\xi^2}{p}= \legendresymbol{\delta}{p}.
\]
Hence we get the contribution $\frac{\chi(p)}{p^{s}}\legendresymbol{\delta}{p}$.

If $u$ is even, then by \autoref{thm:localkloostermansplit} and \autoref{lem:srtzero} we must have $u+2v-w\leq 1$ and $w=0$. Therefore, $u=v=0$. In this case, we have
\[
\frac{{}_c^2\widehat{\Kl}_{p^u,p^v}^{(p)}(\xi,\alpha)\chi(p^{u+2v})}{(p^{u})^{1+s}(p^{v}) ^{1+2s}}=\Delta S_{0,0}^{(p)}(\xi')S_{0,0}^{(p)}(\alpha')=1.
\]
The conclusion now follows from adding the two cases together.
\end{proof}

Next we consider the case that exactly one of $\xi'$ and $\alpha'$ is zero.
\begin{proposition}\label{prop:onezerosplit}
If $\xi'\neq 0$, $\alpha'=0$ or $\xi'=0$, $\alpha'\neq 0$, then $\prescript{2}{c}D_{\xi,\alpha}^{(p)}(s,\chi)$ can be written as
\[
\frac{_cA^{(p)}_{\xi,\alpha}(s)}{1-\chi^2(p)p^{1-2s}},
\]
where $A^{(p)}_{\xi,\alpha}(s)$ is an entire function. Moreover, if $v_p(\alpha')=0$, then 
\[
A_{\xi,\alpha}^{(p)}(s)=1-\chi^2(p)p^{-2s}.
\]
\end{proposition}
\begin{proof}
By symmetry we may assume that $\xi'=0$ and $\alpha'\neq 0$.
 
If $u$ is odd, then by \autoref{thm:localkloostermansplit} we have ${}_c^2\widehat{\Kl}_{p^u,p^v}^{(p)}(\xi,\alpha)=0$ since $\xi'\alpha'=0$. Now we assume that $u$ is even.

\underline{\emph{Case 1:}}\ \ $u=0$. 

We have
\begin{equation}\label{eq:kloostermanuzero}
{}_c^2\widehat{\Kl}_{1,p^v}^{(p)}(\xi,\alpha)=\sum_{w=0}^{2v}=\sum_{w=0}^{2v}\Delta S_{w,2v}^{(p)}(0)S_{2v-w,2v}^{(p)}(\alpha')=\sum_{w=0}^{\min\{2v,v_p(\alpha')\}}\bm\phi(p^{2v-w})p^w
\end{equation}
where $\bm \phi$ is the Euler totient function. Hence for $v$ sufficiently large (say $v>v_0$) we have
\[
{}_c^2\widehat{\Kl}_{1,p^v}(\xi,\alpha)=\sum_{w=0}^{v_p(\alpha')}\bm \phi(p^{2v-w})p^w =(v_p(\alpha')+1)p^{2v}(1-p^{-1}).
\]
Therefore
\[
\sum_{v=v_0+1}^{+\infty}\frac{{}_c^2\widehat{\Kl}_{1,p^v}(\xi,\alpha) \chi(p^{2v})}{(p^{2v})^ {1+2s}}=(v_p(\alpha')+1)(1-p^{-1})\sum_{v=v_0+1}^{+\infty}\frac{\chi(p^{2v})}{(p^{2v})^{-1+2s}}= \frac{(v_p(\alpha')+1)(1-p^{-1})\chi^2(p^{v_0+1})}{(p^{2v_0+2})^{-1+2s}(1-\chi^2(p)p^{-2s+1})}.
\]
Moreover,
\[
\sum_{v=0}^{v_0}\frac{{}_c^2\widehat{\Kl}^{(p)}_{1,p^v}(\xi,\alpha)\chi(p^{2v})}{(p^{2v})^ {1+2s}}
\]
is a finite sum and hence it defines an entire function.

\underline{\emph{Case 2:}}\ \ $u\geq 2$. 

We have
\[
{}_c^2\widehat{\Kl}_{p^u,p^v}^{(p)}(\xi,\alpha)=\sum_{w=0}^{2v}\Delta S_{w,u+2v}^{(p)}(0)\Delta S_{2v-w,u+2v}^{(p)}(\alpha')=\sum_{w=0}^{2v}\bm\phi(p^{u+2v-w})\Delta S_{2v-w,u+2v}^{(p)}(\alpha').
\]
By \autoref{prop:srtexplicit} and \eqref{eq:defdeltasrt} we have
\[
\Delta S_{2v-w,u+2v}^{(p)}(\alpha')=\begin{cases}
                                         p^{u+w}(1-p^{-1}), & u+w\leq v_p(\alpha') , \\
                                        -p^{u+w-1} & u+w-1=v_p(\alpha') ,\\
                                         0, & \text{otherwise}.
                                       \end{cases}
\]
Hence we find that the sum over $u$ is finite. 

Now we fix $u$. For $v$ sufficiently large (say $v>v_u$) we have
\begin{align*}
{}_c^2\widehat{\Kl}_{p^u,p^v}^{(p)}(\xi,\alpha)&=\sum_{w=0}^{v_p(\alpha')-u}\bm \phi(p^{u+2v-w}) \bm\phi(p^{u+w})-\bm\phi(p^{2u+2v-v_p(\alpha')+1})p^{v_p(\alpha')}\\
&=(v_p(\alpha')-u+1)(1-p^{-1})^2p^{2u+2v}-(1-p^{-1})p^{2u+2v+1}.
\end{align*}
Therefore 
\[
\sum_{v=v_u+1}^{+\infty}\frac{{}_c^2\widehat{\Kl}^{(p)}_{p^u,p^v}(\xi,\alpha) \chi(p^{2v})}{(p^{2v})^{1+2s}}=(v_p(\alpha')+1)(1-p^{-1})\sum_{v=v_u+1}^{+\infty} \frac{\chi(p^{2v})}{(p^{2v})^{-1+2s}}= \frac{(v_p(\alpha')+1)(1-p^{-1})\chi^2(p^{v_u+1})}{(p^{2v_u+2})^{-1+2s}(1-\chi^2(p)p^{-2s+1})}.
\]
Finally,
\[
\sum_{v=0}^{v_u}\frac{{}_c^2\widehat{\Kl}^{(p)}_{p^u,p^v}(\xi,\alpha)\chi(p^{u+2v})} {(p^u)^{1+s}(p^{2v})^ {1+2s}}
\]
is a finite sum and hence it defines an entire function.
The first assertion is now proved by adding the cases together.

Now we prove the second assertion. If $v_p(\alpha')=0$, then ${}_c^2\widehat{\Kl}_{p^u,p^v}^{(p)}(\xi,\alpha)=0$ if $u$ is odd or $u\geq 2$, by the above analysis. Moreover, by \eqref{eq:kloostermanuzero}, we have
\[
{}_c^2\widehat{\Kl}_{1,p^v}^{(p)}(\xi,\alpha)= \sum_{w=0}^{\min\{2v,v_p(\alpha')\}}\bm\phi(p^{2v-w})p^w=\bm\phi(p^{2v}).
\]
Therefore,
\[
\prescript{2}{c}D_{\xi,\alpha}^{(p)}(s,\chi)=\sum_{v=0}^{+\infty}\frac{\bm\phi(p^{2v})\chi(p^{2v})} {(p^{2v})^{1+2s}}=\frac{1-\chi^2(p)p^{-2s}}{1-\chi^2(p)p^{1-2s}}.
\]
Hence $A_{\xi,\alpha}^{(p)}(s)=1-\chi^2(p)p^{-2s}$.
\end{proof}

Finally we consider case when $\xi'=\alpha'=0$, which is equivalent to $\xi=\alpha=0$.
\begin{proposition}\label{prop:zerosplit}
If $\xi=\alpha=0$, then
\[
\prescript{2}{c}D_{\xi,\alpha}^{(p)}(s,\chi)=\frac{1-\chi^2(p)p^{-2s}}{(1-\chi^2(p)p^{1-2s}) (1-\chi^2(p)p^{2-2s})}.
\]
\end{proposition}
\begin{proof}
If $u$ is odd, then by \autoref{thm:localkloostermansplit} we must have ${}_c^2\widehat{\Kl}_{p^u,p^v}^{(p)}(\xi,\alpha)=0$ as in the previous proposition. It now suffices to consider the case when $u$ is even.

If $u=0$, then we have
\[
{}_c^2\widehat{\Kl}_{1,p^v}^{(p)}(\xi,\alpha)=\sum_{w=0}^{2v}\Delta S_{w,2v}^{(p)}(0)S_{2v-w,2v}^{(p)}(0)=\sum_{w=0}^{2v}p^{2v}(1-p^{-1})+p^{2v}= (2v(1-p^{-1})+1)p^{2v}.
\]
Hence
\begin{align*}
\sum_{v=0}^{+\infty}\frac{{}_c^2\widehat{\Kl}_{1,p^v}^{(p)}(\xi,\alpha)\chi(p^v)} {(p^v)^{1+2s}}&= \sum_{v=0}^{+\infty}\frac{2v(1-p^{-1})\chi(p)^v}{(p^v)^{-1+2s}}+\sum_{v=0}^{+\infty} \frac{\chi(p)^v}{(p^v)^{-1+2s}}\\
&=(1-p^{-1})\frac{2\chi^2(p)p^{-1+2s}} {(\chi^2(p)-p^{2s-1})^2}+ \frac{1}{1-\chi^2(p)p^{1-2s}}.
\end{align*}
If $u\geq 2$, we have
\[
{}_c^2\widehat{\Kl}_{1,p^v}^{(p)}(\xi,\alpha)=\sum_{w=0}^{2v}\Delta S_{w,u+2v}^{(p)}(0)\Delta S_{2v-w,u+2v}^{(p)}(0)=\sum_{w=0}^{2v}(1-p^{-1})^2p^{2u+2v}=(2v+1)(1-p^{-1})^2p^{2u+2v}.
\]
Hence
\begin{align*}
&\sum_{\substack{u=2\\u\text{ even}}}^{+\infty}\sum_{v=0}^{+\infty}\frac{{}_c^2\widehat{\Kl}_{p^u,p^v}^{(p)}(\xi,\alpha)\chi(p^{u+2v})} {(p^u)^{1+s}(p^v)^{1+2s}}= \sum_{\substack{u=2\\u\text{ even}}}^{+\infty}\sum_{v=0}^{+\infty}\frac{(1-p^{-1})^2(2v+1)\chi(p)^{u+2v}} {(p^u)^{-1+s}(p^v)^{-1+2s}} \\
&=(1-p^{-1})^2\frac{\chi^2(p)}{p^{-2+2s}}\frac{1}{1-\chi^2(p)p^{2-2s}} \left(\frac{2\chi^2(p)p^{-1+2s}}{(\chi^2(p)-p^{2s-1})^2}+\frac{1}{1-\chi^2(p)p^{1-2s}}\right).
\end{align*}
The conclusion holds by adding the two cases together, and then do direct computation.
\end{proof}

Finally we give an estimate of the local Kloosterman-type series.

\subsection{Global computations}
In this subsection we analyze the global Kloosterman-type series 
\[
\prescript{2}{c}D^S_{\xi,\alpha}(s,\chi)=\prod_{p\notin S}\prescript{2}{c}D_{\xi,\alpha}^{(p)}(s,\chi).
\]

For any $0\neq d\in \ZZ^S$, we define
\begin{equation}\label{eq:defsxi}
S_d=S\cup\{p\notin S\,|\, v_p(d)\neq 0\}.
\end{equation}

Let $\delta=\alpha^2-4c\xi^2$. We first assume that $\delta\neq 0$. In this case we also have $\xi',\alpha'\neq 0$ if $\legendresymbol{c}{p}=1$. Hence by \autoref{prop:nonzeroinert} and \autoref{prop:bothnonzerosplit} we have
\begin{align*}
{}^2_cD^S_{\xi,\alpha}(s,\chi)&=\prod_{\substack{p\notin S\\p\mid\delta}}\prescript{2}{c}D_{\xi,\alpha}^{(p)}(s,\chi)\times\prod_{\substack{p\notin S\\p\nmid\delta}} \left(1+\frac{\chi(p)}{p^{s}}\legendresymbol{\delta}{p}\right)
=\prod_{\substack{p\notin S\\p\mid\delta}}\prescript{2}{c}D_{\xi,\alpha}^{(p)}(s,\chi)\prod_{p\notin S_\delta}\frac{1-\frac{\chi^2(p)}{p^{2s}}}{1-\frac{\chi(p)}{p^{s}}\legendresymbol {\delta}{p}}\\
&=\prod_{\substack{p\notin S\\p\mid\delta}}\prescript{2}{c}D_{\xi,\alpha}^{(p)}(s,\chi) \frac{L^{S_\delta}(s,\chi \legendresymbol{\delta}{\cdot})}{L^{S_\delta}(2s,\chi^2)}.
\end{align*}

By \autoref{prop:finitesuminert}, \autoref{prop:finitesumsplit} and the property of $L$-functions,
the function is holomorphic for $\Re s>1/2$ except for a possible pole at $s=1$, which occurs if and only if $\chi=\legendresymbol{\delta}{\cdot}$.

Now let's assume that $\delta=0$. If $c$ is not a square, then we must have $\xi=\alpha=0$. Hence by \autoref{prop:zeroinert} and \autoref{prop:zerosplit} we obtain
\[
{}^2_cD^S_{\xi,\alpha}(s,\chi)=\prod_{p\notin S}\frac{1-\chi^2(p)p^{-2s}}{(1-\chi^2(p)p^{1-2s}) (1-\chi^2(p)p^{2-2s})}=\frac{L^S(2s-1,\chi^2)L^S(2s-2,\chi^2)}{L^S(2s,\chi^2)}.
\]
If $c$ is a square, then $\legendresymbol{c}{p}=1$ for all $p\notin S$ and
\begin{equation}\label{eq:defxiprime}
\xi'=\frac{\xi}{2}+\frac{\alpha}{4\sqrt{c}}\qquad\text{and}\qquad \alpha'=\frac{\xi}{2}-\frac{\alpha}{4\sqrt{c}}.
\end{equation}
both exist in $\ZZ_p$. In this case we have $\delta=-16\xi'\alpha'=0$. Hence $\xi'=0$ or $\alpha'=0$. If $\xi'=\alpha'=0$, then $\xi=\alpha=0$ and the computation is the same as above. Now by symmetry we may assume that $\xi'=0$ and $\alpha'\neq 0$. By \autoref{prop:onezerosplit} we then have
\begin{align*}
\prescript{2}{c}D^S_{\xi,\alpha}(s,\chi)&=\prod_{p\in S_{\alpha'}-S}\frac{\prescript{}{c}{A}_{\xi,\alpha}^{(p)}(s,\chi)}{1-\chi^2(p)p^{1-2s}}\times \prod_{p\notin S_{\alpha'}}\frac{1-\chi^2(p)p^{-2s}}{1-\chi^2(p)p^{1-2s}}=\prod_{p\in S_{\alpha'}-S}\prescript{}{c}{A}_{\xi,\alpha}^{(p)}(s,\chi)\frac{L^{S}(2s-1,\chi^2)}{L^{S_{\alpha'}}(2s,\chi^2)}.
\end{align*}

In sum, we have the following theorem:
\begin{theorem}\label{thm:kloostermantypeseries}
Let $\xi,\alpha,c\in \ZZ^S$ and $p\notin S$ be a prime. We define
$S_\xi$ as in \eqref{eq:defsxi}. Let $\delta=\alpha^2-4c\xi^2$. Let $\xi'$ and $\alpha'$ be as in \eqref{eq:defxiprime} if $c$ is a square.
Then the value of the series ${}_c^2D_{\xi,\alpha}^S(s,\chi)$ is given by the following:
\begin{enumerate}[itemsep=0pt,parsep=0pt,topsep=0pt, leftmargin=0pt,labelsep=2.5pt,itemindent=15pt,label=\upshape{(\arabic*)}]
  \item\label{item:610firstterm} Suppose that $\delta\neq 0$. Then 
  \[
  \prescript{2}{c}D_{\xi,\alpha}^{S}(s,\chi)=E(s) \frac{L^{S_\delta}(s,\chi \legendresymbol{\delta}{\cdot})}{L^{S_\delta}(2s,\chi^2)}
  \]
  where 
\[
E(s)=\prod_{\substack{p\notin S\\p\mid\delta}}\prescript{2}{c}D_{\xi,\alpha}^{(p)}(s,\chi) 
\]
is an entire function (depending on $\xi,\alpha$ and $\chi$) that can be explicitly computed.
\item\label{item:610secondterm} Suppose that $\xi=\alpha=0$. Then 
  \[
  \prescript{2}{c}D_{\xi,\alpha}^{S}(s,\chi)=\frac{L^S(2s-1,\chi^2)L^S(2s-2,\chi^2)}{L^S(2s,\chi^2)}.
  \]
\item\label{item:610thirdterm} Suppose that $\delta=0$ while $(\xi,\alpha)\neq (0,0)$. Then $c$ is a square and
  \[
  \prescript{2}{c}D_{\xi,\alpha}^{S}(s,\chi)=E(s)\frac{L^{S}(2s-1,\chi^2)}{L^{S_{\kappa}}(2s,\chi^2)},
  \]
  where $\kappa\in \{\xi',\alpha'\}$ is nonzero, and 
  \[
  E(s)=\prod_{\substack{p\notin S\\ p\mid \kappa}}\prescript{}{c}{A}_{\xi,\alpha}^{(p)}(s,\chi)
  \] 
  is an entire function (depending on $\xi,\alpha$ and $\chi$) that can be explicitly computed.
\end{enumerate}
In particular, $\prescript{2}{c}D_{\xi,\alpha}^{S}(s,\chi)$ admits a meromorphic continuation to the complex plane for any $\xi,\alpha$ and $\chi$.
\end{theorem}

\section{Bounds of Kloosterman-type series}
In this subsection we will prove that the Kloosterman-type series $_c^2D_{\xi,\alpha}^S(s,\chi)$ is of moderate growth. More precisely, we want to prove that
\begin{theorem}\label{thm:kloostermanmoderate}
For $\Re s>1/2$, there exist constants $A,B,C>0$ such that
\[
_c^2D_{\xi,\alpha}^S(s,\chi)\ll (1+|s|)^A(1+|\xi|_\infty)^B\prod_{i=1}^{r}(1+|\xi|_{q_i})^B (1+|\alpha|_\infty)^C\prod_{i=1}^{r}(1+|\alpha|_{q_i})^C.
\]
The implied constant depends only on $c$, $\chi$, $A$, $B$, and $C$. Moreover, we may take $A=1/2$.
\end{theorem}

In this subsection, we will use the notation $\delta=\alpha^2-4c \xi^2$. 

Since $\xi\neq 0$, \autoref{item:610secondterm} of \autoref{thm:kloostermantypeseries} cannot occur. Hence we only need to consider \autoref{item:610firstterm} and \autoref{item:610thirdterm} of this theorem.

\subsection{Case $\delta\neq 0$} We assume that  \autoref{item:610firstterm}  of \autoref{thm:kloostermantypeseries}  holds, that is, $\delta\neq 0$.
\begin{lemma}\label{lem:kloostermantrivial}
We have
\[
|\prescript{2}{c}{\widehat{\Kl}}_{p^u,p^v}^{(p)}(\xi,\alpha)|\leq (2v+1)p^{2u+2v}.
\]
\end{lemma}
\begin{proof}
This follows directly from \autoref{thm:localkloostermaninert} and \autoref{thm:localkloostermansplit}.
\end{proof}

\begin{lemma}\label{lem:kloostermanlocalestimate}
Suppose that $\delta\neq 0$ and $\Re s>1/2$. Then we have
\[
|{}_c^2D_{\xi,\alpha}^{(p)}(s,\chi)|\ll p^{2+v_p(\delta)}.
\]
\end{lemma}
\begin{proof}
By definition we have
\[
{}_c^2D_{\xi,\alpha}^{(p)}(s,\chi)=\sum_{u,v=0}^{+\infty} \frac{\prescript{2}{c}{\widehat{\Kl}}_{p^u,p^v}^{(p)}(\xi,\alpha)\chi(p^{u+2v})}{(p^u) ^{1+s}(p^u)^{1+2s}}.
\]

We consider the case $\legendresymbol{c}{p}=1$ and $\legendresymbol{c}{p}=-1$ separately.

\underline{\emph{Case 1:}}\ \  $\legendresymbol{c}{p}=-1$.

By \autoref{thm:localkloostermaninert} we find that $\prescript{2}{c}{\widehat{\Kl}}_{p^u,p^v}^{(p)}(\xi,\alpha)\neq 0$ only if $v_p(\xi)\geq u+v-1$ and $v_p(\alpha)\geq u+v-1$.

Hence we have 
\[
\max\{u,v\}\leq 1+\min\{v_p(\xi),v_p(\alpha)\}\leq 1+v_p(\delta).
\]
if $\alpha\neq 0$ and $\max\{u,v\}\leq 1+v_p(\xi)$ if $\alpha=0$.
Hence by \autoref{lem:kloostermantrivial} we obtain
\begin{align*}
|{}_c^2D_{\xi,\alpha}^{(p)}(s,\chi)|&\leq \sum_{u=0}^{1+v_p(\delta)} \sum_{v=0}^{1+v_p(\delta)}\frac{|\prescript{2}{c}{\widehat{\Kl}}_{p^u,p^v}^{(p)} (\xi,\alpha)|}{(p^u)^{1+\Re s}(p^v)^{1+2\Re s}}\leq \sum_{u=0}^{1+v_p(\delta)} \sum_{v=0}^{1+v_p(\delta)}\frac{(2v+1)p^{2u+2v}}{p^{(3/2+\varepsilon) u+(2+\varepsilon)v}}\\
&=\sum_{u=0}^{1+v_p(\delta)}p^{u(1/2-\varepsilon)} \sum_{v=0}^{1+v_p(\delta)}\frac{2v+1}{p^{v\varepsilon}}
\end{align*}
for $\alpha\neq 0$ and similarly
\[
|{}_c^2D_{\xi,\alpha}^{(p)}(s,\chi)|\leq \sum_{u=0}^{1+v_p(\xi)}p^{u(1/2-\varepsilon)} \sum_{v=0}^{1+v_p(\xi)}\frac{2v+1}{p^{v\varepsilon}}
\]
if $\alpha=0$,
where $\varepsilon=\Re s-1/2>0$. Since $\varepsilon>0$, the sum over $v$ is bounded and is independent of $p$. Since $v_p(\delta)=2v_p(\xi)$ if $\alpha=0$, we obtain
\[
|{}_c^2D_{\xi,\alpha}^{(p)}(s,\chi)|\ll \sum_{u=0}^{1+v_p(\delta)}p^{u/2}\ll p^{2+v_p(\delta)}. 
\]
for $\alpha\neq 0$ and
\[
|{}_c^2D_{\xi,\alpha}^{(p)}(s,\chi)|\ll \sum_{u=0}^{1+v_p(\delta)}p^{u/2}\ll p^{2+v_p(\delta)}
\]
for $\alpha=0$.

\underline{\emph{Case 2:}}\ \  $\legendresymbol{c}{p}=1$.

Note that we have $\delta=-16c\xi'\alpha'\neq 0$ by the definition \eqref{eq:defxiprime}. Hence by the proof of \autoref{prop:finitesumsplit}, we know that $\prescript{2}{c}{\widehat{\Kl}}_{p^u,p^v}^{(p)}(\xi,\alpha)\neq 0$ only if 
\[
u+2v-2\leq v_p(\xi')+v_p(\alpha')=v_p(\xi'\alpha')=v_p(\delta).
\]
Therefore, $u\leq 2+v_p(\delta)$ and $v\leq 1+v_p(\delta)/2$. Hence by \autoref{lem:kloostermantrivial}, we have
\begin{align*}
|{}_c^2D_{\xi,\alpha}^{(p)}(s,\chi)|&\leq \sum_{u=0}^{2+v_p(\delta)} \sum_{v=0}^{1+\lfloor v_p(\delta)/2\rfloor}\frac{|\prescript{2}{c}{\widehat{\Kl}}_{p^u,p^v}^{(p)} (\xi,\alpha)|}{(p^u)^{1+\Re s}(p^v)^{1+2\Re s}}\leq \sum_{u=0}^{2+v_p(\delta)} \sum_{v=0}^{1+\lfloor v_p(\delta)/2\rfloor}\frac{(2v+1)p^{2u+2v}}{p^{(3/2+\varepsilon) u+(2+\varepsilon)v}}\\
&=\sum_{u=0}^{2+v_p(\delta)}p^{u(1/2-\varepsilon)} \sum_{v=0}^{1+\lfloor v_p(\delta)/2\rfloor}\frac{2v+1}{p^{v\varepsilon}}\ll p^{1+(1/2-\varepsilon)(2+v_p(\delta))}\ll p^{2+v_p(\delta)}
\end{align*}
since the sum over $v$ is bounded and is independent of $p$.
\end{proof} 

\begin{proposition}\label{prop:kloostermansumbound}
If $\delta\neq 0$ and $\Re s>1/2$, then
\[
E(s)=\prod_{\substack{p\notin S\\p\mid\delta}}\prescript{2}{c}D_{\xi,\alpha}^{(p)}(s,\chi),
\]
defined in \autoref{item:610firstterm} of \autoref{thm:kloostermantypeseries}, satisfies
\[
E(s)\ll \prod_{v\in S}(1+|\xi|_v)^{6+\varepsilon}(1+|\alpha|_v)^{6+\varepsilon},
\]
for any $\varepsilon>0$,
where the implied constant depends only on $c$, $\varepsilon$, $\Re s$ and $\chi$.
\end{proposition}
\begin{proof}
By \autoref{lem:kloostermanlocalestimate} we have
\[
  E(s)\ll\prod_{\substack{p\notin S\\p\mid\delta}}O(1)p^{2+v_p(\delta)}.
\]
We have (see the proof of \cite[Corollary B.3]{cheng2025b})
\[
\prod_{p\mid a}O(1)\ll_\varepsilon |a^{(q)}|^\varepsilon
\]
and
\[
\prod_{p\mid a}p^{2+v_p(a)}=a^{(q)}\left(\prod_{p\mid a}p\right)^2\leq |a^{(q)}|^3.
\]
Hence know that for any $\varepsilon>0$,
\[
E(s)\ll|\delta^{(q)}|^{3+\varepsilon}.
\]

Since $\delta=\alpha^2-4c\xi^2$, we obtain
\[
|\delta|_v\ll_c |\alpha|_v^2+|\xi|_v^2
\]
for any $v\in S$. Therefore
\begin{equation}\label{eq:estimatedeltaq}
|\delta^{(q)}|=\prod_{v\in S}|\delta|_v\ll \prod_{v\in S}(|\alpha|_v^2+|\xi|_v^2)\ll \prod_{v\in S}(1+|\xi|_v)^2(1+|\alpha|_v)^2.
\end{equation}
Hence
\[
E(s)\ll\left[(1+|\xi|_v)^2(1+|\alpha|_v)^2\right]^{3+\varepsilon}=\left[(1+|\xi|_v) (1+|\alpha|_v)\right]^{6+\varepsilon'},
\]
where the implied constant depends only on $c$, $\varepsilon$, $\Re s$ and $\chi$.
\end{proof}

\subsection{Case $\delta=0$} Now we assume that  \autoref{item:610thirdterm}  of \autoref{thm:kloostermantypeseries} holds. In this case, we have $\delta=0$, $c$ is a square, and $\xi'\alpha'=0$. Recall that
\[
E(s)=\prod_{\substack{p\mid \kappa\\ p\notin S}}\prescript{}{c}{A}_{\xi,\alpha}^{(p)}(s,\chi),
\]
where $\kappa\in \{\xi',\alpha'\}$ is nonzero, and
\begin{equation}\label{eq:relad}
\prescript{}{c}{A}_{\xi,\alpha}^{(p)}(s,\chi)= \prescript{2}{c}D_{\xi,\alpha}^{(p)}(s,\chi)(1-\chi^2(p)p^{1-2s}).
\end{equation}
By symmetry we may assume that $\xi'=0$.
\begin{lemma}\label{lem:kloostermanlocalestimatesplit}
Suppose that $\legendresymbol{c}{p}=1$, $\xi'=0$, and $\Re s>1/2$. Then we have
\[
\prescript{}{c}{A}_{\xi,\alpha}^{(p)}(s,\chi)\ll (v_p(\alpha')+1)^2,
\]
where the implied constant depends only on $c$, $\Re s$ and $\chi$.
\end{lemma}
\begin{proof}
Denote
\[
{}_c^2D_{\xi,\alpha}^{(p)}(s,\chi)_{u}:=\sum_{v=0}^{+\infty} \frac{{}_c^2\widehat{\Kl}^{(p)}_{p^u,p^v}(\xi,\alpha)\chi(p^{2v})}{(p^u)^{1+s} (p^{2v})^ {1+2s}}
\]
and
\[
\prescript{}{c}{A}_{\xi,\alpha}^{(p)}(s,\chi)_u:= \prescript{2}{c}D_{\xi,\alpha}^{(p)}(s,\chi)_u(1-\chi^2(p)p^{1-2s})
\]
for $u\in \ZZ_{\geq 0}$. Hence
\[
{}_c^2D_{\xi,\alpha}^{(p)}(s,\chi)=\sum_{u=0}^{+\infty}{}_c^2D_{\xi,\alpha}^{(p)}(s,\chi)_{u}
\]
and
\[
\prescript{}{c}{A}_{\xi,\alpha}^{(p)}(s,\chi)=\sum_{u=0}^{+\infty}\prescript{}{c}{A}_{\xi,\alpha}^{(p)}(s,\chi)_u.
\]
We now recall the proof of \autoref{prop:onezerosplit}. If $u$ is odd, then
\[
{}_c^2D_{\xi,\alpha}^{(p)}(s,\chi)_{u}=0.
\]
If $u=0$, we have
\[
{}_c^2D_{\xi,\alpha}^{(p)}(s,\chi)_{0}=\sum_{v=0}^{v_0} \frac{{}_c^2\widehat{\Kl}^{(p)}_{1,p^v}(\xi,\alpha)\chi(p^{2v})}{(p^{2v})^ {1+2s}}+\frac{(v_p(\alpha')+1) (1-p^{-1})\chi^2(p^{v_0+1})}{(p^{2v_0+2})^{-1+2s}(1-\chi^2(p)p^{-2s+1})}.
\]
Thus by \eqref{eq:relad},
\[
\prescript{}{c}{A}_{\xi,\alpha}^{(p)}(s,\chi)_{u=0}=(1-\chi^2(p)p^{1-2s})\sum_{v=0}^{v_0}\frac{{}_c^2\widehat{\Kl}^{(p)}_{1,p^v}(\xi,\alpha)\chi(p^{2v})}{(p^{2v})^ {1+2s}}+\frac{(v_p(\alpha')+1) (1-p^{-1})\chi^2(p^{v_0+1})}{(p^{2v_0+2})^{-1+2s}}.
\]
Hence by \autoref{lem:kloostermantrivial} we obtain
\[
|\prescript{}{c}{A}_{\xi,\alpha}^{(p)}(s,\chi)_{0}|\ll \sum_{v=0}^{v_0}\frac{p^{2v}}{(p^{2v})^ {1+2\Re s}}+v_p(\alpha')\ll v_p(\alpha')+1.
\]
Now we consider $u\geq 2$ even. By \eqref{eq:relad},
\[
\prescript{}{c}{A}_{\xi,\alpha}^{(p)}(s,\chi)_{u}=(1-\chi^2(p)p^{1-2s})\sum_{v=0}^{v_u}\frac{{}_c^2\widehat{\Kl}^{(p)}_{p^u,p^v}(\xi,\alpha)\chi(p^{u+2v})} {(p^u)^{1+s}(p^{2v})^ {1+2s}}+\frac{(v_p(\alpha')+1)(1-p^{-1})\chi^2(p^{v_u+1})}{(p^{2v_u+2})^{-1+2s} }.
\]
Hence by \autoref{lem:kloostermantrivial} and a similar argument as in the $u=0$ case, we obtain 
\[
|\prescript{}{c}{A}_{\xi,\alpha}^{(p)}(s,\chi)_{u}|\ll v_p(\alpha')+1.
\]
Moreover, by the proof of \autoref{prop:onezerosplit} we know that ${}_c^2\widehat{\Kl}^{(p)}_{p^u,p^v}(\xi,\alpha)=0$ if $u>v_p(\alpha')+1$. Therefore
\[
\prescript{}{c}{A}_{\xi,\alpha}^{(p)}(s,\chi)=\sum_{u=0}^{+\infty}\prescript{}{c}{A}_{\xi,\alpha}^{(p)}(s,\chi)_{u} \ll (1+v_p(\alpha'))^2.\qedhere
\]
\end{proof}

\begin{proposition}\label{prop:kloostermansumbound2}
If $\delta=0$, $\Re s>1/2$, and $(\xi,\alpha)\neq (0,0)$, then
\[
E(s)=\prod_{\substack{p\mid \kappa\\ p\notin S}}\prescript{}{c}{A}_{\xi,\alpha}^{(p)}(s,\chi),
\]
defined in \autoref{item:610thirdterm} of \autoref{thm:kloostermantypeseries}, satisfies
\[
E(s)\ll \prod_{v\in S}(1+|\xi|_v)^{\varepsilon}(1+|\alpha|_v)^{\varepsilon},
\]
for any $\varepsilon>0$,
where the implied constant depends only on $c$, $\varepsilon$, $\Re s$ and $\chi$.
\end{proposition}
\begin{proof}
By symmetry we may assume that $\xi'= 0$ and $\alpha'\neq 0$ so that $\kappa=\alpha'$. By \autoref{lem:kloostermanlocalestimatesplit} we have
\[
E(s)=\prod_{\substack{p\mid \alpha'\\ p\notin S}}O(1)(v_p(\alpha')+1)^2=(\alpha'^{(q)})^\varepsilon\bm d(\alpha'^{(q)})^2\ll (\alpha'^{(q)})^{\varepsilon'}.
\]
By \eqref{eq:defxiprime}, for any $v\in S$ we have
\[
|\alpha'|_v=\left|\frac{\xi}{2}-\frac{\alpha}{4\sqrt{c}}\right|_v\ll_c |\xi|_v+|\alpha|_v\ll (1+|\xi|_v)(1+|\alpha|_v).
\]
Hence
\[
E(s)\ll \left[(1+|\xi|_v)(1+|\alpha|_v)\right]^\varepsilon.\qedhere
\]
\end{proof}

\subsection{Proof of the main theorem in this section} Now we prove \autoref{thm:kloostermanmoderate}.
\begin{proof}[Proof of \autoref{thm:kloostermanmoderate}]
We consider the following cases:

\underline{\emph{Case 1:}} \ \ $\delta\neq 0$.

In this case, $_c^2D_{\xi,\alpha}^S(s,\chi)$ becomes
\[
_c^2D_{\xi,\alpha}^S(s,\chi)=E(s) \frac{L^{S_\delta}(s,\chi \legendresymbol{\delta}{\cdot})}{L^{S_\delta}(2s,\chi^2)}.
\]
We have 
\begin{equation}\label{eq:estimateluxialpha}
\frac{1}{L^{S_\delta}(2s,\chi^2)}=\sum_{\gcd(m,S_\delta)=1}\frac{\bm\mu(m)\chi^2(m)}{m^s}\ll_{\Re s} 1
\end{equation}
since $\Re(2s)>1$, where $\bm\mu(m)$ denotes the M\"obius function. By \autoref{prop:kloostermansumbound}, we have
\begin{equation}\label{eq:kloostermanentire}
E(s)\ll \prod_{v\in S}(1+|\xi|_v)^{6+\varepsilon}(1+|\alpha|_v)^{6+\varepsilon}.
\end{equation}

Finally, we deal with the term $L^{S_\delta}(s,\chi \legendresymbol{\delta}{\cdot})$. We have
\[
L^{S_\delta}\left(s,\chi \legendresymbol{\delta}{\cdot}\right)=\prod_{\substack{p\notin S\\ p\mid\delta}}\left(1-\legendresymbol{\delta}{p}p^{-s}\right)L^{S}\left(s,\chi \legendresymbol{\delta}{\cdot}\right)=L^{S}\left(s,\chi \legendresymbol{\delta}{\cdot}\right).
\]
By the Weyl bound of Dirichlet $L$-functions, we have
\[
L^S\left(s,\chi \legendresymbol{\delta}{\cdot}\right)\ll_{S}\left[(1+|s|)C\!\left(\chi\legendresymbol{\delta}{\cdot}\right) \right]^{1/6+\varepsilon}\ll_{S,\chi} ((1+|s|)|\delta^{(q)}|)^{1/6+\varepsilon}
\]
for $\Re s\in \lopen 1/2,+\infty\ropen$, where $C(\chi)$ denotes the conductor of $\chi$. Now we use \eqref{eq:estimatedeltaq} and obtain
\begin{equation}\label{eq:estimateweylbound}
L^S\left(s,\chi \legendresymbol{\delta}{\cdot}\right)\ll_{S,\chi}  \left[(1+|s|)\prod_{v\in S}(1+|\xi|_v)^{2}\prod_{v\in S}(1+|\alpha|_v)^2\right]^{1/6+\varepsilon}.
\end{equation}
Now \eqref{eq:estimateluxialpha}, \eqref{eq:kloostermanentire}, \eqref{eq:estimateweylbound} yield the result.

\underline{\emph{Case 2:}} \ \ $\delta= 0$. 

By symmetry we may assume that $\xi'=0$. In this case, we obtain
\[
_c^2D_{\xi,\alpha}^S(s,\chi)=E(s)\frac{L^{S}(2s-1,\chi^2)}{L^{S_{\alpha'}}(2s,\chi^2)}.
\]
Since $\Re (2s)>1$, we have
\begin{equation}\label{eq:estimateluxialpha2}
\frac{1}{L^{S_{\alpha'}}(2s,\chi^2)}\ll_{\Re s} 1
\end{equation}
as in the previous case. Also, by the Weyl bound (or merely the convexity bound) we have
\begin{equation}\label{eq:estimateweylbound2}
L^S(2s-1,\chi^2)\ll_\chi (1+|s|)^{1/2}.
\end{equation}
for $\Re s>1/2$. Finally, by \autoref{prop:kloostermansumbound2} we have
\begin{equation}\label{eq:kloostermanentire2}
E(s)\ll \prod_{v\in S}(1+|\xi|_v)^{\varepsilon}(1+|\alpha|_v)^{\varepsilon}.
\end{equation}

Now \eqref{eq:estimateluxialpha2}, \eqref{eq:estimateweylbound2} and \eqref{eq:kloostermanentire2} yield the result.
\end{proof}

\section{Contribution of the $\xi\neq 0$ term}\label{sec:mainanalysis}
Now we back to the main problem. We don't use \autoref{ass:nonarchimedean} except for the last subsection. Recall that in \autoref{subsec:secondpoisson} we showed that
\[
S_G(X)=\frac12\sum_{c\in \{\pm 1\}\times q^{\FF_2^r}}\sum_{k,f\in \ZZ_{(S)}^{>0}}\frac{1}{k^3f^5}\sum_{\xi,\alpha\in \ZZ^S}\prescript{}{c}{\widehat{J}}_{k,f}(X,\xi,\alpha)\prescript{2}{c}{\widehat{\Kl}}_{k,f}^S(\xi,\alpha).
\]
Let
\[
\prescript{}{c}S_G(X)=\frac12\sum_{k,f\in \ZZ_{(S)}^{>0}}\frac{1}{k^3f^5}\sum_{\xi,\alpha\in \ZZ^S}\prescript{}{c}{\widehat{J}}_{k,f}(X,\xi,\alpha)\prescript{2}{c}{\Kl}_{k,f}^S(\xi,\alpha).
\]
Then we obtain
\[
S_G(X)=\sum_{c\in \{\pm 1\}\times q^{\FF_2^r}}\prescript{}{c}S_G(X).
\]
Now we fix $c\in \{\pm 1\}\times q^{\FF_2^r}$.
By \eqref{eq:transformedj} we obtain
\begin{align*}
\prescript{}{c}S_G(X)=&\sum_{\xi,\alpha\in \ZZ^S}\sum_{k,f\in \ZZ_{(S)}^{>0}}\frac{\prescript{2}{c}{\Kl}_{k,f}^S(\xi,\alpha)}{k^3f^5}\prod_{i=1}^{r}(1-q_i^{-1})^{-1} \int_{(x,a)\in\RR^2}\int_{(y,b)\in\QQ_{S_\fin}^2} k^2f^4 G\legendresymbol{kf^2|a|_\infty|b|_q}{X} \theta_\infty(x,ca^2)\\
\times&\sum_{\chi}\widehat{\theta}_{q}(\chi,y,cb^2)\chi(kf^2)\left[F\legendresymbol {1}{|x^2-4ca^2|_\infty^{1/2}|y^2- 4cb^2|_q'^{1/2}}+\frac{1}{|x^2-4ca^2|_\infty^{1/2}|y^2- 4cb^2|_q'^{1/2}}\right.\\
\times&\left.V\legendresymbol{1}{|x^2-4ca^2|_\infty^{1/2}|y^2- 4cb^2|_q'^{1/2}}\right]\rme(-x\xi-a\alpha)\rme_{q}(-y\xi-b\alpha)\rmd x\rmd a\rmd y\rmd b.
\end{align*}
Let $\widetilde{G}(s)$ be the Mellin transform of $G$. By Mellin inversion formula for $G$ we then have
\begin{align*}
&\prescript{}{c}S_G(X)=\prod_{i=1}^{r}(1-q_i^{-1})^{-1} \sum_{\xi,\alpha\in \ZZ^S}\sum_{k,f\in \ZZ_{(S)}^{>0}}\frac{\prescript{2}{c}{\Kl}_{k,f}^S(\xi,\alpha)}{kf} \int_{(x,a)\in\RR^2}\int_{(y,b)\in\QQ_{S_\fin}^2} \int_{(\sigma)}\widetilde{G}(s)\frac{X^s}{(kf^2)^s|a|_\infty^s|b|_q^s}\rmd s \\
\times&\theta_\infty(x,ca^2)\sum_{\chi}\widehat{\theta}_{q}(\chi,y,cb^2)\chi(kf^2) \left[F\legendresymbol{1}{|x^2-4ca^2|_\infty^{1/2}|y^2- 4cb^2|_q'^{1/2}}+\frac{1}{|x^2-4ca^2|_\infty^{1/2}|y^2- 4cb^2|_q'^{1/2}}\right.\\
\times&\left.V\legendresymbol{1}{|x^2-4ca^2|_\infty^{1/2}|y^2- 4cb^2|_q'^{1/2}}\right]\rme(-x\xi-a\alpha)\rme_{q}(-y\xi-b\alpha)\rmd x\rmd a\rmd y\rmd b\\
=&\prod_{i=1}^{r}(1-q_i^{-1})^{-1}\sum_{\chi}\int_{(\sigma)}\widetilde{G}(s) \sum_{\xi,\alpha\in \ZZ^S}{}^2_c D_{\xi,\alpha}^S(s,\chi){}_c\Phi_{\xi,\alpha}(s,\chi)X^s\rmd s,
\end{align*}
where the sum over $\chi$ is over all characters such that it is unramified at places outside $S$, and
\begin{equation}\label{eq:defphi}
\begin{split}
\prescript{}{c}\Phi_{\xi,\alpha}(s,\chi)=& \int_{(x,a)\in\RR^2}\int_{(y,b)\in\QQ_{S_\fin}^2} \frac{1}{|a|_\infty^s|b|_q^s} \theta_\infty(x,ca^2)\widehat{\theta}_{q}(\chi,y,cb^2) \left[F\legendresymbol{1}{|x^2-4ca^2|_\infty^{1/2}|y^2- 4cb^2|_q'^{1/2}}\right.\\
+&\left.\frac{1}{|x^2-4ca^2|_\infty^{1/2}|y^2- 4cb^2|_q'^{1/2}}V\legendresymbol{1}{|x^2-4ca^2|_\infty^{1/2}|y^2- 4cb^2|_q'^{1/2}}\right]\\
\times&\rme(-x\xi-a\alpha)\rme_{q}(-y\xi-b\alpha)\rmd x\rmd a\rmd y\rmd b.
\end{split}
\end{equation}
\subsection{Estimate of $\prescript{}{c}\Phi_{\xi,\alpha}(s,\chi)$}
By making change of variable $x\mapsto ax$ and $y\mapsto by$, we obtain
\begin{align*}
&\prescript{}{c}\Phi_{\xi,\alpha}(s,\chi)= \int_{(x,a)\in\RR^2}\int_{(y,b)\in\QQ_{S_\fin}^2} \frac{1}{|a|_\infty^{s-1}|b|_q^{s-1}} \theta_\infty(x,c)\widehat{\theta}_{q}(\chi,by,cb^2) \left[F\legendresymbol{1/(|a|_\infty|b|_q)}{|x^2-4c|_\infty^{1/2}|y^2- 4c|_q'^{1/2}}\right.\\
+&\left.\frac{1/(|a|_\infty|b|_q)}{|x^2-4c|_\infty^{1/2}|y^2- 4c|_q'^{1/2}}V\legendresymbol{1/(|a|_\infty|b|_q)}{|x^2-4c|_\infty^{1/2}|y^2- 4c|_q'^{1/2}}\right]\rme(-ax\xi-a\alpha)\rme_{q}(-by\xi-b\alpha)\rmd x\rmd a\rmd y\rmd b.
\end{align*}

The main goal of this subsection is to prove the following theorem:
\begin{theorem}\label{thm:globalestimatefinal}
$\prescript{}{c}\Phi_{\xi,\alpha}(s,\chi)$, as a function of $s$, can be analytically continued to the half plane $\Re s>1/2$. Moreover, for any $M_1,M_2,N_1,N_2>0$ we have
\[
\prescript{}{c}\Phi_{\xi,\alpha}(s,\chi)\ll|\xi|^{-M_1}\prod_{i=1}^{r}(1+|\xi|_{q_i})^{-M_2}(1+|\alpha|)^{-N_1} \prod_{i=1}^{r}(1+|\alpha|_{q_i})^{-N_2},
\]
where the implied constant depends only on $\sigma=\Re s$, $c$, $\chi$, $M_1$, $M_2$, $N_1$ and $N_2$.
\end{theorem}

\subsubsection{Warm up: The unramified case} 
Our aim is to prove an estimate of $\prescript{}{c}\Phi_{\xi,\alpha}(s,\chi)$. In order to increase readability, we first consider the unramified case. In this case, the analog of $\prescript{}{c}\Phi_{\xi,\alpha}(s,\chi)$ is 
\[
\prescript{}{c}\Phi_{\xi,\alpha}(s)=\int_{\RR^2}\frac{1}{|a|^{s-1}}\theta_\infty(x,c)\left[F\legendresymbol {1/|a|}{|x^2-4c|^{1/2}}+ \frac{1/|a|}{|x^2-4c|^{1/2}} V\legendresymbol{1/|a|}{|x^2-4c|^{1/2}}\right]\rme(-ax\xi-a\alpha)\rmd x\rmd a.
\]

By making change of variable $x\mapsto 2\sqrt{|c|}x$ and $a\mapsto |c|^{-1/2}a/2$ we may assume that $c=\pm 1/4$. Hence it suffices to consider
\begin{equation}\label{eq:archimedeanintegralestimate}
\prescript{}{\pm}\Phi_{\xi,\alpha}(s)=\int_{\RR^2}\frac{1}{|a|^{s-1}}\theta_\infty^\pm(x)\left[F\legendresymbol {1/|a|}{|x^2\mp 1|^{1/2}}+ \frac{1/|a|}{|x^2\mp 1|^{1/2}} V\legendresymbol{1/|a|}{|x^2\mp 1|^{1/2}}\right]\rme(-ax\xi-a\alpha)\rmd x\rmd a.
\end{equation}

The main observation is that we can split the integral involving $F$ and $V$. Hence it suffices to consider a general case. 
\begin{theorem}\label{thm:archimedeanintegralestimate}
Let $H\in \cS$ and $\iota\in \{0,1\}$. Let $\varphi(x)$ be a continuous function on $X_\iota$ supported on a bounded set, of the form $\varphi(x)=\overline{\varphi}(x)|x^2\mp 1|^{d/2}$ for some smooth function $\overline{\varphi}$ and $d\in \ZZ_{\geq -1}$.
Then
\[
\Phi(s):=\int_{\RR}\int_{X_\iota}\frac{1}{|a|^{s}}\varphi(x)H\legendresymbol {1/|a|}{|x^2\mp 1|^{1/2}}\rme(-ax\xi-a\alpha)\rmd x\rmd a 
\]
has an analytic continuation to the half plane $\Re s>-1/2$, and 
\[
\Phi(s)\ll_{\sigma,M,N} |\xi|^{-M}(1+|\alpha|)^{-N}
\] 
for $\sigma\in \lopen -1/2,+\infty\ropen$.
\end{theorem}
\begin{proof}
We have
\[
\Phi(s):=\int_{\RR}\int_{X_\iota}\frac{1}{|a|^{s}}\varphi(x)H\legendresymbol {1/|a|}{|x^2\mp 1|^{1/2}}\rme^{-\dpii ax\xi}\rme^{-\dpii a\alpha}\rmd x\rmd a.
\]

\underline{\emph{Case 1:}}\ \ $|\alpha|\gg 1$.

Since $H$ is of rapid decay, the integrand is smooth at $a=0$. By integration by parts $N$ times with respect to the variable $a$, and the Leibniz rule, we have
\begin{align*}
\Phi(s)=&\frac{1}{(\dpii\alpha)^N}\int_{\RR}\int_{X_\iota}\varphi(x)\frac{\diff^N}{\diff a^N}\left(\frac{1}{|a|^s}H\legendresymbol{1/|a|}{|x^2\mp 1|^{1/2}}\rme^{-\dpii ax\xi}\right)\rme^{-\dpii a\alpha}\rmd x\rmd a\\
=&\frac{1}{(\dpii\alpha)^N}\int_{\RR}\int_{X_\iota}\varphi(x)\sum_{j=1}^{N}\binom{N}{j} \frac{\diff^{N-j}}{\diff a^{N-j}}\left(\frac{1}{|a|^s}H\legendresymbol{1/|a|}{|x^2\mp 1|^{1/2}}\right)(-\dpii x\xi)^j\rme^{-\dpii ax\xi}\rme^{-\dpii a\alpha}\rmd x\rmd a.
\end{align*}
Then we consider the integral over $x$. Since $H$ and all its derivatives are of rapid decay, the integrand vanishes on the boundary of $X_\iota$.
By integration by parts $L>N$ times we obtain
\begin{align*}
\Phi(s)=&\frac{1}{(\dpii\alpha)^N}\sum_{j=1}^{N}\binom{N}{j}\int_{\RR}\int_{X_\iota} \frac{1}{(\dpii a\xi)^L}\frac{\diff^{L}}{\diff x^{L}}\left[\varphi(x)\right.\\
\times&\left. \frac{\diff^{N-j}}{\diff a^{N-j}}\left(\frac{1}{|a|^s}F\legendresymbol{1/|a|}{|x^2\mp 1|^{1/2}}\right)(-\dpii x\xi)^j\right]\rme^{-\dpii ax\xi}\rme^{-\dpii a\alpha}\rmd x\rmd a.
\end{align*}
Since $H$ and all its derivatives are of rapid decay, the integrand vanishes at $a=0$. Also, by the expression it is direct to show that the integrand is $\ll |a|^{-\Re s-L}$ as $|a|\to +\infty$. Hence the integral is absolutely convergent on the half plane $\Re s>1-L$ and thus $\Re s>-1/2$ for $L\geq 2$. Moreover, in this region we have
\[
\Phi(s)\ll_{\sigma,L,N} |\alpha|^{-N}\sum_{j=1}^{N}\binom{N}{j}|\xi|^{j-L}\ll |\alpha|^{-N}|\xi|^{N-L}.
\]
Let $M=L-N$, we then obtain $\Phi(s)\ll_{\sigma,L,N}|\xi|^{-M}|\alpha|^{-N}$.

\underline{\emph{Case 2:}}\ \ $|\alpha|\ll 1$.

In this case, we directly consider the integral over $x$. Using integration by parts $M$ times yields
\[
\Phi(s)=\int_{\RR}\int_{X_\iota} \frac{1}{(\dpii a\xi)^M}\frac{\diff^{M}}{\diff x^{M}}\left[\varphi(x) \frac{1}{|a|^s}F\legendresymbol{1/|a|}{|x^2\mp 1|^{1/2}}\right]\rme^{-\dpii ax\xi}\rme^{-\dpii a\alpha}\rmd x\rmd a.
\]
A similar argument shows that $ \Phi(s)$ can be analytically continued to the half plane $\Re s>-1/2$ and  
\[
\Phi(s)\ll_{\sigma,L,N}|\xi|^{-M}.
\]

Combining the two cases yields the result.
\end{proof}

Therefore we obtain
\begin{corollary}
The function $\prescript{}{\pm}\Phi_{\xi,\alpha}(s)$ has an analytic continuation to the half plane $\Re s>1/2$. Moreover, for any $\sigma\in \lopen 1/2,+\infty\ropen$ we have
\[
\prescript{}{\pm}\Phi_{\xi,\alpha}(s)\ll_{\sigma,M,N} |\xi|^{-M}(1+|\alpha|)^{-N}
\]
for any $s$ with $\Re s=\sigma$ and $M,N\in \ZZ_{\geq 0}$.
\end{corollary}
\begin{proof}
Since we can replace $H(t)$ by $F(t)$ or $tV_\iota(t)$ in \autoref{thm:archimedeanintegralestimate}, by the expression \eqref{eq:archimedeanintegralestimate} we get the result.
\end{proof}

\subsubsection{The ramified case}
Now we consider the ramified case.
Observe that $\prescript{}{c}\Phi_{\xi,\alpha}(s,\chi)$ is the sum of
\begin{align*}
\prescript{}{c}\Phi_{\xi,\alpha}^1(s,\chi)=& \int_{(x,a)\in\RR^2}\int_{(y,b)\in\QQ_{S_\fin}^2} \frac{1}{|a|_\infty^{s-1}|b|_q^{s-1}} \theta_\infty(x,c)\widehat{\theta}_{q}(\chi,by,cb^2) \\
\times&F\legendresymbol{1/(|a|_\infty|b|_q)}{|x^2-4c|_\infty^{1/2}|y^2- 4c|_q'^{1/2}}\rme(-ax\xi-a\alpha)\rme_{q}(-by\xi-b\alpha)\rmd x\rmd a\rmd y\rmd b
\end{align*}
and
\begin{align*}
\prescript{}{c}\Phi_{\xi,\alpha}^2(s,\chi)=& \int_{(x,a)\in\RR^2}\int_{(y,b)\in\QQ_{S_\fin}^2} \frac{1}{|a|_\infty^{s-1}|b|_q^{s-1}} \theta_\infty(x,c)\widehat{\theta}_{q}(\chi,by,cb^2) \frac{1/(|a|_\infty|b|_q)}{|x^2-4c|_\infty^{1/2}|y^2- 4c|_q'^{1/2}}\\
\times&V\legendresymbol{1/(|a|_\infty|b|_q)}{|x^2-4c|_\infty^{1/2}|y^2- 4c|_q'^{1/2}}\rme(-ax\xi-a\alpha)\rme_{q}(-by\xi-b\alpha)\rmd x\rmd a\rmd y\rmd b.
\end{align*}

Hence it suffices to prove the following theorem:
\begin{theorem}\label{thm:globalnintegralestimate}
Let $H\in \cS$ and $\iota\in \{0,1\}$, $\epsilon\in \{0,\pm 1\}^r$. Let $\chi$ be a character of $\ZZ_{S_\fin}^\times$. Then the integral
\begin{align*}
\Phi(s):=& \int_{(x,a)\in\RR^2}\int_{(y,b)\in\QQ_{S_\fin}^2} \frac{1}{|a|_\infty^{s}|b|_q^{s}} \theta_\infty(x,c)\widehat{\theta}_{q}(\chi,by,cb^2) \\
\times&H\legendresymbol{1/(|a|_\infty|b|_q)}{|x^2-4c|_\infty^{1/2}|y^2- 4c|_q'^{1/2}}\rme(-ax\xi-a\alpha)\rme_{q}(-by\xi-b\alpha)\rmd x\rmd a\rmd y\rmd b
\end{align*}
can be analytically continued to the half plane $\Re s>-1/2$. Moreover, for any $\sigma\in\lopen -1/2,+\infty\ropen$ and $M_0$, $M_1$, $N_0$, $N_1>0$ we have
\[
\Phi(s)\ll |\xi|_\infty^{-M_0}\prod_{i=1}^{r}(1+|\xi|_{q_i})^{-M_1}(1+|\alpha|)^{-N_0} \prod_{i=1}^{r}(1+|\alpha|_{q_i})^{-N_1},
\]
where the implied constant depends only on $\sigma=\Re s$, $M_0$, $M_0$, $N_0$ and $N_1$.
\end{theorem}
Now \autoref{thm:globalnintegralestimate} implies \autoref{thm:globalestimatefinal}.

To enhance readability, we only consider the nonarchimedean analog and we will prove \autoref{thm:globalnintegralestimate} in \autoref{sec:globalnintegralestimateproof}.

\begin{definition}
Let $\chi$ be a character on $\ZZ_p^\times$. We say a function $f$ on $\QQ_p$ is \emph{$\chi$-spherical} if for any $a\in \ZZ_p^\times$ and $x\in \QQ_p$ we have
\[
f(ax)=\chi(a)f(x).
\]
We say $f$ is \emph{spherical} if it is $\triv$-spherical.
\end{definition}
Clearly $\widehat{\theta}_{q_i}(\chi_i,by,cb^2)$ is $\chi_i$-spherical. Hence the nonarchimedean analog can be written as follows.
\begin{theorem}\label{thm:nonarchimedeanintegralestimate}
Let $H\in \cS$ and $\epsilon\in \{0,\pm 1\}$. Let $\chi$ be a character of $\ZZ_p^\times$. Let $\psi(y,b)$ be a continuous function on $Y_\epsilon\times \QQ_p$ such that the function is compactly supported (in $\QQ_p^\times$) and $\chi$-spherical with respect to $b$, and compactly supported (in $\QQ_p$) with respect to $y$ uniformly in $b$, and $\psi(y,b) =\overline{\psi}_i(y_i,b_i)|y_i^2-4c|_p'^{d_i/2}$ for some smooth function $\overline{\psi}_i$ and $d_i\in \ZZ_{\geq -1}$.
Then
\[
\Phi(s):=\int_{\QQ_p}\int_{Y_\epsilon}\frac{1}{|b|_p^{s}}\psi(y,b)H\legendresymbol {1/|b|_p}{|y^2-4c|_p'^{1/2}}\rme_p(-by\xi-b\alpha)\rmd y\rmd b 
\]
has an analytic continuation to the half plane $\Re s>-1/2$, and 
\[
\Phi(s)\ll(1+|\xi|_p)^{-M}(1+|\alpha|_p)^{-N}
\] 
for $\sigma\in \lopen -1/2,+\infty\ropen$. The implied constant depends only on $\psi$, $H$, $\sigma$, $M$ and $N$. 
\end{theorem}
Before proving this theorem, we need the following lemma.
\begin{lemma}
If $f$ is $\chi$-spherical, then $|f|$ is spherical.
\end{lemma}
\begin{proof}
Clear since $\chi$ is unitary.
\end{proof}
\begin{lemma}
Let $f\in L^1(\QQ_p)$. If $f$ is $\chi$-spherical, then its Fourier transform $\widehat{f}$ is $\chi^{-1}$-spherical.
\end{lemma}
\begin{proof}
For any $a\in \ZZ_p^\times$ we have
\[
\widehat{f}(a\xi)=\int_{\QQ_p}f(x)\rme_p(-xa\xi)\rmd x.
\]
By making change of variable $x\mapsto a^{-1}x$, we obtain
\[
\widehat{f}(a\xi)=\int_{\QQ_p}f(a^{-1}x)\rme_p(-x\xi)\rmd x=\chi(a^{-1}) \int_{\QQ_p}f(a^{-1}x)\rme_p(-x\xi)\rmd x=\chi^{-1}(a)\widehat{f}(\xi).\qedhere
\]
\end{proof}
By the above two lemmas, we obtain the following corollary:
\begin{corollary}\label{cor:absolutespherical}
If $f\in L^1(\QQ_p)$ is $\chi$-spherical, then $|\widehat{f}(\xi)|_p$ is spherical.
\end{corollary}

\begin{lemma}\label{lem:rapiddcay}
Under the assumptions of \autoref{thm:nonarchimedeanintegralestimate}, we have
\[
\int_{U}\psi(y,b)H\legendresymbol {1/|b|_p}{|y^2-4c|_p'^{1/2}}\rme_p(-by\xi)\rmd y\ll (1+|b\xi|_p)^{-N}
\]
for any $N>0$, where $U$ is a compact open subset of $Y_\epsilon$. The implied constant depends only on $U$, $\psi$ and $H$.
\end{lemma}
\begin{proof}
This is essentially a baby case of \cite[Lemma B.8]{cheng2025}.
\end{proof}

Now we have all the ingredients for proving  \autoref{thm:nonarchimedeanintegralestimate}.
\begin{proof}[Proof of \autoref{thm:nonarchimedeanintegralestimate}]
Since $\psi(y,b)$ is compactly supported for $b\in \QQ_p^\times$, it is zero near $0$ and infinity. Hence $\Phi(s)$ is an entire function for $s$.

Note that the function
\[
\frac{1}{|b|_p^{s}}\psi(y,b)H\legendresymbol {1/|b|_p}{|y^2-4c|_p'^{1/2}}
\]
is $\chi$-spherical with respect to $b$ by assumption. We will split the integral over $y$ into three pieces: Let
\[
Y_<=\{y\in\QQ_p\,|\, |y\xi|_p<|\alpha|_p\},\ Y_==\{y\in\QQ_p\,|\, |y\xi|_p=|\alpha|_p\},\ Y_>=\{y\in\QQ_p\,|\, |y\xi|_p>|\alpha|_p\}.
\] 
and we define $Y_{\epsilon,i}=Y_\epsilon\cap Y_i$ for $i\in\{<,=,>\}$. Hence
\[
Y_\epsilon=Y_{\epsilon,<}\sqcup Y_{\epsilon,=}\sqcup Y_{\epsilon,>}.
\]
We write the corresponding functions of $\Phi$ as $\Phi_<$, $\Phi_=$ and $\Phi_>$. 

\underline{\emph{Case 1:}}\ \ $|y\xi|_p<|\alpha|_p$.

In this case $|y\xi+\alpha|_p=|\alpha|_p$. Hence by \autoref{cor:absolutespherical} we obtain
\[
|\Phi_<(s)|=\left|\int_{\QQ_p}\int_{Y_{\epsilon,<}}\frac{1}{|b|_p^{s}}\psi(y,b)H\legendresymbol {1/|b|_p}{|y^2-4c|_p'^{1/2}}\rme_p(-b\alpha)\rmd y\rmd b\right|.
\]
Hence it suffices to estimate the integral in the absolute value.

We claim that $\Phi_<(s)=0$ if $|\alpha|_p$ is sufficiently large (independent of $y$ and $\xi$). Indeed, since the integrand is zero near $0$ and infinity,
\[
\frac{1}{|b|_p^{s}}\psi(y,b)H\legendresymbol {1/|b|_p}{|y^2-4c|_p'^{1/2}}
\]
is smooth and compactly supported with respect to $b$. Hence the Fourier transform is $0$ for $|\alpha|_p$ sufficiently large.

Now we consider the case $|\alpha|_p\ll 1$. By definition of $\Phi_<(s)$ we have
\[
\Phi_<(s)=\int_{\QQ_p}\int_{Y_{\epsilon,<}}\frac{1}{|b|_p^{s}}\psi(y,b)H\legendresymbol {1/|b|_p}{|y^2-4c|_p'^{1/2}}\rme_p(-by\xi)\rme_p(-b\alpha)\rmd y\rmd b. 
\]

In this time, we consider the integral over $y$, that is
\begin{equation}\label{eq:integraly}
\int_{Y_{\epsilon,<}}\psi(y,b)H\legendresymbol {1/|b|_p}{|y^2-4c|_p'^{1/2}}\rme_p(-by\xi)\rmd y.
\end{equation}
Recall that
\[
Y_{\epsilon,<}=\left\{y\in Y_\epsilon\,\middle|\, |y|_p<\frac{|\alpha|_p}{|\xi|_p}\right\}
\]
Since the function
\[
\psi(y,b)H\legendresymbol {1/|b|_p}{|y^2-4c|_p'^{1/2}}
\]
is smooth at $0$, there exists a constant $C$ (independent of $\alpha$) depending only on $\psi$ such that if $|\xi|_p>C$, then the above function is a smooth function. Hence we find that $\eqref{eq:integraly}=0$ if $|b\xi|_p$ is sufficiently large (independent of $\alpha$). Since $|b|_p\asymp 1$ in the support of $\Phi$, we may replace $|b\xi|_p$ by $|\xi|_p$. Finally, for any $\alpha$ and $\xi$ we have the trivial estimate
\[
\Phi_<(s)\ll\int_{\QQ_p}\int_{Y_{\epsilon,<}}\left|\frac{1}{|b|_p^{s}}\psi(y,b)H\legendresymbol {1/|b|_p}{|y^2-4c|_p'^{1/2}}\right|\rmd y\rmd b\ll_{\sigma,\psi,H} 1.
\]
Therefore we have
\[
\Phi_<(s)\ll (1+|\xi|_p)^{-M}(1+|\alpha|_p)^{-N},
\]
where the implied constant depends only on $\psi$, $H$, $\sigma$, $M$ and $N$. 

\underline{\emph{Case 2:}}\ \ $|y\xi|_p>|\alpha|_p$.

In this case $|y\xi+\alpha|_p=|y\xi|_p$. Hence by \autoref{cor:absolutespherical} we obtain
\[
|\Phi_>(s)|=\left|\int_{\QQ_p}\int_{Y_{\epsilon,>}}\frac{1}{|b|_p^{s}}\psi(y,b)H\legendresymbol {1/|b|_p}{|y^2-4c|_p'^{1/2}}\rme_p(-by\xi)\rmd y\rmd b\right|.
\]
As in Case 1, the integration over $b$ vanishes if $|y\xi|_p$ is sufficiently large (independent of $\alpha$). Since $|y\xi|_p>|\alpha|_p$, $\Phi_>(s)=0$ for $|\alpha|_p$ sufficiently large. 

Now we assume that $|\alpha|_p\ll 1$. By \autoref{lem:rapiddcay} we have
\[
\int_{Y_{\epsilon,>}}\psi(y,b)H\legendresymbol {1/|b|_p}{|y^2-4c|_p'^{1/2}}\rme_p(-by\xi)\rmd y\ll (1+|b\xi|_p)^{-M}
\]
for any $M>0$. Since $|b|_p\asymp 1$, we know that the above integral is $(1+|\xi|_p)^{-M}$.

Therefore we have
\[
\Phi_>(s)\ll (1+|\xi|_p)^{-M}(1+|\alpha|_p)^{-N},
\]
where the implied constant depends only on $\psi$, $H$, $\sigma$, $M$ and $N$. 

\underline{\emph{Case 3:}}\ \ $|y\xi|_p=|\alpha|_p$.

First we assume that $|\alpha|\gg 1$. Then we have
\[
\Phi_=(s)=\frac{1}{|\alpha|_p^N}\int_{\QQ_p}\int_{Y_{\epsilon,=}}\frac{1}{|b|^s}\psi(y,b)H\legendresymbol{1/|b|_p }{|y^2-4c|_p'^{1/2}}|y\xi|_p^N\rme_p(-by\xi)\rme_p(-b\alpha)\rmd y\rmd b
\]
for any $N>0$.

Since $\psi$ is compactly supported with respect to $y$, we have $|y|_p^N\ll 1$ and hence by \autoref{lem:rapiddcay}, for any $M>0$ we have
\[
\int_{Y_{\epsilon,=}}\psi(y,b)H\legendresymbol {1/|b|_p}{|y^2-4c|_p'^{1/2}}|y\xi|_p^N\rme_p(-by\xi)\rmd y\ll |\xi|_p^N(1+|b\xi|_p)^{-M-N}\asymp (1+|\xi|_p)^{-M}
\]
since $|b|\asymp 1$. Therefore we obtain
\[
\Phi_=(s)\ll |\alpha|_p^{-N}(1+|\xi|_p)^{-M}.
\]
If $|\alpha|_p\ll 1$, we bound the integral over $y$ directly and obtain
\[
\int_{Y_{\epsilon,=}}\psi(y,b)H\legendresymbol {1/|b|_p}{|y^2-4c|_p'^{1/2}}\rme_p(-by\xi)\rmd y\ll (1+|b\xi|_p)^{-M}\asymp (1+|\xi|_p)^{-M}.
\]
Combining them we obtain 
\[
\Phi_=(s)\ll (1+|\alpha|_p)^{-N}(1+|\xi|_p)^{-M}.
\]
where the implied constant depends only on $\psi$, $H$, $\sigma$, $M$ and $N$.

Now the theorem follows from the above three cases and that
\[
\Phi(s)=\Phi_<(s)+\Phi_=(s)+\Phi_>(s).\qedhere
\]
\end{proof}

\subsection{Asymptotic formula for $S_G^{\xi\neq 0}(X)$}
In this subsection we compute the contribution of the $\xi\neq 0$ part. Recall that
\[
_cS_G^{\xi\neq 0}(X)=\prod_{i=1}^{r}(1-q_i^{-1})^{-1}\sum_{\chi}\frac{1}{\dpii}\int_{(\sigma)}\widetilde{G}(s) \sum_{\substack{\xi,\alpha\in \ZZ^S\\ \xi\neq 0}}{}^2_c D_{\xi,\alpha}^S(s,\chi){}\prescript{}{c}\Phi_{\xi,\alpha}(s,\chi)X^s\rmd s.
\]

We want to move the contour from $(\sigma)$ to $(\frac12+\varepsilon)$. By \autoref{thm:kloostermantypeseries}, the poles only come from the following cases:
\begin{enumerate}[itemsep=0pt,parsep=0pt,topsep=0pt,leftmargin=0pt,labelsep=2.5pt,itemindent=9pt,label=\textbullet]
  \item For $\xi,\alpha\in \ZZ^S$ such that $\delta=\xi^2-4c\alpha^2\neq 0$, there is a simple pole at $s=1$ if and only if 
      \begin{equation}\label{eq:defchidelta}
      \chi=\chi_\delta:=\legendresymbol{D}{\cdot},
      \end{equation}
      where $D\in \ZZ^S$ such that $v_p(D)\in \{0,1\}$ for all $p\notin S$ and $\delta=Dy^2$ for some $y\in \ZZ_{(S)}$. The residue at $s=1$ is 
      \[
      \res_{s=1}\prescript{2}{c}D_{\xi,\alpha}^S(s,\chi)=E(1)\prod_{p\in S_\delta-S}(1-p^{-1})\frac{1}{\zeta^{S_\delta}(2)}=\prod_{p\in S_\delta-S}\prescript{2}{c}D_{\xi,\alpha}^{(p)}(1,\chi)\prod_{p\in S_\delta-S}(1-p^{-1})\frac{1}{\zeta^{S_\delta}(2)}.
      \]
      Since $\chi$ is unramified at all places outside $S$ by definition, we must have $D\in (\ZZ^S)^\times$. Equivalently, $v_p(\delta)$ is even for all $p\notin S$. 
      
  \item For $\xi,\alpha\in \ZZ^S$ such that $\delta=\alpha^2-4c \xi^2=0$, we know that $c$ is a square. By rescaling we may assume that $c=1/4$. Hence we have $\xi=\pm\alpha$.
      
      First we assume that $\xi=\alpha$. In this case we have $\alpha'= 0$ and $\xi'=\xi\neq 0$ by \eqref{eq:defxiprime}. Hence $\kappa=\xi'=\xi$ and there exists a simple pole at $s=1$ with 
      \[
      \res_{s=1}{}\prescript{2}{\frac14}D_{\xi,\xi}^S(s,\chi)=\frac12E(1)\prod_{p\in S}(1-p^{-1})\frac{1}{\zeta^{S_{\xi}}(2)}=\frac12\prod_{\substack{p\mid \xi\\ p\notin S}}\prescript{}{\frac14}A_{\xi,\xi}^{(p)}(1,\chi)\prod_{p\in S}(1-p^{-1})\frac{1}{\zeta^{S_{\xi}}(2)}
      \]
      for any quadratic character $\chi$.
      
      Similarly, for the case $\xi=-\alpha$, we have $\xi'=0$ and $\alpha'=-\alpha=\xi$. Hence $\kappa=\alpha'=\xi$ and there exists a simple pole at $s=1$ with 
      \[
      \res_{s=1}\prescript{2}{\frac14}D_{\xi,-\xi}^S(s,\chi)=\frac12\prod_{\substack{p\mid \xi\\ p\notin S}}\prescript{}{\frac14}A_{\xi,-\xi}^{(p)}(1,\chi)\prod_{p\in S}(1-p^{-1})\frac{1}{\zeta^{S_{\xi}}(2)}
      \]
      for any quadratic character $\chi$.
\end{enumerate} 
\begin{lemma}\label{lem:sintegersum}
Suppose that $M>1$ and $N>M$. Then
\[
\sum_{\xi\in \ZZ^S-\{0\}}|\xi|_\infty^{-M}\prod_{i=1}^{r}(1+|\xi|_{q_i})^{-N}
\]
converges absolutely.
\end{lemma}
\begin{proof}
The proof is similar to the last part of the proof of \cite[Proposition B.6]{cheng2025}. For rigidity we will give a proof. We have
\begin{align*}
   \sum_{\xi\in \ZZ^S}|\xi|_\infty^{-M}\prod_{i=1}^{r}(1+|\xi|_{q_i})^{-N}
   \ll&
   \sum_{I\subseteq \{1,\dots,r\}}\sum_{\substack{c_i\in \ZZ_{<0}\ i\in I\\ c_j\in \ZZ_{\geq 0}\ j\notin I}}\sum_{\substack{\xi\in \ZZ^S\\ v_{q_i}(\xi)=c_i \ i\in I}}|\xi|_\infty^{-M}\prod_{i=1}^{r} (1+|\xi|_{q_i})^{-N}\\
     \ll &\sum_{I\subseteq \{1,\dots,r\}}\prod_{i\in I}\sum_{\substack{c_i=-\infty\\ i\in I}}^{-1}\sum_{\xi\in \ZZ-\{0\}}\left(|\xi| \prod_{i\in I}q_i^{c_i}\right)^{-M}\prod_{i\in I}(1+q_i^{-c_i})^{-N}
     \\
     \ll &\sum_{\xi\in \ZZ-\{0\}}|\xi|^{-M}\sum_{I\subseteq \{1,\dots,r\}}\prod_{i\in I}\sum_{\substack{c_i=-\infty\\ i\in I}}^{-1}\frac{q_i^{Mc_i}}{(1+q_i^{c_i})^{N}},
\end{align*}
which converges absolutely since  $M>1$ and $N>M$.
\end{proof}
Now we can prove the asymptotic formula for $S_G^{\xi\neq 0}(X)$. 
\begin{theorem}\label{thm:contributionxinot0}
For $\alpha,\xi\in \ZZ^S$ we define $\delta=\prescript{}{c}\delta_{\xi,\alpha}=\alpha^2-4c\xi^2$. Then for any $\varepsilon>0$, we have
\[
S_G^{\xi\neq 0}(X)=\widetilde{G}(1)\ms{B}X+\widetilde{G}(1)\ms{C}X+O(\|G\|_{2,1}X^{\frac{1}{2}+\varepsilon}),
\]
where
\[
  \ms{B}=\sum_{c\in \{\pm 1\}\times q^{\FF_2^r}}\sum_{\substack{\xi,\alpha\in \ZZ^S\, \xi\neq 0\\ 2\mid v_p(\delta)\,\forall p\notin S}}\prod_{p\in S_\delta-S}\prescript{2}{c}D_{\xi,\alpha}^{(p)}(1,\chi_\delta) \\
  \times\prod_{p\in S_\delta-S}(1-p^{-1})\frac{1}{\zeta^{S_\delta}(2)}{}_c \Phi_{\xi,\alpha}(1,\chi_\delta)
  \]
and
\[
  \ms{C}=\frac12\sum_{\chi^2=\triv}\sum_{\xi\in \ZZ^S-\{0\}}\left(\prod_{\substack{p\mid \xi\\ p\notin S}}\prescript{}{\frac14}A_{\xi,\xi}^{(p)}(1,\chi){}_\frac14 \Phi_{\xi,\xi}(1,\chi)+\prod_{\substack{p\mid \xi\\ p\notin S}}\prescript{}{\frac14}A_{\xi,-\xi}^{(p)}(1,\chi)\prescript{}{\frac14}\Phi_{\xi,-\xi}(1,\chi)\right) \frac{1}{\zeta^{S_{\xi}}(2)}.
\]
The functions $\prescript{}{c}\Phi_{\xi,\alpha}(1,\chi)$, $\prescript{2}{c}D_{\xi,\alpha}^{(p)}(1,\chi)$ and $\prescript{}{c}{A}_{\xi,\alpha}^{(p)}(1,\chi)$ are defined in \eqref{eq:defphi}, \autoref{prop:splitkloostermanseries} and \autoref{prop:onezerosplit}, respectively, $S_\delta$ is defined in \eqref{eq:defsxi}, and $\chi_\delta$ is defined in \eqref{eq:defchidelta}. The sum over $\chi$ in the second equation is over all the quadratic characters such that it is unramified at places outside $S$.
The implied constant depends only on $f_\infty$, $f_{q_i}$ and $\varepsilon$.
\end{theorem}
\begin{proof}
We first consider the estimate of $\prescript{}{c}S_G^{\xi\neq 0}(X)$. 
By moving the contour from $(\sigma)$ to $(\frac12+\varepsilon)$, the residue formula, and the argument at the beginning of this subsection, we have
\begin{align*}
_cS_G^{\xi\neq 0}(X)&=\prod_{i=1}^{r}(1-q_i^{-1})^{-1}\sum_{\chi}\frac{1}{\dpii}\int_{(\frac12+\varepsilon)}\widetilde{G}(s) \sum_{\substack{\xi,\alpha\in \ZZ^S\\ \xi\neq 0}}{}^2_c D_{\xi,\alpha}^S(s,\chi){}\prescript{}{c}\Phi_{\xi,\alpha}(s,\chi)X^s\rmd s\\
&+\sum_{\substack{\xi,\alpha\in \ZZ^S\, \xi\neq 0\\ 2\mid v_p(\delta)\,\forall p\notin S}}\prod_{p\in S_\delta-S}\prescript{2}{c}D_{\xi,\alpha}^{(p)}(1,\chi_\delta) \prod_{p\in S_\delta-S}(1-p^{-1})\frac{1}{\zeta^{S_\delta}(2)}\prescript{}{c}\Phi_{\xi,\alpha} (1,\chi_\delta)
\end{align*}
if $c\in \{\pm 1\}\times q^{\FF_2^r}$ is not a square and for $c=1/4$,
\begin{align*}
_cS_G^{\xi\neq 0}(X)=&\prod_{i=1}^{r}(1-q_i^{-1})^{-1}\sum_{\chi}\frac{1}{\dpii}\int_{(\frac12+\varepsilon)}\widetilde{G}(s) \sum_{\substack{\xi,\alpha\in \ZZ^S\\ \xi\neq 0}}{}^2_c D_{\xi,\alpha}^S(s,\chi){}\prescript{}{c}\Phi_{\xi,\alpha}(s,\chi)X^s\rmd s\\
+&\sum_{\substack{\xi,\alpha\in \ZZ^S\, \xi\neq 0\\ 2\mid v_p(\delta)\,\forall p\notin S}}\prod_{p\in S_\delta-S}\prescript{2}{c}D_{\xi,\alpha}^{(p)}(1,\chi_\delta) \prod_{p\in S_\delta-S}(1-p^{-1})\frac{1}{\zeta^{S_\delta}(2)}\prescript{}{c}\Phi_{\xi,\alpha} (1,\chi_\delta)\\
+&\sum_{\chi^2=\triv}\sum_{\xi\in \ZZ^S-\{0\}}\left(\prod_{\substack{p\mid \xi\\ p\notin S}}{}_{\frac14}A_{\xi,\xi}^{(p)}(1,\chi){}_\frac14 \Phi_{\xi,\xi}(1,\chi)+\prod_{\substack{p\mid \xi\\ p\notin S}}{}_{\frac14}A_{\xi,-\xi}^{(p)}(1,\chi){}_\frac14 \Phi_{\xi,-\xi}(1,\chi)\right) \frac{1}{\zeta^{S_{\xi}}(2)}.
\end{align*}

Next we estimate the integral
\begin{equation}\label{eq:sxialphaestimate}
\int_{(\frac12+\varepsilon)}\widetilde{G}(s) \sum_{\substack{\xi,\alpha\in \ZZ^S\\ \xi\neq 0}}{}^2_c D_{\xi,\alpha}^S(s,\chi){}\prescript{}{c}\Phi_{\xi,\alpha}(s,\chi)X^s\rmd s.
\end{equation}

By \autoref{thm:kloostermanmoderate}, there exists $B,C>0$ such that 
\[
{}^2_c D_{\xi,\alpha}^S(s,\chi)\ll (1+|s|)^{1/2}(1+|\xi|)^B\prod_{i=1}^{r}(1+|\xi|_{q_i})^B (1+|\alpha|)^C\prod_{i=1}^{r}(1+|\alpha|_{q_i})^C.
\]
By \autoref{thm:globalestimatefinal} for any $M_1,M_2,N_1,N_2>0$ we have
\[
\prescript{}{c}\Phi_{\xi,\alpha}(s,\chi)\ll|\xi|^{-M_1}\prod_{i=1}^{r}(1+|\xi|_{q_i})^{-M_2}(1+|\alpha|)^{-N_1} \prod_{i=1}^{r}(1+|\alpha|_{q_i})^{-N_2}.
\]
Hence we know that for any $K_1,K_2,L_1,L_2>0$ we have
\[
\eqref{eq:sxialphaestimate}\ll \int_{(\frac12+\varepsilon)}|\widehat{G}(s)|(1+|s|)^{1/2}\sum_{\substack{\xi,\eta\in \ZZ^S\\\xi\neq 0}}|\xi|^{-K_1} \prod_{i=1}^{r}(1+|\xi|_{q_i})^{-K_2}(1+|\alpha|)^{-L_1} \prod_{i=1}^{r}(1+|\alpha|_{q_i})^{-L_2}\rmd |s| X^{\frac12+\varepsilon}.
\]
Clearly we may choose $K_i$ and $L_i$ such that $1<K_1<K_2$ and $1<L_1<L_2$. Hence by \autoref{lem:sintegersum} we have
\[
\sum_{\substack{\xi,\eta\in \ZZ^S\\\xi\neq 0}}|\xi|^{-K_1} \prod_{i=1}^{r}(1+|\xi|_{q_i})^{-K_2}(1+|\alpha|)^{-L_1} \prod_{i=1}^{r}(1+|\alpha|_{q_i})^{-L_2}\ll 1.
\]
Therefore by \autoref{cor:mellinnorm},
\[
\eqref{eq:sxialphaestimate}\ll \int_{(\frac12+\varepsilon)}|\widehat{G}(s)|(1+|s|)^{1/2}\rmd |s|X^{\frac12+\varepsilon}\ll \|G\|_{M_{1/2+\varepsilon}^{1/2}}X^{\frac12+\varepsilon}\ll \|G\|_{2,1}X^{\frac12+\varepsilon}.
\]

Since
\[
S_G^{\xi\neq 0}(X)=\sum_{c\in \{\pm 1\}\times q^{\FF_2^r}}{}_cS_G^{\xi\neq 0}(X),
\]
we obtain the desired result.
\end{proof}

Combining \autoref{thm:contributionxi0} and \autoref{thm:contributionxinot0}, we obtain

\begin{theorem}\label{thm:contributionxig}
Suppose that \autoref{ass:nonarchimedean} holds. Then for any $\varepsilon>0$, we have
\[
S_G(X)=\widetilde{G}(1)(\ms{A}+\ms{B}+\ms{C})X+O(\|G\|_{3,1}X^{\frac{1}{2}+\varepsilon}),
\]
where the implied constant depends only on $f_\infty$, $f_{q_i}$ and $\varepsilon$.
\end{theorem}

\section{Comparison of $S(X)$ and $S_G(X)$}
The goal of this section is to give an asymptotic formula of $S(X)$ from that of $S_G(X)$. To do this, we need to bound a single term $\Sigma^n(\xi)$. We will always assume that \autoref{ass:nonarchimedean} holds in this section. The method of this section is the same as that of \cite[Section 7]{cheng2025c}.
\begin{proposition}\label{prop:ramanujan}
For any $\varepsilon>0$, we have
\[
\Sigma^n(\xi)\ll_\varepsilon n^{\varrho+\varepsilon},
\]
where $\varrho$ is a bound towards the Ramanujan conjecture.
\end{proposition}
\begin{proof}
By the definition of $\varrho$, we have
\[
I_\cusp(f^n)\ll_\varepsilon n^{\varrho+\varepsilon}
\]
for any $\varepsilon>0$. By \autoref{ass:nonarchimedean} we obtain
\[
\Sigma^n(\xi)-\Sigma^n(\square)=I_\el(f^n)\ll_\varepsilon n^{\varrho+\varepsilon}.
\]
Hence it suffices to show that $\Sigma^n(\square)=0$.

Recall that
\begin{align*}
\Sigma^n(\square)=&2\sum_{\pm}\sum_{\nu\in \ZZ^r}q^{\nu/2}\sum_{\substack{T\in \ZZ^S\\ T^2\mp 4nq^\nu= \square}}\sum_{\substack{f\in \ZZ_{(S)}\\ f^2\mid T^2\mp 4nq^\nu}}  \sum_{k\in \ZZ_{(S)}^{>0}}\frac{1}{kf}\legendresymbol{(T^2\mp 4nq^\nu)/f^2}{k}\theta_\infty^\pm\legendresymbol{T}{2n^{1/2}q^{\nu/2}}\\
    \times &\prod_{i=1}^{r}\theta_{q_i}(T,\pm nq^\nu)\left[F\legendresymbol{kf^2}{|T^2\mp 4nq^\nu|_{\infty,q}'^{\vartheta}}+\frac{kf^2}{\sqrt{|T^2\mp 4nq^\nu|_{\infty,q}'}}V\legendresymbol{kf^2}{|T^2\mp 4nq^\nu|_{\infty,q}'^{1-\vartheta}}\right].
\end{align*}

Suppose that $T^2\mp 4nq^\nu$ is a square. If $T^2\mp 4nq^\nu\neq 0$, then the element $\gamma$ with $\Tr\gamma=T$ and $\det\gamma=\pm nq^\nu$ is regular hyperbolic. Since we are in \autoref{ass:nonarchimedean}, there exists $v\in S$ such that $f_v$ is supercuspidal. Hence by \autoref{lem:hyperbolicvanish}, we know that 
\[
\theta_\infty^\pm\legendresymbol{T}{2n^{1/2}q^{\nu/2}}=0
\]
if $f_\infty$ is supercuspidal and
\[
\theta_{q_i}(T,\pm nq^\nu)=0
\]
if $f_{q_i}$ is supercuspidal. Thus we proved that $\Sigma^n(\square)=0$.
\end{proof}

Now we choose $G$ more specific. Fix $\delta\in \lopen 0,1\ropen$. Let $\phi\in C_c^\infty(\lopen -1,1\ropen)$ such that $\int\phi=1$. For any $Y>1$, let $\phi_{Y^\delta}(x)=Y^{1-\delta}\phi(xY^{1-\delta})$ be the kernel function. For $\beta\in [1/2,1\ropen$, we define
\[
G_{Y^\delta,\beta}(x)=\triv_{[\beta,1]}* \phi_{Y^\delta}(x)=Y^{1-\delta}\int_{\RR}\triv_{[\beta,1]}(x-y) \phi(yY^{1-\delta})\rmd y.
\]
Thus we have
\begin{equation}\label{eq:propertyconvolution}
G_{Y^\delta,\beta}\legendresymbol{x}{Y}=\begin{cases}
                                          1, & \beta Y+Y^\delta\leq x\leq Y-Y^\delta \\
                                          O(1), & |x-\beta Y|<Y^\delta\ \text{or}\ |x-Y|<Y^\delta\\
                                          0,&\text{otherwise}.
                                        \end{cases}
\end{equation}

Since $G_{Y^\delta,\beta}^{(j)}=\triv_{[\beta,1]}* \phi_{Y^\delta}^{(j)}$, we have
\[
G_{Y^\delta,\beta}(x)=\int_{x-1}^{x-\beta}\phi_{Y^\delta}(y)\rmd y,
\]
we get
\[
  G'_{Y^\delta,\beta}(x)
  =\phi_{Y^\delta}(x-\beta)-\phi_{Y^\delta}(x-1).
\]
More generally, for \(j\ge 1\),
\[
  G_{Y^\delta,\beta}^{(j)}(x)
  =\phi_{Y^\delta}^{(j-1)}(x-\beta)
   -\phi_{Y^\delta}^{(j-1)}(x-1).
\]
Therefore
\[
  \|G_{Y^\delta,\beta}\|_\infty\ll 1,\qquad
  \|G_{Y^\delta,\beta}^{(j)}\|_\infty
  \ll_j Y^{j(1-\delta)}\quad (j\geq1),
\]
and
\[
  \|G_{Y^\delta,\beta}\|_1\ll 1,\qquad
  \|G_{Y^\delta,\beta}^{(j)}\|_1
  \ll_j Y^{(j-1)(1-\delta)}\quad (j\geq1).
\]
Hence for $M\geq 1$ we have $\|G_{Y^\delta,\beta}\|_{M,1}\ll Y^{(M-1)(1-\delta)}$ and $\|G_{Y^\delta,\beta}\|_{M,\infty}\ll Y^{M(1-\delta)}$.
The implied constants in the above two estimates depend only on $\phi$ and $\delta$.

\begin{lemma}
We have $\widetilde{G}_{X^\delta,\beta}(1)=1-\beta$.
\end{lemma}
\begin{proof}
This is precisely a part of \cite[Lemma 7.11]{cheng2025c}.
\end{proof}
Hence we obtain
\begin{theorem}\label{thm:contributionxigfinal}
We have
\[
S_{G_{X^\delta,\beta}}(X)=(\ms{A}+\ms{B}+\ms{C})(1-\beta)X+O(X^{2(1-\delta)}X^{\frac12+\varepsilon}).
\]
\end{theorem}
\begin{proof}
The theorem follows from the previous lemma and \autoref{thm:contributionxig}.
\end{proof}

We end this section by giving an asymptotic formula for $S(X)$.
\begin{proposition}
Let $K=\lfloor \log_2 X\rfloor$. Then we have
\[
S(X)-\sum_{j=1}^{K}S_{G_{2^{j\delta},1/2}}(2^j)-S_{G_{X^{\delta},X^2/K}}(X)\ll X^{2\varrho+\delta+\varepsilon},
\]
where the implied constant depends only on $f_\infty$ and $f_{q_i}$.
\end{proposition}
\begin{proof}
The proof is nearly the same as \cite[Proposition 7.15]{cheng2025c}. Since the method is simple, we repeat it.
We have
\[
S(X)=\sum_{j=1}^{K}S(2^{j-1},2^j)+S(2^K,X),
\]
where
\[
S(2^{j-1},2^j)=\sum_{\substack{2^{j-1}\leq n<2^j\\n\in \ZZ_{(S)}^{>0}}}\Sigma^{n^2}(\xi)
\]
and
\[
S(2^K,X)=\sum_{\substack{2^K\leq n<X\\n\in \ZZ_{(S)}^{>0}}}\Sigma^{n^2}(\xi).
\]
By the definition of $S_G(X)$ and \eqref{eq:propertyconvolution} we obtain
\[
S_{G_{2^{j\delta},1/2}}(2^j)-S(2^{j-1},2^j)\ll \sum_{2^j-2^{j\delta}<n<2^j+2^{j\delta}}\Sigma^{n^2}(\xi)+ \sum_{2^{j-1}-2^{(j-1)\delta}<n<2^{j-1}+2^{(j-1)\delta}} \Sigma^{n^2}(\xi).
\]
By \autoref{prop:ramanujan},
\[
\Sigma^{n^2}(\xi)\ll_{f_\infty,f_{q_i},\varepsilon} n^{2\varrho+\varepsilon}.
\]
Hence
\[
S_{G_{2^{j\delta},1/2}}(2^j)-S(2^{j-1},2^j)\ll_{f_\infty,f_{q_i},\varepsilon} 2^{j(\delta+2\varrho+\varepsilon)}.
\]
Similarly we have
\[
S_{G_{X^{\delta},X^2/K}}(X)-S(2^K,X)\ll X^{\delta+2\varrho+\varepsilon}.
\]
Therefore
\begin{align*}
S(X)-\sum_{j=1}^{K}S_{G_{2^{j\delta},1/2}}(2^j)-S_{G_{X^{\delta},X^2/K}}(X)\ll& \sum_{j=1}^{K}2^{j(\delta+2\varrho+\varepsilon)}+X^{\delta+2\varrho+\varepsilon}\\
\ll& (2^K)^{\delta+2\varrho+\varepsilon}+X^{\delta+2\varrho+\varepsilon}\ll X^{\delta+2\varrho+\varepsilon}. \qedhere
\end{align*}
\end{proof}

\begin{proposition}
Let $K=\lfloor \log_2 X\rfloor$. Then we have
\[
\sum_{j=1}^{K}S_{G_{2^{j\delta},1/2}}(2^j) +S_{G_{X^{\delta},X^2/K}}(X)
=(\ms{A}+\ms{B}+\ms{C})X+ O(X^{2(1-\delta)+\frac12+\varepsilon}).
\]
\end{proposition}
\begin{proof}
By \autoref{thm:contributionxigfinal} we have
\begin{align*}
&\sum_{j=1}^{K}S_{G_{2^{j\delta},1/2}}(2^j) +S_{G_{X^{\delta},X^2/K}}(X)\\
=&\sum_{j=1}^{K}(\ms{A}+\ms{B}+\ms{C})\frac{2^j}{2}+(\ms{A}+\ms{B}+\ms{C})\left(1-\frac{2^K}{X}\right)
+\sum_{j=1}^{K}O((2^j)^{2(1-\delta)+\frac12+\varepsilon})+O(X^{2(1-\delta)+\frac12+\varepsilon}).
\end{align*}
For the main term we have
\[
\sum_{j=1}^{K}2^{j-1}+1-\frac{2^K}{X}=X+O(1).
\]
For the error term we have
\[
\sum_{j=1}^{K}O((2^j)^{2(1-\delta)+\frac12+\varepsilon})+O(X^{2(1-\delta)+\frac12+\varepsilon})\ll (2^K)^{2(1-\delta)+\frac12+\varepsilon}+X^{2(1-\delta)+\frac12+\varepsilon}\ll X^{2(1-\delta)+\frac12+\varepsilon}.
\]
Hence we obtain the desired result.
\end{proof}

Combining the above two propositions with $\delta=5/6-2/3\varrho$, we obtain
\begin{theorem}\label{thm:contributionfinal}
Suppose that $\varrho<1/8$ is a bound towards the Ramanujan conjecture. Then for any $\varepsilon>0$, we have
\[
S(X)=(\ms{A}+\ms{B}+\ms{C})X+O(X^{\frac43\varrho+\frac56+\varepsilon}).
\] 
where the implied constant depends only on$\varepsilon$, $f_\infty$ and $f_{q_i}$.
\end{theorem}

\section{Relation between the standard and the symmetric square representation}\label{sec:sym2std}
Now we have given an asymptotic formula for
\[
\sum_{\substack{n<X\\ \gcd(n,S)=1}}I_\cusp(f^{n^2}).
\]
However, what we want is the asymptotic behavior of
\[
\sum_{\substack{n<X\\ \gcd(n,S)=1}}I_\cusp(f^{n,2}).
\]
We have $a_{\pi,\mathrm{Sym}^2}(p)=a_\pi(p^2)$ for any prime $p\notin S$. However, it is not true if $p$ is replaced by an arbitrary $n$. Thus the asymptotic formula that we proved does not imply any information of $\res_{s=1}L^S(s,\pi,\mathrm{Sym}^2)$. However, it implies some information of
\[
\res_{s=1}\frac{L^S(s,\pi,\mathrm{Sym}^2)}{L^S(2s,\chi_\pi^2)},
\]
where $\chi_\pi$ is the Dirichlet character associated with the central character of $\pi$.

The main goal of this section is to give a formula of
\[
\sum_{n\in \ZZ_{(S)}^{>0}}\frac{a_{\pi}(n^2)}{n^s}=\prod_{p\notin S}\sum_{u=0}^{+\infty}\frac{a_\pi(p^{2u})}{p^{us}}
\]
involving the symmetric square representation. We first recall some facts of the automorphic representations on $\GL_2$. For each $p\notin S$, let $\{\alpha_p,\beta_p\}$ be the Satake parameter of $\pi$. Then we have
\begin{equation}\label{eq:standardrep}
L^S(s,\pi)=\prod_{p\notin S}\frac{1}{(1-\alpha_pp^{-s})(1-\beta_pp^{-s})}=\sum_{n\in \ZZ_{(S)}^{>0}}\frac{a_\pi(n)}{n^s}
\end{equation}
and
\[
L^S(s,\pi,\Sym^2)=\prod_{p\notin S}\frac{1}{(1-\alpha_p^2p^{-s})(1-\alpha_p\beta_pp^{-s})(1-\beta_p^2p^{-s})}
\]
for $\Re s$ sufficiently large.
Also we have $\alpha_p\beta_p=\chi_\pi(p)$. Hence $L^S(s,\pi,\Sym^2)$ can be rewritten as
\begin{equation}\label{eq:symrep}
\begin{split}
L^S(s,\pi,\Sym^2)&=\prod_{p\notin S}\frac{1}{(1-\alpha_p^2p^{-s})(1-\chi_\pi(p)p^{-s})(1-\beta_p^2p^{-s})}\\
&=\prod_{p\notin S}\frac{1}{(1-\alpha_p^2p^{-s})(1-\beta_p^2p^{-s})}L^S(s,\chi_\pi)
\end{split}
\end{equation}
for $\Re s$ sufficiently large.

Now we denote
\begin{equation}\label{eq:symrepcoefficient}
\prod_{p\notin S}\frac{1}{(1-\alpha_p^2p^{-s})(1-\beta_p^2p^{-s})}=\sum_{n\in \ZZ_{(S)}^{>0}}\frac{A_\pi(n)}{n^s}=\prod_{p\notin S}\sum_{u=0}^{+\infty}\frac{A_\pi(p^u)}{p^{us}}.
\end{equation}

\begin{proposition}\label{prop:sym2stdrelation}
For $n\in \ZZ_{>0}$ and $p\notin S$ we have
\[
a_\pi(p^{2n})=A_\pi(p^n)+\chi_\pi(p)A_\pi(p^{n-1}).
\]
\end{proposition}
\begin{proof}
In the proof we omit the subscript $\pi$.

By expanding the formulas \eqref{eq:standardrep} and \eqref{eq:symrepcoefficient}, we have
\[
a(p^{2n})=\sum_{j=0}^{2n}\alpha_p^j\beta_p^{2n-j}
\]
and 
\[
A(p^n)=\sum_{j=0}^{n}\alpha_p^{2j}\beta_p^{2n-2j}.
\]
Since $\chi(p)=\alpha_p\beta_p$, we obtain
\begin{align*}
A_\pi(p^n)+\chi_\pi(p)A_\pi(p^{n-1})&=\sum_{j=0}^{n}\alpha_p^{2j}\beta_p^{2n-2j}+ \alpha_p\beta_p\sum_{j=0}^{n-1}\alpha_p^{2j}\beta_p^{2n-2-2j}=\sum_{j=0}^{n}\alpha_p^{2j}\beta_p^{2n-2j}+ \sum_{j=0}^{n-1}\alpha_p^{1+2j}\beta_p^{2n-1-2j}\\
&=\sum_{j=0}^{2n}\alpha_p^j\beta_p^{2n-j}=a(p^{2n}).\qedhere
\end{align*}
\end{proof}

Now we prove the main theorem in this section.
\begin{theorem}
We have
\[
\sum_{n\in \ZZ_{(S)}^{>0}}\frac{a_\pi(n^2)}{n^s}=\frac{L^S(s,\pi,\Sym^2)}{L^S(2s,\chi_\pi^2)}.
\]
\end{theorem}
\begin{proof}
We have
\[
\sum_{n\in \ZZ_{(S)}^{>0}}\frac{a_\pi(n^2)}{n^s}=\prod_{p\notin S}\sum_{u=0}^{+\infty}\frac{a_\pi(p^{2u})}{p^{2us}}
\]
and
\[
\frac{L^S(s,\pi,\Sym^2)}{L^S(2s,\chi_\pi^2)}=\prod_{p\notin S}\frac{1-\chi_\pi(p)^2p^{-2s}}{(1-\alpha_p^2p^{-s})(1-\beta_p^2p^{-s})(1-\chi_\pi(p)p^{-s})}= \prod_{p\notin S}\frac{1+\chi_\pi(p)p^{-s}}{(1-\alpha_p^2p^{-s})(1-\beta_p^2p^{-s})}.
\]
Hence it suffices to show that
\begin{equation}\label{eq:sym2relation}
\sum_{u=0}^{+\infty}\frac{a_\pi(p^{2u})}{p^{2us}}=\frac{1+\chi_\pi(p)p^{-s}}{(1-\alpha_p^2p^{-s})(1-\beta_p^2p^{-s})}.
\end{equation}
By \autoref{prop:sym2stdrelation} we have
\[
\sum_{u=0}^{+\infty}\frac{a_\pi(p^{2u})}{p^{us}}= 1+\sum_{u=1}^{+\infty}\frac{A_\pi(p^u)+\chi_\pi(p)A_\pi(p^{u-1})}{p^{us}}=1+ \sum_{u=1}^{+\infty}\frac{A_\pi(p^u)}{p^{us}}+\frac{\chi_\pi(p)}{p^s}\sum_{u=0}^{+\infty} \frac{A_\pi(p^{u})}{p^{us}},
\]
which is precisely \eqref{eq:sym2relation} by the definition \eqref{eq:symrepcoefficient}.
\end{proof}

The theorem implies the following philosophy: Suppose that we have known the meromorphic continuation of $L^S(s,\pi,\Sym^2)$ and it has at most a simple pole at $s=1$. Then we have
\[
\lim_{X\to +\infty}\frac{1}{X}\sum_{\substack{n<X\\\gcd(n,S)=1}}a_\pi(n^2)= \res_{s=1}\frac{L^S(s,\pi,\Sym^2)}{L^S(2s,\chi_\pi^2)}.
\]
Thus by using the trace formula, the average analog is
\[
\lim_{X\to +\infty}\frac{1}{X}\sum_{\substack{n<X\\\gcd(n,S)=1}}I_\cusp(f^{n^2})=\sum_{\pi}m_\pi \prod_{v\in S}\Tr(\pi_v(f_v))\res_{s=1}\frac{L^S(s,\pi,\Sym^2)}{L^S(2s,\chi_\pi^2)}.
\]
Thus we can detect the automorphic representations such that 
\[
\frac{L^S(s,\pi,\Sym^2)}{L^S(2s,\chi_\pi^2)}
\]
and thus $L^S(s,\pi,\Sym^2)$ has a pole at $s=1$.
\appendix

\section{Estimate of the semilocal Fourier transform}\label{sec:globalnintegralestimateproof}
In this appendix we will prove \autoref{thm:globalnintegralestimate}. By \autoref{subsec:singularities} and linearity, it suffices to consider a more general case:
\begin{theorem}\label{thm:globalnintegralestimatefinal}
Let $H\in \cS$ such that the Mellin transform of all the derivatives of $H$ has rapid decay vertically. (In particular, holds for $F$ and $V$ by \autoref{subsec:deffv}.) Let $\iota\in \{0,1\}$, $\epsilon\in \{0,\pm 1\}^r$. Let $\chi$ be a character of $\ZZ_{S_\fin}^\times$. Let $\varphi(x)$ be a continuous function on $X_\iota$ supported on a bounded set, of the form $\varphi(x)=\overline{\varphi}(x)|x^2\mp 1|^{d/2}$ for some smooth function $\overline{\varphi}$ and $d\in \ZZ_{\geq -1}$. Let $\psi_i(y_i,b_i)$ be a continuous function on $Y_\epsilon\times \QQ_{q_i}$ such that the function is compactly supported (in $\QQ_{q_i}^\times$) and $\chi_i$-spherical with respect to $b_i$, and compactly supported (in $\QQ_{q_i}$) with respect to $y_i$ uniformly in $b_i$, and $\psi_i(y_i,b_i) =\overline{\psi}_i(y_i,b_i)|y_i^2-4c|_p'^{d_i/2}$ for some smooth function $\overline{\psi}_i$ and $d_i\in \ZZ_{\geq -1}$.
Then the integral
\begin{align*}
\Phi(s):=& \int_{(x,a)\in X_\iota\times\RR}\int_{(y,b)\in Y_\epsilon\times\QQ_{S_\fin}} \frac{1}{|a|_\infty^{s}|b|_q^{s}} \varphi(x)\prod_{i=1}^{r}\psi_{i}(y_i,b_i) \\
\times&H\legendresymbol{1/(|a|_\infty|b|_q)}{|x^2-4c|_\infty^{1/2}|y^2- 4c|_q'^{1/2}}\rme(-ax\xi-a\alpha)\rme_{q}(-by\xi-b\alpha)\rmd x\rmd a\rmd y\rmd b
\end{align*}
can be analytically continued to the half plane $\Re s>-1/2$. Moreover, for any $\sigma\in\lopen -1/2,+\infty\ropen$ and $M_0$, $M_1$, $N_0$, $N_1>0$ we have
\[
\Phi(s)\ll |\xi|_\infty^{-M_0}\prod_{i=1}^{r}(1+|\xi|_{q_i})^{-M_1}(1+|\alpha|)^{-N_0} \prod_{i=1}^{r}(1+|\alpha|_{q_i})^{-N_1},
\]
where the implied constant depends only on $\sigma=\Re s$, $M_0$, $M_1$, $N_0$ and $N_1$.
\end{theorem}

The proof of the above theorem is actually a mix of the proofs of \autoref{thm:archimedeanintegralestimate} and \autoref{thm:nonarchimedeanintegralestimate}. The difference is that we need the Mellin transform to separate the finite places, and thus we use \cite[Theorem A.5]{cheng2025b} instead of \autoref{lem:rapiddcay}.

\begin{proof}
We first deal with the archimedean case as in the proof of \autoref{thm:archimedeanintegralestimate}. 

First, we assume that $|\alpha|\gg 1$. Since $H$ is of rapid decay, the integrand is smooth at $a=0$. By integration by parts $N$ times with respect to the variable $a$, and the Leibniz rule, we have
\begin{align*}
\Phi(s)=&\frac{1}{(\dpii\alpha)^N}\int_{\RR}\int_{X_\iota}\int_{Y_\epsilon\times \QQ_{S_\fin}}\frac{1}{|b|_q^s}\varphi(x)\prod_{i=1}^{r}\psi_{i}(y_i,b_i)\\
\times&\frac{\diff^N}{\diff a^N}\left[\frac{1}{|a|^s}H\legendresymbol{1/|a||b|_q}{|x^2-4c|^{1/2}|y^2-4c|_q'^{1/2}}\rme^{-\dpii ax\xi}\right]\rme^{-\dpii a\alpha}\rme_{q}(-by\xi-b\alpha)\rmd x\rmd a\rmd y\rmd b\\
=&\frac{1}{(\dpii\alpha)^N}\int_{\RR}\int_{X_\iota}\int_{Y_\epsilon\times \QQ_{S_\fin}}\frac{1}{|b|_q^s}\varphi(x)\prod_{i=1}^{r}\psi_{i}(y_i,b_i)\sum_{j=1}^{N}\binom{N}{j}\\
\times& \frac{\diff^{N-j}}{\diff a^{N-j}}\left(\frac{1}{|a|^s}H\legendresymbol{1/|a||b|_q}{|x^2-4c|^{1/2}|y^2-4c|_q'^{1/2}}\right)(-\dpii x\xi)^j\rme^{-\dpii ax\xi}\rme^{-\dpii a\alpha}\rme_{q}(-by\xi-b\alpha)\rmd x\rmd a\rmd y\rmd b.
\end{align*}
Then we consider the integral over $x$. Since $H$ and all its derivatives is of rapid decay, the integrand vanishes on the boundary of $X_\iota$.
By integration by parts $L>N$ times we obtain
\begin{align*}
\Phi(s)=&\frac{1}{(\dpii\alpha)^N}\sum_{j=1}^{N}\binom{N}{j}\int_{\RR}\int_{X_\iota}\int_{Y_\epsilon\times \QQ_{S_\fin}} \prod_{i=1}^{r}\psi_{i}(y_i,b_i) \frac{1}{(\dpii a\xi)^L}\frac{\diff^{L}}{\diff x^{L}}\left[\varphi(x)\frac{\diff^{N-j}}{\diff a^{N-j}}\left(\frac{1}{|a|^s}\right.\right.\\
\times&\left.\left. H\legendresymbol{1/|a||b|_q}{|x^2-4c|^{1/2}|y^2-4c|_q'^{1/2}}\right)(-\dpii x\xi)^j\right]\rme^{-\dpii ax\xi}\rme^{-\dpii a\alpha}\rme_{q}(-by\xi-b\alpha)\rmd x\rmd a\rmd y\rmd b.
\end{align*}
Since $H$ and all its derivatives are of rapid decay, the integrand vanishes at $a=0$. Also, by the expression it is direct to show that the integrand is $\ll |a|^{-\Re s-L}$ as $|a|\to +\infty$. Hence the integral is absolutely convergent on the half plane $\Re s>1-L$ and thus $\Re s>-1/2$ for $L\geq 2$. Moreover, in this region we have
\[
\Phi(s)\ll_{\sigma,L,N} |\alpha|^{-N}\sum_{j=1}^{N}\binom{N}{j}|\xi|^{j-L}\times\text{nonarchimedean part}
\]
Let $M=L-N$, we then obtain $\Phi(s)\ll_{\sigma,L,N}|\xi|^{-M}|\alpha|^{-N}\times\text{nonarchimedean part}$.

If $|\alpha|\ll 1$, we also get $\Phi(s)\ll_{\sigma,L,N}|\xi|^{-M}(1+|\alpha|)^{-N}\times\text{nonarchimedean part}$ by a similar argument (without integration by parts with respect to $a$).

Now it suffices to prove the estimate of nonarchimedean part uniformly with respect to archimedean part since the support of $\theta(x,c)$ is compact and $N$ can be chosen to be sufficiently large such that we still have the analytic continuation to $\Re s>-1/2$. More precisely, it suffices to prove the following lemma:
\begin{lemma}
Let $H\in \cS$ such that the Mellin transform of $H$ has rapid decay vertically, and $\epsilon\in \{0,\pm 1\}^r$. Let $\chi$ be a character of $\ZZ_{S_\fin}^\times$. Let $\psi_i(y_i,b_i)$ be a continuous function on $Y_\epsilon\times \QQ_{q_i}$ such that the function is compactly supported (in $\QQ_{q_i}^\times$) and $\chi_i$-spherical with respect to $b_i$, and compactly supported (in $\QQ_{q_i}$) with respect to $y_i$ uniformly in $b_i$, and $\psi_i(y_i,b_i) =\overline{\psi}_i(y_i,b_i)|y_i^2-4c|_p'^{d_i/2}$ for some smooth function $\overline{\psi}_i$ and $d_i\in \ZZ_{\geq -1}$.
Then the integral
\[
\Psi(s):=\int_{(y,b)\in Y_\epsilon\times\QQ_{S_\fin}} \frac{1}{|a|_\infty^{s}|b|_q^{s}} \prod_{i=1}^{r}\psi_{i}(y_i,b_i) H\legendresymbol{1/(|a|_\infty|b|_q)}{|x^2-4c|_\infty^{1/2}|y^2- 4c|_q'^{1/2}}\rme_{q}(-by\xi-b\alpha)\rmd y\rmd b
\]
can be analytically continued to the half plane $\Re s>-1/2$. Moreover, for any $\sigma\in\lopen -1/2,+\infty\ropen$ and $M_1$, $N_1>0$, there exists $L>0$ such that we have
\[
\Psi(s)\ll \prod_{i=1}^{r}(1+|\xi|_{q_i})^{-M_1}\prod_{i=1}^{r}(1+|\alpha|_{q_i})^{-N_1}|a|^L,
\]
where the implied constant depends only on $\sigma=\Re s$, $L$, $M_1$ and $N_1$, and uniformly in each compact subset of $x\in X_\iota$.
\end{lemma}
\begin{insertproof}
We will use the notation $\sigma=\Re s$ and $\tau=\Re u$ throughout the proof.
Since $\psi_i(y_i,b_i)$ is compactly supported with respect to $b_i\in \QQ_{q_i}^\times$, $\Psi(s)$ is automatically an entire function. By Mellin inversion formula with respect to $H$, we have
\begin{align*}
\Psi(s)&=\int_{Y_\epsilon\times\QQ_{S_\fin}} \frac{1}{\dpii}\int_{(\tau)}\widetilde{H}(u)|x^2-4c|_\infty^{u/2}|a|^u\frac{1}{|b|_q^{s+u}} \prod_{i=1}^{r}\psi_{i}(y_i,b_i)|y^2-4c|_q'^{u/2}\rme_{q}(-by\xi-b\alpha)\rmd y\rmd b\\
&=\frac{1}{\dpii}\int_{(\tau)}\widetilde{H}(u)|x^2-4c|_\infty^{u/2}|a|^u\prod_{i=1}^{r} \int_{Y_{\epsilon_i}\times \QQ_{q_i}}\frac{1}{|b_i|_{q_i}^{s+u}} \psi_{i}(y_i,b_i)|y_i^2-4c|_{q_i}'^{u/2}\rme_{q_i}(-b_iy_i\xi-b_i\alpha)\rmd y_i\rmd b_i.
\end{align*}
We will choose $\tau$ sufficiently large to obtain the desired theorem.

The key point is to estimate the local integral
\[
\Omega_i(s):=\int_{Y_{\epsilon_i}\times \QQ_{q_i}}\frac{1}{|b_i|_{q_i}^{s+u}} \psi_{i}(y_i,b_i)|y_i^2-4c|_{q_i}'^{u/2}\rme_{q_i}(-b_iy_i\xi-b_i\alpha)\rmd y_i\rmd b_i.
\]
As in the proof of \autoref{thm:nonarchimedeanintegralestimate}, we split the integral over $i$ into three parts. We define $Y_{\epsilon_i,<}$, $Y_{\epsilon_i,=}$ and $Y_{\epsilon_i,>}$, and $\Omega_{i,<}(s)$, $\Omega_{i,=}(s)$ and $\Omega_{i,>}(s)$ as in the proof of \autoref{thm:nonarchimedeanintegralestimate}. Then we consider the following three cases:

\underline{\emph{Case 1:}}\ \ $|y_i\xi|_{q_i}<|\alpha|_{q_i}$.

In this case $|y_i\xi+\alpha|_{q_i}=|\alpha|_{q_i}$. Hence by \autoref{cor:absolutespherical} we obtain
\[
|\Omega_{i,<}(s)|=\left|\int_{\QQ_{q_i}}\int_{Y_{\epsilon,<}} \frac{1}{|b_i|_{q_i}^{s+u}}\psi_{i}(y_i,b_i)|y_i^2-4c|_{q_i}'^{u/2}\rme_{q_i}(-b_i\alpha)\rmd y_i\rmd b_i\right|.
\]
Hence it suffices to estimate the integral in the absolute value.

Since the integrand for $b_i$ is zero near $0$ and infinity, the function
\[
\frac{1}{|b_i|_{q_i}^{s+u}}\psi_i(y_i,b_i)
\]
is smooth and compactly supported with respect to $b_i$. Hence the Fourier transform is $0$ for $|\alpha|_{q_i}$ sufficiently large.

Now we consider the case $|\alpha|_{q_i}\ll 1$. By definition of $\Omega_{i,<}(s)$ we have
\[
\Omega_{i,<}(s)=\int_{\QQ_{q_i}}\int_{Y_{\epsilon_i,<}}\frac{1}{|b_i|_{q_i}^{s}} \psi_i(y_i,b_i)|y_i^2-4c|_{q_i}'^{u/2}\rme_{q_i}(-b_iy_i\xi)\rme_{q_i}(-b_i\alpha)\rmd y_i\rmd b_i. 
\]

In this time, we consider the integral over $y_i$, that is
\begin{equation}\label{eq:integralyappendix}
\int_{Y_{\epsilon_i,<}}\psi_i(y_i,b_i)|y_i^2-4c|_{q_i}'^{u/2}\rme_{q_i}(-b_iy_i\xi)\rmd y.
\end{equation}
Since the function $\psi_i(y_i,b_i)|y_i^2-4c|_p'^{u/2}$
is smooth at $y_i=0$, there exists a constant $C$ (independent of $\alpha$) depending only on $\psi$ such that if $|\xi|_{q_i}>C$, then the above function is a smooth function. Hence we find that $\eqref{eq:integralyappendix}=0$ if $|b_i\xi|_{q_i}$ is sufficiently large (independent of $\alpha$). Since $|b_i|_{q_i}\asymp 1$ in the support of $\Psi$, we may replace $|b_i\xi|_{q_i}$ by $|\xi|_{q_i}$. Finally, for any $\alpha$ and $\xi$ we have the trivial estimate
\[
\Omega_{i,<}(s)\ll\int_{\QQ_{q_i}}\int_{Y_{\epsilon,<}}\left|\frac{1}{|b_i|_{q_i}^{s+u}}\psi_i(y_i,b_i) |y_i^2-4c|_{q_i}'^{u/2}\right|\rmd y_i\rmd b_i\ll_{\sigma,\tau,\psi_i} 1.
\]
since $Y_{\epsilon_i,<}$ is bounded.

Therefore we have
\[
\Omega_{i,<}(s)\ll (1+|\xi|_{q_i})^{-M}(1+|\alpha|_{q_i})^{-N},
\]
where the implied constant depends only on $\psi_i$, $\sigma$, $\tau$, $M$ and $N$. 

\underline{\emph{Case 2:}}\ \ $|y\xi|_{q_i}>|\alpha|_{q_i}$.

In this case $|y_i\xi+\alpha|_{q_i}=|y_i\xi|_{q_i}$. Hence by \autoref{cor:absolutespherical} we obtain
\[
|\Omega_{i,>}(s)|=\left|\int_{\QQ_{q_i}}\int_{Y_{\epsilon_i,>}}\frac{1}{|b_i|_{q_i}^{s}} \psi_i(y_i,b_i)|y_i^2-4c|_{q_i}'^{u/2}\rme_{q_i}(-b_iy_i\xi)\rmd y_i\rmd b_i\right|.
\]
As in Case 1, the integration over $b_i$ vanishes if $|y_i\xi|_{q_i}$ is sufficiently large (independent of $\alpha$). Since $|y_i\xi|_{q_i}>|\alpha|_{q_i}$, $\Omega_{i,>}(s)=0$ for $|\alpha|_{q_i}$ sufficiently large. 

Now we assume that $|\alpha|_{q_i}\ll 1$. In this case, by \cite[Theorem A.5]{cheng2025b}, we have
\[
\int_{Y_{\epsilon_i,>}}\psi(y_i,b_i)|y^2-4c|_{q_i}'^{u/2}\rme_{q_i}(-b_iy_i\xi)\rmd y\ll (1+|b_i\xi|_{q_i})^{-d/2-\tau/2}
\]
Since $|b|_{q_i}\asymp 1$, we know that the above integral is $(1+|\xi|_{q_i})^{-d_i/2-\tau/2}$.

Therefore we have
\[
\Phi_>(s)\ll (1+|\xi|_{q_i})^{-d_i/2-\tau/2}(1+|\alpha|_{q_i})^{-N},
\]
where the implied constant depends only on $\psi_i$, $\sigma$, $\tau$ and $N$. 

\underline{\emph{Case 3:}}\ \ $|y\xi|_{q_i}=|\alpha|_{q_i}$.

First we assume that $|\alpha|_{q_i}\gg 1$. Then we have
\[
\Omega_{i,=}(s)=\frac{1}{|\alpha|_{q_i}^N}\int_{\QQ_{q_i}}\int_{Y_{\epsilon_i,=}} \frac{1}{|b_i|_{q_i}^s}\psi_i(y_i,b_i)|y_i^2-4c|_p'^{u/2} |y_i\xi|_{q_i}^N\rme_{q_i}(-b_iy_i\xi)\rme_{q_i}(-b_i\alpha)\rmd y_i\rmd b_i
\]
for any $N>0$.

Since $\psi_i$ is compactly supported with respect to $y_i$, we have $|y_i|_{q_i}^N\ll 1$ and hence by \cite[Theorem A.5]{cheng2025b}, for any $M>0$ we have
\[
\int_{Y_{\epsilon_i,=}}\psi_i(y_i,b_i)|y_i^2-4c|_{q_i}'^{u/2}|y_i\xi|_p^N\rme_{q_i}(-b_iy_i\xi)\rmd y_i\ll |\xi|_{q_i}^N(1+|b_i\xi|_{q_i})^{-M-N}\asymp (1+|\xi|_{q_i})^{-d_i/2-\tau/2}
\]
since $|b_i|_{q_i}\asymp 1$. Therefore we obtain
\[
\Omega_{i,=}(s)\ll |\alpha|_{q_i}^{-N}(1+|\xi|_{q_i})^{-d_i/2-\tau/2}.
\]

If $|\alpha|_{q_i}\ll 1$, we bound the integral over $y_i$ directly and obtain
\[
\int_{Y_{\epsilon_i,=}}\psi_i(y_i,b_i)|y_i^2-4c|_{q_i}'^{u/2}\rme_{q_i}(-b_iy_i\xi)\rmd y_i\ll (1+|b_i\xi|_{q_i})^{-M}\asymp (1+|\xi|_{q_i})^{-d_i/2-\tau/2}.
\]
Combining them we obtain 
\[
\Omega_{i,=}(s)\ll (1+|\alpha|_{q_i})^{-N}(1+|\xi|_{q_i})^{-d_i/2-\tau/2}.
\]
where the implied constant depends only on $\psi_i$, $\sigma$, $\tau$ and $N$.

Since $d_i\in \ZZ_{\geq -1}$, we obtain
\[
\Omega_i(s)=\Omega_{i,<}(s)+\Omega_{i,=}(s)+\Omega_{i,>}(s)\ll (1+|\alpha|_{q_i})^{-N}(1+|\xi|_{q_i})^{-M}
\]
for $\tau$ sufficiently large.

Finally, we move the $u$-contour to $(\tau)$ sufficiently large and obtain
\[
\Psi(s)
\ll\int_{(\tau)}|\widetilde{H}(u)||x^2-4c|_\infty^{u/2}|a|^\tau\prod_{i=1}^{r} |\Omega_i(s)|\ll |a|^L\prod_{i=1}^{r}(1+|\xi|_{q_i})^{-M}(1+|\alpha|_{q_i})^{-N}
\]
since $\widetilde{H}(u)$ is of rapid decay vertically.
\end{insertproof}
This lemma then yields \autoref{thm:globalnintegralestimatefinal}.
\end{proof}

\section*{Acknowledgements}
First, I would like to thank all the professors, postdocs and the graduate students teaching me or communicating with me about mathematics, in particular: \emph{C\'edric Bonnaf\'e}, \emph{Jin Cao}, \emph{Fei Chen}, \emph{Xiaojiang Cheng}, \emph{Yitwah Cheung}, \emph{Taiwang Deng}, \emph{Hansheng Diao}, \emph{William Donovan}, \emph{Heng Du}, \emph{Lecouturier Emmanuel}, \emph{Lei Fu}, \emph{Honghao Gao}, \emph{Yueke Hu}, \emph{Chang Huang}, \emph{Zhen Huang}, \emph{Wenjia Jing}, \emph{Krishnarjun Krishnamoorthy}, \emph{Hao Li}, \emph{Huajie Li}, \emph{Pengcheng Li}, \emph{Si Li}, \emph{Sixu Liu}, \emph{Yu Liu}, \emph{Zhengwei Liu}, \emph{Zhilin Luo}, \emph{Cezar Lupu}, \emph{Qiang Ma}, \emph{Yunjian Peng}, \emph{Yu Qiu}, \emph{Yiannis Sakellaridis}, \emph{Peter Sarnak}, \emph{Peng Shan}, \emph{Koji Shimizu}, \emph{Chen Wan}, \emph{Xin Wan}, \emph{Jun Wang}, \emph{Zida Wang}, \emph{Yunhui Wu}, \emph{Bin Xu}, \emph{Jinxin Xue}, \emph{Shuhang Yang}, \emph{Xiaokui Yang}, \emph{Liyuan Ye}, \emph{Chenglong Yu}, \emph{Cheng Zhang}, \emph{Jie Zhou}, \emph{Yi Zhu}, \emph{Yihang Zhu}, and \emph{Yongchang Zhu}.

Next, I would like to thank all the academic affairs of Qiuzhen college, especially \emph{Ziyu Cheng}, \emph{Mingli Li}, \emph{Xin Liu}, \emph{Zhaoqing Jiang}, \emph{Hui Wang}, \emph{Xiaofang Wang} and \emph{Xinmeng Zhao}, although a few of them have left employment.

Finally, I would really like to thank my family that gives me lots of support and the opportunity of exploring mathematics, what I really want to do as a career in my life. Perhaps it is my unique past experiences that have shaped who I am today. I also want to express my heartfelt gratitude to those I truly love or have once loved, for they have given me the strength to overcome difficulties.

\nocite{cheng2026}
\bibliography{ref.bib}
\bibliographystyle{amsalpha}

\end{document}